\documentclass[a4paper,11pt]{amsart}
\usepackage{microtype,tikz-cd,amsmath,amsfonts,amssymb,amsthm,indentfirst,mathtools}
\usepackage[text={156mm,246mm},centering]{geometry}

\makeatletter
\renewenvironment{proof}[1][\proofname] {\par\pushQED{\qed}\normalfont\topsep6\p@\@plus6\p@\relax\trivlist\item[\hskip\labelsep\bfseries#1\@addpunct{.}]\ignorespaces}{\popQED\endtrivlist\@endpefalse}
\makeatother

\usepackage{enumitem}
\setlist[enumerate]{left=\parindent,label=\textup{(\roman*)}, ref=\textup{\roman*}}
\setlist[itemize]{left=\parindent}
\usepackage{hyperref}
\numberwithin{equation}{subsection}

\newtheorem{theorem}[equation]{Theorem}
\newtheorem*{theorem*}{Theorem}
\newtheorem{proposition}[equation]{Proposition}
\newtheorem{corollary}[equation]{Corollary}
\newtheorem{lemma}[equation]{Lemma}
\newtheorem{claim}[equation]{Claim}
\newtheorem{variant}[equation]{Variant}
\newtheorem{question}[equation]{Question}
\theoremstyle{definition}
\newtheorem{definition}[equation]{Definition}

\newtheorem{notation}[equation]{Notation}
\newtheorem{construction}[equation]{Construction}
\newtheorem{example}[equation]{Example}

\newtheorem*{remark*}{Remark}
\newtheorem*{remarks*}{Remarks}

\DeclareMathOperator{\pH}{{}^{\mathrm{p}}\!\mathcal{H}}
\newcommand{\pR}{{}^{\mathrm{p}}R}
\newcommand{\Hom}{\mathcal{H}\!\mathit{om}}

\AtBeginDocument{
  \let\leq=\leqslant
  \let\geq=\geqslant
  \renewcommand{\setminus}{\mathbin{\rule[0.2em]{0.67em}{0.12em}}}%
}

\title{Betti number bounds for varieties and exponential sums}

\author{Daqing Wan}
\address{Department of Mathematics, University of California, Irvine, CA 92697-3875 USA.}
\email{dwan@math.uci.edu}

\author{Dingxin Zhang}
\address{Yau Mathematical Sciences Center, Tsinghua University, Beijing~100086 China.}
\email{dingxin@tsinghua.edu.cn}

\date{}

\begin{document}

\begin{abstract}
Using basic properties of perverse sheaves, we give new upper bounds
for compactly supported Betti numbers for arbitrary affine varieties
in $\mathbb{A}^n$ defined by $r$ polynomial equations of degrees at
most $d$. As arithmetic applications, new total degree bounds are
obtained for zeta functions of varieties and L-functions of
exponential sums over finite fields, improving the classical results
of Bombieri, Katz, and Adolphson--Sperber.  In the complete
intersection case, our total Betti number bound is asymptotically
optimal as a function in $d$.  In general, it remains an open problem
to find an asymptotically optimal bound as a function in $d$.
\end{abstract}

\maketitle


\section{Introduction}

In this paper, we address the problem of bounding compactly supported
Betti numbers for certain families of algebraic varieties and
exponential sums.  While this is a fundamental question in topology,
algebraic geometry, arithmetic geometry and their algorithmic aspects, our motivation is
strengthened by the fact that such bounds can lead to improved
estimates in analytic number theory over global function fields.  Our
focus is therefore on obtaining quantitative and explicit estimates
that are as sharp as possible.

For ``nondegenerate'' objects, Betti numbers can often be computed
exactly.  However, degeneration significantly complicates precise
general control.  We develop a framework (see
\S\ref{sec:specialization}) for obtaining estimates and demonstrate
that, under suitable hypotheses, the Betti numbers undergo minimal
deterioration during degeneration.

Regarding Betti numbers of affine varieties, we improve upon previous
results by Katz \cite{katz:sums-of-betti-numbers}.  In the most
general case, we achieve the correct asymptotic order
(\ref{theorem:order}).  For the complete intersection case, we obtain
asymptotically optimal results (\ref{theorem:main}).

For exponential sums on affine varieties, we extend and refine the
results of
\cite{adolphson-sperber:newton-polyhedra-degree-l-function,bombieri:exponential-sums-in-finite-fields-2}
by lifting their L-function bounds to the cohomological level
(\ref{theorem:as-improvement}, \ref{eq:more-like-as},
\ref{eq:power-as}).  Our bounds are asymptotically optimal in the
complete intersection case.  Moreover, our method reveals an upper
semicontinuity property for Newton polygons associated with
exponential sums (\ref{theorem:semicont-newton}).

While our approach is applicable in many other contexts, we focus on
theorems about Betti numbers of varieties and exponential sums, as
these are the objects most frequently encountered in algorithmic
arithmetic geometry and analytic number theory over function fields.
Another approach for obtaining Betti bounds based on ramification
theory has been recently developed by Haoyu Hu and J.-B.~Teyssier
\cite{hu-teyssier}. Their bounds are expressed in terms of the
logarithmic conductors of the sheaf at infinity.  We believe readers
will find our concrete explicit results and the general theorems of Hu
and Teyssier to be mutually enriching.

A detailed summary of our results follows.

\subsection{Compactly supported Betti numbers of affine varieties}

Let \(k\) be an algebraically closed field, and let \(\ell\) be a
prime that is different from the characteristic of \(k\).  For an
algebraic variety \(V\) over \(k\), we define
\[
B_{c}(V, \ell) \coloneqq \sum \dim \mathrm{H}^{i}_{c}(V; \overline{\mathbb{Q}}_{\ell}),
\]
where \(\mathrm{H}^{i}_{c}(V; \overline{\mathbb{Q}}_{\ell})\) denotes
the \(i\)\textsuperscript{th} compactly supported \(\ell\)-adic
cohomology group.  Note that \(B_{c}(V, \ell)\) depends solely on the
reduced structure of \(V\). The number \(B_{c}(V, \ell)\) might depend
on the choice of the prime $\ell$. The conjectural independence of
\(B_{c}(V, \ell)\) on $\ell$ remains wide open in general for singular
or non-proper varieties. In many applications, an explicit good upper
bound for \(B_{c}(V, \ell)\) is already very useful. For instance, a
good upper bound for \(B_{c}(V, \ell)\) gives a good upper bound for
the total degree of the zeta function of an algebraic variety over a
finite field.  By Grothendieck's $\ell$-adic trace formula and
Deligne's theorem on the Weil conjectures, the quality of the total
degree bound is well known to be one of the main factors for point
counting estimates. A good total degree bound is also crucial in
estimating the complexity of effective algorithms in computing zeta
functions in arithmetic geometry, see \cite{AW08,Wan08,Har15}.

Let $n, r, d$ be positive integers.  The primary purpose of this
paper, also our original motivation, is to establish an asymptotically
optimal upper bound for \(B_{c}(V, \ell)\) as \(V\) ranges over all
Zariski closed subsets of the affine space \(\mathbb{A}^{n}_{k}\)
defined by \(r\) polynomial equations, each of degree at most \(d\),
over all algebraically closed field $k$ of characteristic different
from $\ell$.
To report our progress,  it is convenient to introduce the quantity
\begin{equation*}
B_{c}(n,r;d) \coloneqq \sup_{\ell\not= {\rm char}(k)} \left\{B_{c}(V, \ell) :
\begin{array}{l}
  V = \operatorname{Spec}k[x_{1},\ldots,x_{n}]/(f_{1},\ldots,f_{r})  \\
  \deg f_{i} \leq d, \ 1\leq i\leq r.
\end{array}
\right\},
\end{equation*}
where $k$ runs over all algebraically closed fields and $\ell$ runs
over all prime numbers different from the characteristic of \(k\).
The number $B_{c}(n,r;d)$ clearly depends only on the three parameters
$n ,r, d$.  By adding redundant equations, if necessary, one sees
that the number $B_{c}(n,r;d)$ is increasing in $r$. Namely,
\begin{equation}\label{eq:chain-of-b}
B_c(n, 1;d) \leq B_c(n, 2;d) \leq \cdots \leq B_c(n, r;d)  \leq B_c(n, r+1;d)\leq \cdots.
\end{equation}
Furthermore, by Kronecker's reduction \cite{perron:kronecker-theorem},
which holds for any infinite field, this sequence in $r$ stabilizes
once $r$ reaches $n+1$, that is,
\begin{equation*}
B_c(n, n+1;d) = B_c(n, n+2;d) = B_c(n, n+3;d) =\cdots.
\end{equation*}
Thus, without loss of generality, we can always assume that $r\leq n+1$.

Our main purpose of this paper is to study the growth of the number
$B_{c}(n,r;d)$ as a function of $d$.  Our first theorem states:

\begin{theorem}[See \S\ref{sec:proof-theorem-order}]\label{theorem:order}
With the above notation, we have
\[
(d-1)^{n} \leq B_{c}(n,r;d) \leq 3^{r} \times \binom{n+r-1}{r-1}\times (2d + 1)^{n}.
\]
\end{theorem}

Previously, the best known bound is due to Katz
\cite{katz:sums-of-betti-numbers}, who showed that
\begin{equation}
\label{eq:k-1}\tag{K}
B_{c}(n,r;d) \leq 2^{r} \times 3 \times 2 \times (rd+3)^{n+1}.
\end{equation}
If we treat \(n\) and \(r\) as fixed while considering \(d\) as a
variable, \eqref{eq:k-1} yields the asymptotic behavior
\(B_{c}(n,r;d) \ll_{n,r} d^{n+1}\).  In contrast, Theorem
\ref{theorem:order} implies \(B_{c}(n,r;d)\ll_{n,r}d^{n}\), thus
improving Katz's estimate \eqref{eq:k-1} by one order.  The lower
bound \((d-1)^{n} \leq B_{c}(n,r;d)\) indicates that there is no further
scope for refinement as far as enhancing the order is concerned. As a consequence, we obtain

\begin{corollary}\label{corollary:total}
For any positive integers $n, r, d$, we have
\begin{equation*}
 B_{c}(n,r;d) \asymp_{n,r} d^{n}.
\end{equation*}
\end{corollary}

\begin{remarks*}
\begin{enumerate}[wide, label={({\alph*)}}]

\item Theorem~\ref{theorem:order} also immediately implies a
``projective Betti bound'' for projective varieties in
\(\mathbb{P}^{n}\).  See Corollary~\ref{corollary:proj-order}.

\item The Betti bounds lead to improved total degree bounds for zeta
functions, which can be used to speed up algorithms for computing zeta
functions of varieties over finite fields and arithmetic schemes, see
\cite{AW08,Wan08,Har15}. Theorem~\ref{theorem:order} also leads to a
new total degree bound for the L-function of a linear representation
attached to the Galois group of a finite Galois covering of varieties
over finite fields. This new bound improves previous total degree
bounds in \cite{BS88,FW03}. One simply replaces the Bombieri bound
\cite{bombieri:exponential-sums-in-finite-fields-2} and the Katz bound
\cite{katz:sums-of-betti-numbers} by the stronger bound in
Theorem~\ref{theorem:order} throughout the proofs of \cite{BS88,FW03}.

\item
In some applications, the number of defining polynomials for a variety
can be quite large, or even exponentially large, see for instance
\cite{CMW22} for an application in theoretical computer science, where
one needs to quickly decide if two varieties over finite fields have
the same zeta function.  In such cases, the Kronecker reduction may
not apply for small finite fields, but the following corollary is
still useful to bound the total degree of the zeta function.

\begin{corollary}\label{corollary:total-bound-without-r}
Suppose \(V \subset \mathbb{A}^{n}_{k}\) is the zero locus of some
polynomials of degree at most \(d\).  Then
\begin{align*}
B_{c}(V, \ell)
&\leq 3^{n+1}\times \binom{2n}{n} \times (2d+1)^{n} < 3 (24d+12)^{n}.
\end{align*}
\end{corollary}

\begin{proof}
By Kronecker \cite{perron:kronecker-theorem},
\(V\) can be set-theoretically cut out by at most \(n+1\) polynomials
of degree at most \(d\).  We then apply Theorem~\ref{theorem:order}
with \(r = n+1\).
\end{proof}

\item Katz's original method only yields a non-optimal bound
\(B_{c}(n,r;d) \ll_{n,r} d^{n+1}\) due to a technical limitation: to
apply the weak Lefschetz theorem, he needed to reduce to a nonsingular
affine variety.  Thus he has to pass from, e.g., a hypersurface in the
affine \(n\)-space to its complement, which is smooth, then applies
the Rabinowitsch trick, which converts this hypersurface complement in
dimension \(n\) to a nonsingular hypersurface in dimension \(n+1\),
thereby increasing the order by one.

In addition to the proof in the main text, we present a shorter, more
elementary alternative proof of Corollary~\ref{corollary:total} in
\S\ref{sec:elementary}.  This proof improves upon Katz's method in two
ways: we apply Deligne's weak Lefschetz theorem for perverse sheaves
(cf.~Lemma~\ref{lemma:gysin-perv}) to avoid the Rabinowitsch trick,
and we utilize an algebraic lemma from our work on Frobenius colevels
\cite{wz1,wz2} for additional refinements.  While these improvements
yield a better bound than Katz's original approach, and gives the
correct order, the implied constant generally remains weaker than that
of Theorem~\ref{theorem:order} (see Theorem~\ref{theorem:elementary}).
However, this elementary approach may provide sharper estimates when
specific information about dimension and minimal number of defining
equations is available, which justifies its inclusion.

\item The sum
\(B(V,\ell) \coloneqq \sum \dim \mathrm{H}^{i}(V;\overline{\mathbb{Q}}_{\ell})\)
of the Betti numbers for \(\ell\)-adic cohomology with closed supports
is more difficult to estimate than the sum of compactly supported
ones, since \(B(V,\ell)\) has more complicated formal properties than
\(B_{c}(V,\ell)\).  Tweaking the method of Katz by computing a
spectral sequence backwards, the second author \cite{zhang-betti}
showed that if \(V \subset \mathbb{A}^{n}_{k}\) is the zero locus of
polynomials \(f_{1},\ldots,f_{r}\) with \(\deg f_{i} \leq d\), then
\(B(V,\ell) \leq 2r^{r}(rd+3)^{n}\).
\end{enumerate}
\end{remarks*}

\subsection{Compactly supported Betti numbers of affine complete intersections}

Having established that the correct order of \(B_{c}(n,r;d)\) is
\(d^{n}\), a natural question arises: is there an asymptotic formula
for \(B_{c}(n,r;d)\) as \(d \to \infty\)?
Namely, is there a positive constant $b_{n,r}$ such that
\begin{equation*}
  B_{c}(n,r;d)  = b_{n,r}d^n +o_{n,r}(d^n)?
\end{equation*}
If so, what is the leading constant $b_{n,r}$?
While the constant term
\(3^{r}\times \binom{n+r-1}{r-1}\times 2^n\) from Theorem~\ref{theorem:order} is
certainly not optimal, the correct constant remains unclear.  If \(V\)
is a sufficiently general complete intersection cut out by \(r\)
polynomials of degree \(d\), it can be shown that
\(B_{c}(V, \ell) = \binom{n-1}{r-1}d^{n} + O_{n,r}(d^{n-1})\) (see
Lemma~\ref{lemma:exact-betti-for-generic-affine-complete-intersection}).
In view of this, the best we could hope would be a positive answer to  the following question.

\begin{question}
Let \(r \leq  (n+1)/2\). Is it true that
\begin{equation*}
B_{c}(n,r;d) = \binom{n-1}{r-1}d^{n} +o_{n,r}(d^n) ?
\end{equation*}

\end{question}

Note that the condition \(r \leq (n+1)/2\) cannot be dropped for the
above question to have a positive answer. The reason is that the
number \(B_{c}(n,r;d)\) is an increasing function in $r$.  On the
other hand, the binomial coefficient $\binom{n-1}{r-1}$ is increasing
in $r$ only for $ 1\leq r \leq (n+1)/2$ and becomes decreasing in $r$
for $r \geq (n+1)/2$. This shows that for the expected asymptotic
formula to hold, it is necessary to have \(r \leq
(n+1)/2\). Alternatively, without the condition \(r \leq (n+1)/2\),
one can ask if
\begin{equation*}
B_{c}(n,r;d) = \binom{n-1}{\min\{r-1, \lfloor\frac{n-1}2\rfloor\}}d^{n} +o_{n,r}(d^n)?
\end{equation*}

Instead of tackling the above question in general, we will focus on a class of
subvarieties of \(\mathbb{A}^{n}\): (set-theoretic) complete
intersections.  For \(r \leq n\), we define
\begin{equation*}
B^{\text{ci}}_{c}(n,r;d) \coloneqq \sup_{\ell\not = {\rm char}(k)}\left\{B_{c}(V, \ell) :
\begin{array}{l}
  V = \operatorname{Spec} k[x_{1},\ldots,x_{n}]/(f_{1},\ldots,f_{r}),  \\
  \deg f_{i} \leq d, \text{ and } \dim V = n-r
\end{array}
\right\},
\end{equation*}
where, again, $k$ runs over all algebraically closed fields and $\ell$
runs over all prime numbers different from the characteristic of
\(k\).  Then \(B_{c}(n,1;d) = B_{c}^{\mathrm{ci}}(n,1;d)\) and, for
\(r\geq 1\), \(B_{c}(n,r;d) \geq B_{c}^{\mathrm{ci}}(n,r;d)\).  Our
second theorem provides an asymptotically optimal bound for
\(B_{c}^{\mathrm{ci}}(n,r;d)\) as \(d \to \infty\).

\begin{theorem}[= \ref{corollary:middle-cleaner-bound}]
\label{theorem:main}
Assuming \(r \leq n\), we have
\begin{equation*}
\binom{n-1}{r-1}(d-1)^{n} \leq B_{c}^{\mathrm{ci}}(n,r;d) \leq \binom{n-1}{r-1}(d+1)^{n}.
\end{equation*}
\end{theorem}

The lower bound follows from an exact computation of \(B_{c}(V, \ell)\)
(Lemma~\ref{lemma:exact-betti-for-generic-affine-complete-intersection})
for a sufficient general \(V\). As a consequence, we deduce
\begin{corollary}
Assuming \(r \leq n\), we have
\begin{equation*}
B_{c}^{\mathrm{ci}}(n,r;d) = \binom{n-1}{r-1}d^{n} +O_{n,r}(d^{n-1}).
\end{equation*}
\end{corollary}

\begin{remarks*}
\begin{enumerate}[wide, label={(\alph*)}]
\item As detailed in Theorem~\ref{theorem:affine-betti-bound-detail},
our method provides estimates for the compactly supported Betti
numbers in each individual cohomology degree.
Theorem~\ref{theorem:main} follows as a straightforward consequence of
these bounds.  Moreover, these affine bounds naturally extend to
projective set-theoretic complete intersections, as shown in
Proposition~\ref{proposition:projective-bound}.

\item In Sawin's recent work \cite{sawin20,Saw21} on function field
analytic number theory in short interval, one key step is to estimate
the total compact Betti number of affine complete intersections.
Theorem~\ref{theorem:main}, combined with Sawin's results, is thus
useful in improving various function field short interval bounds in
\cite{sawin20,Saw21}.  The improvement comes from replacing Katz's
estimate for \(B_{c}(V, \ell)\) with our sharper bound in
Theorem~\ref{theorem:main}.

For instance, in \cite{Saw21}, one needs
to estimate \(B_{c}(V,\ell)\) for the variety
\(V \subset \mathbb{A}^{n}_{\mathbb{F}_{q}}\) defined by the first $r$
elementary symmetric polynomials in $n$ variables.  This variety is a
complete intersection in $\mathbb{A}^n_{\mathbb{F}_{q}}$ with
$r=d\leq n$. For this case, while Katz's general bound gives either
$2^d\times 6\times (d^2+3)^{n+1}$ or $3(d+2)^{n+r}$ depending on which
bound of Katz one uses, our complete intersection bound in
Theorem~\ref{theorem:main} yields the better estimate
$\binom{n-1}{d-1}(d+1)^n \leq (2d+2)^n$.

Similarly, in \cite{sawin20}, one needs to estimate the total compact
Betti number of another more general affine variety
\[
Z_{n,\sum_{i=1}^{r}d_{i},\sum_{i=r+1}^{r+\widetilde{r}}} / (S_{d_{1}}\times \cdots \times S_{d_{r+\widetilde{r}}}) \subset \mathbb{A}^{\sum_{i=1}^{r+\widetilde{r}}d_{i}}
\]
defined by \(n\) polynomial equations of degree at most
\(\max(r,\widetilde{r})\). Precisely, this variety is the moduli space
of tuples of monic univariate polynomials
$f_1(t),..., f_{r+\widetilde{r}}(t)$, with $f_i(t)$ of degree $d_i$,
such that the leading $n+1$ coefficients of $\prod_{i=1}^rf_i(t)$ and
$\prod_{i=r+1}^{r+\widetilde{r}} f_i(t)$ agree.  The equality of the
leading coefficient is trivial, while the equality of the remaining
$n$ coefficients is a system of $n$ polynomial equations of degrees
$\max(r, \widetilde{r})$ in $\sum_{i=1}^{r+\widetilde{r}}d_i$
variables.  One checks that this affine variety is also a complete
intersection.  While Sawin's original result uses Katz's bound,
applying our Theorem~\ref{theorem:main} yields a sharper estimate.
Specifically, the upper bound
\(4(2+\max(r,\widetilde{r}))^{n+\sum_{i=1}^{r+\widetilde{r}}d_{i}}\)
from \cite[(2-10)]{sawin20} can be improved to
\[
\binom{\sum_{i=1}^{r+\widetilde{r}}d_{i}-1}{n-1}(\max(r,\widetilde{r})+2)^{\sum_{i=1}^{r+\widetilde{r}}d_{i}} < (2\max(r,\widetilde{r})+4)^{\sum_{i=1}^{r+\widetilde{r}}d_{i}}.
\]

The improved Betti number bounds combined with Sawin's results lead
to corresponding improvements in the error terms of various short
interval estimates in \cite{sawin20,Saw21}. For example,
the constant factor $3(k+2)^{2n-h}$ in \cite[Theorem 1.1]{Saw21} can be
replaced by $(2k+2)^n$, and similarly the
constant factor $3(n+2)^{2n-h}$ in \cite[Theorem 1.2]{Saw21} can be
replaced by $(2n+2)^n$.

\item
The proofs of Theorem~\ref{theorem:order} and
Theorem~\ref{theorem:main} are quite different, yet both are related
to the following specialization principle: when working with a family
of perverse sheaves on an affine family of varieties, the dimension of
the key cohomology factor decreases upon specialization.  The precise
statement is provided in Lemma~\ref{lemma:middle-estimate} and, in
mixed characteristics, in Variant~\ref{variant:middle-estimate}.

Theorem~\ref{theorem:main} follows from the aforementioned principle
and a Lefschetz-type argument.  Specifically, the principle enables us
to bound the ``middle cohomology'' by that of a sufficiently general
\(V\), which can be computed exactly.  The Lefschetz-type theorem
(Lemma~\ref{lemma:gysin-perv}) allows us to reduce higher cohomology
to the middle cohomology of complete intersections in a smaller affine
space, thereby enabling the use of induction.

However, this argument does not imply Theorem~\ref{theorem:order}.
While it provides a bound for \(\dim \mathrm{H}^{n-r}_{c}\) (which
becomes irrelevant if \(r\) is too large), the limitations of the
Lefschetz-type theorem prevent us from obtaining information about
\(\dim \mathrm{H}^{i}_{c}\) when \(i\) lies between \(n-r\) and
\(\dim V\).

We will prove Theorem~\ref{theorem:order} by revisiting
Adolphson–Sperber's total degree bounds for L-functions of toric exponential
sums from the perspective of \(\ell\)-adic cohomology.  Instead of
relying on Dwork theory, we will use the specialization
Lemma~\ref{lemma:middle-estimate} as our primary tool.  This approach
allows us to improve their bounds purely through \(\ell\)-adic
cohomology.  In doing so, we also achieve a cohomological enhancement
that leads directly to Theorem~\ref{theorem:order} via a toric reduction and a volume computation.  Further details
on exponential sums and our new bounds are discussed in
\S\ref{sec:intro-exp} below.
\end{enumerate}
\end{remarks*}

\subsection{Compactly supported Betti numbers of exponential sums}
\label{sec:intro-exp}
Let \(q\) be a power of a prime \(p\).  Let \(U\) be an
algebraic variety over \(\mathbb{F}_{q}\), and let
\(f\colon U \to \mathbb{A}^{1}_{\mathbb{F}_{q}}\) be a morphism.  For
each \(m \geq 1\), the exponential sum in \(\mathbb{F}_{q^{m}}\)
associated to \(f\) is
\begin{equation}
S_{m}(f;\psi) = \sum_{x \in U(\mathbb{F}_{q^{m}})} \psi(\operatorname{Tr}_{\mathbb{F}_{q^{m}}/\mathbb{F}_{p}}f(x)),
\end{equation}
where \(\psi\colon \mathbb{F}_{p} \to \mathbb{C}^{\ast}\) is a
nontrivial additive character.  In the sequel we shall fix one such
nontrivial \(\psi\), and we shall simply write \(S_{m}(f)\) instead of
\(S_{m}(f;\psi)\).

Associated to these exponential sums is their L-function
\begin{equation}\label{eq:l-function-exp-sum}
L_{f}(t) = \exp\left\{ \sum_{m=1}^{\infty} S_{m}(f) \frac{t^{m}}{m} \right\}.
\end{equation}
By a well-known theorem of Dwork and Grothendieck, \(L_{f}(t)\) is a
rational function expressed as \(L_{f}(t) = P(t)/Q(t)\), where
\(P, Q \in 1 + t\mathbb{Z}[\zeta_{p}, t]\).  Here, \(\zeta_{p}\) is a
primitive \(p\)\textsuperscript{th} root of unity.
When \(P\) and \(Q\) are taken to be coprime, the sum
\(\deg P + \deg Q\) is referred to as the \emph{total degree} of the rational function
\(L_{f}(t)\).

In \cite{adolphson-sperber:newton-polyhedra-degree-l-function},
Adolphson and Sperber proved an upper bound for the total degree of
\(L_{f}(t)\) when \(U\) is an algebraic torus
\(\mathbb{G}_{\mathrm{m}}^{n}\), generalizing and refining previous
results of Bombieri \cite{bombieri:exponential-sums-in-finite-fields-2}.
Their bound is expressed in terms of a combinatorial invariant of
\(f\), called the \emph{Newton polytope} of \(f\) with respect to
\(\infty\), denoted by \(\Delta_{\infty}(f)\) (see
\S\ref{sec:tori-exp-sums} for the definition).  This is a convex
polytope in \(\mathbb{R}^{n}\) containing the origin.
Adolphson and Sperber
\cite[Theorem~1.8]{adolphson-sperber:newton-polyhedra-degree-l-function}
proved that, among others, if \(\dim \Delta_{\infty}(f) = n\), then
\begin{equation}\label{eq:as-exp-total-degree}\tag{AS}
\text{Total degree of }L_{f}(t) \leq 2^{n}(1 + 2^{\frac{n+1}{n}})^{n} n! \text{Vol}(\Delta_{\infty}(f)).
\end{equation}
We shall prove a \emph{cohomologically enhanced} improvement of
\eqref{eq:as-exp-total-degree}, which yields
\[
\text{Total degree of }L_{f}(t) \leq 2^{n}n!\mathrm{Vol}(\Delta_{\infty}(f)).
\]

Just like the zeta function of a variety, the L-function
\eqref{eq:l-function-exp-sum} also has a cohomological interpretation.
The additive character \(\psi\) determines an \emph{Artin--Schreier}
sheaf \(\mathcal{L}_{\psi}\) on \(\mathbb{A}^{1}_{\mathbb{F}_{q}}\)
(see \cite[Sommes~trig.]{deligne:sga4.5} for a more systematic
introduction).  The link between the Artin--Schreier sheaf and
exponential sums is provided by the Grothendieck trace formula:
\begin{equation}
\label{eq:trace-formula}
S_{m}(f) = \sum_{i \in \mathbb{Z}} (-1)^{i} \mathrm{Tr}(F^{m}|\mathrm{H}^{i}_{c}(U_{\overline{\mathbb{F}}_{q}};f^{\ast}\mathcal{L}_{\psi})),
\end{equation}
where \(\overline{\mathbb{F}}_{q}\) is an algebraic closure of
\(\mathbb{F}_{q}\),
\(U_{\overline{\mathbb{F}}_{q}} = U \otimes_{\mathbb{F}_{q}}
\overline{\mathbb{F}}_{q}\), and \(F^{m}\) is the
\(m\)\textsuperscript{th} iteration of the \(q\)-power Frobenius
\(F\).  The multiplicative form of the trace formula
\eqref{eq:trace-formula} is
\begin{equation}\label{eq:multiplicative-trace-formula}
\tag{\ref*{eq:trace-formula}\({}^{\prime}\)}
L_{f}(t) = \prod_{i\in \mathbb{Z}} \det(1 - t \cdot F| \mathrm{H}^{i}_{c}(U_{\overline{\mathbb{F}}_{q}};f^{\ast}\mathcal{L}_{\psi}))^{(-1)^{i-1}}.
\end{equation}

The formulae \eqref{eq:trace-formula} and
\eqref{eq:multiplicative-trace-formula} allow us to use the
eigenvalues of
\(F|\mathrm{H}^{i}_{c}(U_{\overline{\mathbb{F}}_{q}};f^{\ast}\mathcal{L}_{\psi})\)
to deduce properties of the exponential sums and their L-function.
Contemplating that there might be cancellations of linear factors in
\eqref{eq:multiplicative-trace-formula}, we have
\begin{equation}
\text{Total degree of }L_{f}(t) \leq \sum \dim \mathrm{H}^{i}_{c}(U_{\overline{\mathbb{F}}_{q}};f^{\ast}\mathcal{L}_{\psi}).
\end{equation}
Thus, any upper bound of the compactly supported Betti sum will give
an upper bound of the total degree of \(L_{f}(t)\), but the converse
is only true when
\[
P_{\text{even}}(t) = \prod_{i\text{ even}} \det(1-tF|\mathrm{H}^{i}_{c}(U_{\overline{\mathbb{F}}_{q}};f^{\ast}\mathcal{L}_{\psi})) \]
\text{and}
\[
P_{\text{odd}}(t) = \prod_{i\text{ odd}} \det(1-tF|\mathrm{H}^{i}_{c}(U_{\overline{\mathbb{F}}_{q}};f^{\ast}\mathcal{L}_{\psi}))
\]
do not share common zeros.  This latter condition holds when the
Artin--Schreier cohomology spaces of \(f\) are pure, but should fail
in general.

We shall provide an upper bound for the sum of compactly supported
cohomology dimensions of \(f^{\ast}\mathcal{L}_{\psi}\), expressed in terms of
Minkowski mixed volumes of polytopes, recalled below.

\begin{notation}[Minkowski mixed volume]\label{notation:minkowski-mixed-volume}
For a monomial
\(m(T_{1},T_{2},\ldots,T_{r}) = T_{1}^{a_{1}}\cdots T_{r}^{a_{r}}\),
we set
\begin{equation*}
m(\Delta_{1},\ldots,\Delta_{r})=
\begin{cases}
0 & \text{if } a_{1} + \cdots + a_{r} \neq n, \\
n!V(\Delta_{1}[a_{1}],\ldots,\Delta_{r}[a_{r}]) & \text{if } a_{1} + \cdots + a_{r} = n.
\end{cases}
\end{equation*}
Here \(V(\Delta_{1}[a_{1}],\ldots,\Delta_{r}[a_{r}])\) is the
\emph{mixed volume} of the convex polytopes \(\Delta_{1},\ldots,\Delta_{r}\) in \(\mathbb{R}^{n}\) (with indicated
repetitions).  These numbers can be defined as the coefficients in the
expansion of volume of Minkowski sums:
\begin{equation*}
\mathrm{Vol}(\lambda_{1}\Delta_{1}+\cdots + \lambda_{r}\Delta_{r}) =
\sum_{\substack{i_{1}+\cdots + i_{r} = n \\ 0\leq i_{1},\ldots,i_{r}\leq n}}
\frac{n!}{i_{1}!\cdots i_{r}!} V(\Delta_{1}[i_{1}],\ldots,\Delta_{r}[i_{r}]) \lambda_{1}^{i_{1}}\cdots \lambda_{r}^{i_{r}}.
\end{equation*}
Thus, for an \(n\)-dimensional lattice polytope $\Delta$ in $\mathbb{R}^n$, we have
\(\Delta^{n} = n!\mathrm{Vol}(\Delta)\) is the usual normalized
volume.

For any formal power series
\(u(T_{1},\ldots,T_{r}) \in \mathbb{Z}[\![T_{1},\ldots,T_{r}]\!]\),
\(u = \sum b_{w} T_{1}^{w_{1}}\cdots T_{r}^{w_{r}}\), we set
\(u(\Delta_{1},\ldots,\Delta_{r}) = \sum b_{w}\Delta^{w_{1}}_{1}\cdots \Delta_{r}^{w_{r}}\).
\end{notation}

The toric case of our main result for exponential sums is

\begin{theorem}[= \ref{eq:total-degree-bound-toric-exponential-sum}]
\label{theorem:as-improvement}
Let \(f\colon \mathbb{G}_{\mathrm{m}}^{n}\to\mathbb{A}^{1}\) be
a Laurent polynomial whose Newton polytope
\(\Delta=\Delta_{\infty}(f)\) is \(n\)-dimensional\footnote{if
  \(\Delta\) is not \(n\)-dimensional, then we can write
  \(\mathbb{G}_{\mathrm{m}}^{n}=\mathbb{G}_{\mathrm{m}}^{n^{\prime}}\times\mathbb{G}_{\mathrm{m}}^{n^{\prime\prime}}\),
  and \(f\) factors as
  \(\mathbb{G}_{\mathrm{m}}^{n}\to\mathbb{G}_{\mathrm{m}}^{n^{\prime}}\xrightarrow{f^{\prime}}\mathbb{A}^{1}\),
  where the first arrow is the projection to
  \(\mathbb{G}_{\mathrm{m}}^{n^{\prime}}\), and the Newton polytope of
  \(f^{\prime}\) in \(\mathbb{R}^{n\prime}\) is of dimension
  \(n^{\prime}\).  Therefore, it suffices to consider Laurent
  polynomials whose Newton polytope is \(n\)-dimensional.}.
Let \(\psi\) be a nontrivial additive character.  Then
\begin{equation*}
\sum_i \dim \mathrm{H}^{i}_{c}(\mathbb{G}^{n}_{\mathrm{m},\overline{\mathbb{F}}_{q}};f^{\ast}\mathcal{L}_{\psi})
\leq \inf_{S}\left\{ \Delta^{n} + \sum_{i=1}^{n} 2^{i-1}\Delta^{n-i}S^{i} \right\}.
\end{equation*}
Here \(S\) ranges in the set of all \(n\)-dimensional lattice polytopes in \(\mathbb{R}^{n}\).
\end{theorem}

Since \(S\) can be chosen arbitrarily in
Theorem~\ref{theorem:as-improvement}, we can take \(S = \Delta\).
Then we get
\begin{equation}\label{eq:more-like-as}
\sum_i \dim \mathrm{H}^{i}_{c}(\mathbb{G}_{\mathrm{m},\overline{\mathbb{F}}_{q}}^{n};f^{\ast}\mathcal{L}_{\psi})
\leq 2^{n}\Delta^{n} = 2^{n} n!\mathrm{Vol}(\Delta).
\end{equation}
This improves the bound of Adolphson--Sperber
\eqref{eq:as-exp-total-degree}.

One novelty of Theorem~\ref{theorem:as-improvement} is that the bounds
are obtained at the cohomology level, for compactly supported Betti
numbers, rather than merely addressing the total degree of the
L-function.  This allows us to deduce the Betti sum bound
(Theorem~\ref{theorem:order}) through an elementary counting argument (see \S\ref{sec:proof-theorem-order}) via a toric reduction and a volume computation.

Taking \(S=\Delta\) is a crude choice in many cases, because the
Minkowski mixed volume gets smaller if \(S\) is small and more
similar to \(\Delta\).  Thus, if we can choose a rational number
\(d >2\) such that \(\Delta = d S\) for some lattice polytope $S$, then we get
\begin{align}\label{eq:power-as}
\sum_i \dim \mathrm{H}^{i}_{c}(\mathbb{G}_{\mathrm{m},\overline{\mathbb{F}}_{q}}^{n};f^{\ast}\mathcal{L}_{\psi})
&\leq n!\mathrm{Vol}(S) \cdot \left[  d^{n} + \frac{d^{n}-2^{n}}{d-2} \right]  \\
&= n!\mathrm{Vol}(\Delta) \cdot \left[  1 + \frac{1-{(\frac{2}{d}})^{n}}{{d}-2}\right].
  \nonumber
\end{align}
If $d$ is large, the upper bound on the right side is at most
$(1+\frac{1}{d-2})n!\mathrm{Vol}(\Delta)$, which is close to
the optimal number $n!\mathrm{Vol}(\Delta)$ achieved in the non-degenerate case as shown in \cite{adolphson-sperber:exponential-sums-newton-polyhedra}.

For an $n$-dimensional lattice polytope $\Delta$ in $\mathbb{R}^n$ containing the origin, define
\begin{equation*}
E_{c}(\Delta) \coloneqq \sup_{\ell\not= p} \left\{\sum_i \dim \mathrm{H}^{i}_{c}(\mathbb{G}_{\mathrm{m},\overline{\mathbb{F}}_{p}}^{n};f^{\ast}\mathcal{L}_{\psi}):
  f \in \overline{\mathbb{F}}_{p}[x_1^{\pm 1}, \cdots, x_n^{\pm 1}], \
  \Delta_{\infty}(f)=\Delta
\right\},
\end{equation*}
where $p$ and $\ell$ run over all prime numbers with $p \not= \ell$, $\psi$ runs over all non-trivial additive characters of $\mathbb{F}_p$.
The number $E_{c}(\Delta)$ clearly depends only on the
lattice polytope $\Delta$. Replacing $\Delta$ by the multiple $d\Delta$, the above discussion gives

\begin{corollary}For an $n$-dimensional lattice polytope $\Delta$ in $\mathbb{R}^n$ containing the origin, we have
\begin{equation*}
 n!\mathrm{Vol}(\Delta)d^n   \leq E_c(d\Delta) \leq
 n!\mathrm{Vol}(\Delta)d^n \cdot \left[  1 + \frac{1-{(\frac{2}{d}})^{n}}{{d}-2}\right]
\end{equation*}
 In particular, as a function in $d$, we have the asymptotic formula
 \begin{equation*}
  E_c(d\Delta) =
 n!\mathrm{Vol}(\Delta)\left( d^n + O(d^{n-1}) \right).
 \end{equation*}
\end{corollary}

The toric exponential sums are examples of exponential sums over a
nonsingular variety.  In applications, however, studying exponential
sums over a singular variety is inevitable.  Although in principle
singular exponential sums can be reduced to nonsingular exponential
sums, the estimates coming from the standard reduction procedure are
often not optimal.  Previously, Katz \cite{katz:singular-sums} studied
exponential sums over singular affine complete intersections via the
weak Lefschetz theorem.  Our approach combines the Lefschetz method as
well as the aforementioned specialization method.  The general
abstract result is formulated in \S\ref{sec:katz-framework}.  In this
introduction, we are content with stating the following theorem:

\begin{theorem}\label{theorem:last}
Let \(V\) be a closed subvariety of
\(\mathbb{A}^{n}_{\mathbb{F}_{q}}\).  Let
\(f \in \mathbb{F}_{q}[x_{1},\ldots,x_{n}]\) be a polynomial of degree
\(\leq d\).
\begin{enumerate}
\item\label{item:bomb-restricted-to-sub}
\textup{(See \ref{sec:proof-theorem-order})}
If \(V\) is cut out by \(r\) polynomials of
degree \(\leq d\), then
\begin{equation*}
\sum_i \dim \mathrm{H}^{i}_{c}(V_{\overline{\mathbb{F}}_{q}};f^{\ast}\mathcal{L}_{\psi})
\leq 3^{r} \binom{n+r}{r}(2d+1)^{n}.
\end{equation*}

\item \label{item:bomb2} If \(V\) is cut out by some polynomials of degree at most \(d\), then
\begin{equation*}
\sum_i \dim \mathrm{H}^{i}_{c}(V_{\overline{\mathbb{F}}_{q}};f^{\ast}\mathcal{L}_{\psi})
\leq 3^{n+1}\binom{2n+1}{n+1}(2d+1)^{n}.
\end{equation*}

\item \textup{(See \ref{corollary:katz-to-an})} If \(V\) is cut out
by \(r\) polynomials of degree \(\leq d\),
and \(\dim V = n-r\), then
\begin{equation*}
\sum_i \dim \mathrm{H}^{i}_{c}(V_{\overline{\mathbb{F}}_{q}};f^{\ast}\mathcal{L}_{\psi})
\leq \binom{n}{r}(d+1)^{n}.
\end{equation*}
\end{enumerate}
\end{theorem}

In the
context of (\ref{item:bomb-restricted-to-sub}), Bombieri
\cite{bombieri:exponential-sums-in-finite-fields-2} proved a total
degree upper bound \((4d+5)^{n+r}\); while Katz
\cite{katz:sums-of-betti-numbers} improved this upper bound to
$3(d+2)^{n+r}$.

Item (\ref{item:bomb2}) follows from
(\ref{item:bomb-restricted-to-sub}) by applying Kronecker's theorem,
as in the proof of Corollary~\ref{corollary:total-bound-without-r}.
To see how far these bounds from being optimal as $d$ varies, let us
introduce
\begin{equation*}
E_{c}(n,r;d) \coloneqq \sup_{\ell\not =p}\left\{ \sum_i \dim \mathrm{H}^{i}_{c}(V_{\overline{\mathbb{F}}_{p}};f^{\ast}\mathcal{L}_{\psi}):
\begin{array}{l}
  V = \operatorname{Spec} \overline{\mathbb{F}}_{p}[x_{1},\ldots,x_{n}]/(f_{1},\ldots,f_{r}),  \\
  \deg f_{i} \leq d
\end{array}
\right\},
\end{equation*}
\begin{equation*}
E^{\text{ci}}_{c}(n,r;d) \coloneqq \sup_{\ell\not = p}\left\{ \sum_i \dim \mathrm{H}^{i}_{c}(V_{\overline{\mathbb{F}}_{p}};f^{\ast}\mathcal{L}_{\psi}):
\begin{array}{l}
  V = \operatorname{Spec} \overline{\mathbb{F}}_{p}[x_{1},\ldots,x_{n}]/(f_{1},\ldots,f_{r}),  \\
  \deg f_{i} \leq d, \text{ and } \dim V = n-r
\end{array}
\right\},
\end{equation*}
where $p$ and $\ell$ run over all prime numbers with $p \not= \ell$, $\psi$ runs over all non-trivial additive characters of $\mathbb{F}_p$,
and $f$ runs over all
polynomials in $\overline{\mathbb{F}}_{p}[x_{1},\ldots,x_{n}]$ with $\deg(f)\leq d$.

By adding redundant equations, if necessary, one sees that the number $E_{c}(n,r;d)$ is increasing in $r$. Namely,
\begin{equation*}
E_c(n, 1;d) \leq E_c(n, 2;d) \leq \cdots \leq E_c(n, r;d)  \leq E_c(n, r+1;d)\leq \cdots.
\end{equation*}
Furthermore, by Kronecker's reduction, this sequence in $r$ stabilizes once $r$ reaches $n+1$.
Thus, without loss of generality, we can always assume that $r\leq n+1$.
In terms of these notations, the above corollary
can be restated as

\begin{corollary} Let $n, r, d$ be positive integers.
Then, we have
\begin{equation*}
 E_c(n, r;d) \leq  3^{r} \binom{n+r}{r}(2d+1)^{n}, \ \
 E_c(n, r;d) \leq  3^{n+1} \binom{2n+1}{n+1}(2d+1)^{n}.
\end{equation*}
\end{corollary}

For exponential sums over a complete intersection, we have an
asymptotic formula.

\begin{corollary} Let $n, r, d$ be positive integers with $r\leq n$.
We have
\begin{equation*}
\binom{n}{r}(d-1)^{n} \leq E^{\mathrm{ci}}_c(n,r;d) \leq  \binom{n}{r}(d+1)^{n}.
\end{equation*}
In particular
\(E^{\textup{ci}}_c(n, r;d) = \binom{n}{r}d^{n} + O_{n,r}(d^{n-1})\).
\end{corollary}

\begin{proof}
When \(d\) is coprime to the characteristic of \(p\), we will prove
(see Corollary~\ref{corollary:katz-to-an}(\ref{item:lowerbound})) that
for a sufficiently general pair \((V,f)\), the sum
\(\sum_{m} \dim \mathrm{H}^{m}_{c}(V;f^{\ast}\mathcal{L}_{\psi}) \geq\binom{n}{r}(d-1)^{n}\).  Since $E_c^{\mathrm{ci}}(n,r;d)$ is a supremum taken over all primes $p$,
this establishes the lower bound.
\end{proof}

Since \(E_{c}(n,r;d) \geq E_{c}^{\mathrm{ci}}(n,r;d)\), the above two
corollaries together imply \(E_{c}(n,r;d) \asymp_{n,r} d^{n}\).
Hence, the correct order of \(E_{c}(n,r;d)\) is \(d^{n}\).  As before,
a natural question arises: is there an asymptotic formula for
\(E_{c}(n,r;d)\) as \(d \to \infty\)?  Namely, is there a positive
constant $e_{n,r}$ such that
\begin{equation*}
  E_{c}(n,r;d)  = e_{n,r}d^n +o_{n,r}(d^n)?
\end{equation*}
If so, what is the leading constant $e_{n,r}$?

\begin{remark*}
An expert in Dwork theory understands that the total degree bounds of
Bombieri and Adolphson--Sperber \eqref{eq:as-exp-total-degree} stem
from chain-level overcounts.  When viewed cohomologically, these
bounds can naturally be refined to Betti number bounds for \emph{some}
cohomology, such as Berthelot's rigid cohomology.  The necessary tools
to establish this have been laid out in our previous work \cite{wz1}.

On the other hand, as Dwork theory is fundamentally a \(p\)-adic theory, it
yields Betti number bounds only for \(p\)-adic cohomology.  Since the
equality between rigid cohomology Betti numbers and \(\ell\)-adic
Betti numbers is still conjectural, the Dwork-theoretic method does not
establish the \(\ell\)-adic version of \eqref{eq:more-like-as}.
\end{remark*}

\subsection{Upper semicontinuity of the middle Newton polygon}
The other novelty of our approach is that it allows us to extract more
information than just Betti numbers.  Using the Grothendieck
specialization theorem of Newton polygons, we can infer the \(p\)-adic
information of the Frobenius eigenvalues.

Recall that we can use the Newton polygon to capture \(p\)-adic
information about the roots of a polynomial.  If
\(P(t) = a_{0} + a_{1}t + \cdots + a_{n}t^{n}\) is a polynomial with
coefficients in \(\overline{\mathbb{Q}}_{p}\) and \(a_{0} \neq 0\),
and if \(q\) is a power of \(p\), then the Newton polygon of \(P\)
with respect to \(\mathrm{ord}_{q}\) is defined as the convex hull in $\mathbb{R}^2$ of
the points
\[
(0, \mathrm{ord}_{q}(a_{0})), (1, \mathrm{ord}_{q}(a_{1})), \ldots, (n, \mathrm{ord}_{q}(a_{n})).
\]

If \(L\) is a side of the Newton polygon of \(P(t)\) (i.e., \(L\) is a
line segment joining two vertices of the Newton polygon) with slope
\(s\), and the projection of \(L\) onto the horizontal axis has length
\(m\), we say that \(s\) is a slope of the Newton polygon and \(m\) is
the multiplicity of \(s\).  When \(s\) is a slope of the Newton
polygon with multiplicity \(m\), there are precisely \(m\) roots of \(P\)
(counting multiplicity) \(r_{1}, \ldots, r_{m}\) in
\(\overline{\mathbb{Q}}_{p}\) such that
\(\mathrm{ord}_{q}(r_{i}^{-1}) = s\).

Let
\(\psi\colon \mathbb{F}_{p} \to \overline{\mathbb{Q}}_{\ell}^{\ast}\)
be a nontrivial additive character, and let
\(f\colon U \to \mathbb{A}^{1}\) be a regular function between
varieties over \(\mathbb{F}_{q}\).  We are interested in investigating
the Newton polygon of the polynomials
\[
\iota\det(1-tF|\mathrm{H}^{i}_{c}(U_{\overline{\mathbb{F}}_{q}};f^{\ast}\mathcal{L}_{\psi})),
\]
where \(\iota\) is an isomorphism between
\(\overline{\mathbb{Q}}_{\ell}\) and \(\overline{\mathbb{Q}}_{p}\).
The integrality of the Frobenius eigenvalues (cf.~\cite[XXI,
Appendice]{sga7-2}) implies that the Newton polygon of
\(\iota\det(1-tF|\mathrm{H}^{i}_{c}(U_{\overline{\mathbb{F}}_{q}};f^{\ast}\mathcal{L}_{\psi}))\)
is independent of the choice of \(\iota\).

The Newton polygons associated with a ``nondegenerate'' toric
exponential sum have been thoroughly studied by
Adolphson and Sperber
\cite{adolphson-sperber:exponential-sums-newton-polyhedra} through
Dwork theory, where they establish the ``Newton above Hodge''
inequality. This inequality is shown to be an equality in many cases, see \cite{Wan93,Wan04}.
However, for degenerate exponential sums or those that are
not toric in nature, very little is known.

The specialization lemma~\ref{lemma:middle-estimate} has an arithmetic
variant (Variant~\ref{variant:newton}).  When applied to exponential
sums, it yields the following global arithmetic analogue of
Steenbrink's lower semicontinuity theorem for Hodge spectra
(\cite{steenbrink:semicontinuity-of-spectrum}):

\begin{theorem}[Upper semicontinuity of the middle Newton polygon, = \ref{variant:np-exp}]
\label{theorem:semicont-newton}
Let \(\Gamma \subset \mathbb{R}_{\geq 0}^{2}\) be a convex polygon
with \(0\) as one of its vertices.  Let
\(f\colon U \times S \to \mathbb{A}^{1}\) be a morphism of varieties
over \(\mathbb{F}_{q}\), where \(S\) is a smooth, geometrically
connected curve, and \(U\) is a purely \(n\)-dimensional affine
variety with at worst local complete intersection singularities.
Suppose there is an open dense subset \(S^{\prime}\) of \(S\) such
that for any closed point \(s\) of \(S^{\prime}\), the following
holds:
\begin{itemize}[label={}]
\item If \(\kappa(s) = \mathbb{F}_{q^{m}}\), then the Newton polygon
(with respect to the valuation \(\operatorname{ord}_{q^{m}}\)) of
\[
\det(1-tF^{m}|\mathrm{H}^{n}(U_{\overline{\mathbb{F}}_{q}} \times \{\overline{s}\}; f^{\ast}\mathcal{L}_{\psi}))
\]
is on or above \(\Gamma\), where \(F\colon U \to U\) is the
\(q\)-power Frobenius substitution, and \(\overline{s}\) is a
\(\overline{\mathbb{F}}_{q}\)-valued geometric point of \(S\) lying
above \(s\).
\end{itemize}
Then the same property holds for all \(s\).
\end{theorem}

Applying this to \(U = \mathbb{A}^{n}\), and utilizing
Adolphson--Sperber's bound for nondegenerate polynomials along with the
Lefschetz-type theorem~\ref{lemma:gysin-perv}, we obtain the following
corollary regarding the Newton polygon for exponential sums over the
affine space, which does not appear to be known.

\begin{corollary}[= \ref{theorem:bombieri-improvement}]
Suppose
\(f\colon\mathbb{A}_{\mathbb{F}_{q}}^{n}\to\mathbb{A}_{\mathbb{F}_{q}}^{1}\)
is any polynomial of degree \(\leq d\).  Assume that \((q,d) = 1\).
Then for $0\leq j\leq n$, the Newton polygon of
\[
\det(1-tF|\mathrm{H}^{n+j}_{c}(\mathbb{A}_{\overline{\mathbb{F}}_{q}}^{n}; f^{\ast}\mathcal{L}_{\psi}))
\]
lies on or above the Newton polygon of
\[
\prod_{m=0}^{(n-j)(d-2)}(1 - q^{j+\frac{m+n-j}{d}}t)^{U_{m,n-j}},
\]
where \(U_{m,n}\) is the coefficient of \(x^{m}\) in the expansion of \((1+x+\cdots+x^{d-2})^{n}\).
\end{corollary}


The hypothesis that \((d,q) = 1\) is a limitation of the
specialization method, as we need to know the information regarding
the Newton polygon of a generic member.  However, we believe that the
corollary should still hold true even when \((d,q) \neq 1\).

\subsection*{Organization}
The paper is organized as follows.  In \S\ref{sec:preliminaries}, we
fix some notation and recall some basics in sheaf theory.  In
\S\ref{sec:specialization}, we prove the specialization lemma
mentioned above.  \S\ref{sec:betti} is devoted to the proof of
Theorem~\ref{theorem:main} and related results.  In
\S\ref{sec:bounds-exp-sums} we study exponential sums from a broader
context than toric ones.  Specifically, we shall explain how to make
an exact computation for ``well-behaved'' functions.
Theorem~\ref{theorem:as-improvement} will then be deduced from a
theorem of Khovanskii
\cite{khovanskii:newton-polyhedra-and-genus-of-complete-intersections}.
In this section, we also use the corollary \eqref{eq:more-like-as} of
Theorem~\ref{theorem:as-improvement} to prove
Theorem~\ref{theorem:order}.

\subsection*{Acknowledgment}

We extend our sincere gratitude to Haoyu Hu for the stimulating
conversations and valuable exchange of ideas during our visits to
Nanjing in Spring 2024 and Shenzhen in Winter 2024. The second author
would like to thank Yongqiang Liu and L.~Maxim for the enlightening
conversations during his visit to Hefei in Fall 2024.

\section{Notation, conventions, and recapitulations}
\label{sec:preliminaries}
In the following, \(k\) will be a field, and \(\ell\) will be a prime
number.  When working over a base scheme \(S\), we shall always
implicitly assume that \(\ell\) is different from the characteristic
of the residue field of any point of \(S\).

\subsection{Perverse sheaves}

We shall use the language of perverse sheaves throughout.  In
\cite{beilinson-bernstein-deligne:perverse-sheaves}, a selfdual
perverse t-structure is constructed on the constructible derived
category \(D^{b}_{c}(X;\overline{\mathbb{Q}}_{\ell})\) of
\(\overline{\mathbb{Q}}_{\ell}\)-sheaves on a scheme \(X\) over a
finite field or an algebraically closed field.  We denote this
t-structure by
\(({}^{\mathrm{p}}\mathcal{D}_{X}^{\leq 0},
{}^{\mathrm{p}}\mathcal{D}^{\geq 0}_{X})\), and the perverse
cohomology of an object
\(\mathcal{F} \in D^{b}_{c}(X;\overline{\mathbb{Q}}_{\ell})\) is
denoted by \(\pH^{\ast}(\mathcal{F})\).  Objects of
\({}^{\mathrm{p}}\mathcal{D}_{X}^{\leq 0}\) are said to satisfy the
support condition, while objects of
\({}^{\mathrm{p}}\mathcal{D}^{\geq 0}_{X}\) are said to satisfy the
cosupport condition.

We will freely use basic properties of perverse sheaves documented in
the standard references such as
\cite{beilinson-bernstein-deligne:perverse-sheaves,
  kiehl-weissauer:weil-conjecture-perverse-sheaf-fourier-transform}

Here is a simple lemma that we will need.

\begin{lemma}%
\label{lemma:cosupp}
Let \(Y\) be an algebraic variety over an algebraically closed field.
Let \(\mathcal{F}\) be an object of
\({}^{\mathrm{p}}\mathcal{D}_Y^{\geq 0}\).  Let \(X\) be a closed
subvariety of \(Y\), which is locally defined by the vanishing of
\(r\) regular functions.  Let \(i\colon X \to Y\) be the closed
immersion.  Then \(i^{\ast}[-r]\) is left t-exact, i.e.,
\(i^{\ast}[-r]\left( {}^{\mathrm{p}}\mathcal{D}^{\geq0}_{Y} \right) \subset {}^{\mathrm{p}}\mathcal{D}_{X}^{\geq 0}\).
\end{lemma}

\begin{proof}
By induction we may assume \(r = 1\).  Let \(j\colon U \to Y\) be the
complement of \(X\) in \(Y\).  Let \(\mathcal{F}\) be an object of
\({}^{\mathrm{p}}\mathcal{D}^{\geq0}_{Y}\).  Consider the
distinguished triangle
\begin{equation*}
\mathcal{F}[-1] \to i_{\ast}i^{\ast}\mathcal{F}[-1] \to j_{!}j^{\ast}\mathcal{F} \xrightarrow{+1}.
\end{equation*}
Since \(j\) is both quasi-finite and affine, \(j_{!}\) is t-exact.
Since \(j^{\ast}\) is also t-exact, we find
\(j_{!}j^{\ast}\mathcal{F}\) still satisfies the cosupport condition.
Since \(\mathcal{F}[-1] \in {}^{\mathrm{p}}\mathcal{D}^{\geq1}_{Y}\)
trivially satisfies the cosupport condition, we conclude that
\(i_{\ast}i^{\ast}\mathcal{F}[-1]\) satisfies the cosupport condition
as well.  Since \(i\) is a closed immersion, \(i_{\ast}\) is fully
faithful and t-exact, it follows that \(i^{\ast}\mathcal{F}[-1]\)
satisfies the cosupport condition.  This completes the proof.
\end{proof}




\subsection{Newton polygon}

Now we recall some basics about Newton polygons associated with
sheaves on varieties.  Let \(k\) be an algebraic closure of
\(\mathbb{F}_{p}\), and fix an isomorphism \(\iota\) between
\(\overline{\mathbb{Q}}_{\ell}\) and \(\overline{\mathbb{Q}}_{p}\).

Suppose \(B_{0}\) is a scheme over \(\mathbb{F}_{q}\), and let
\(B = B_{0} \otimes_{\mathbb{F}_{q}} k\).  Let \(\mathcal{F}_{0}\) be
a Weil \(\overline{\mathbb{Q}}_{\ell}\)-sheaf on \(B_{0}\)
(cf.~\cite[Définition~1.1.10]{deligne:weil-2}).  Each
\(\mathbb{F}_{q^{m}}\)-valued point
\[
x_{0} \colon \operatorname{Spec}\mathbb{F}_{q^{m}} \to B_{0}
\]
defines a geometric point
\begin{equation}
\label{eq:factorization-geometric-point}
x \colon \operatorname{Spec}k \to \operatorname{Spec}\mathbb{F}_{q^{m}} \xrightarrow{x_{0}} B_{0}
\end{equation}
and a geometric Frobenius operation
\[
F_{x_{0}} \colon \mathcal{F}_{x} \xrightarrow{\sim} \mathcal{F}_{x}.
\]

\begin{definition}
The \emph{Newton polygon of \(\mathcal{F}\) at the closed point \(x\) of \(B\)} is defined as
\[
\mathrm{NP}(\mathcal{F}_{x}) = \text{Newton polygon of } \det(1-tF_{x_{0}}|\mathcal{F}_{x}) \text{ with respect to } \mathrm{ord}_{q^{m}}.
\]
\end{definition}

As the notation indicates, this is manifestly a geometric notion,
independent of the factorization
\eqref{eq:factorization-geometric-point}.  Indeed, each closed point
\(x\) of \(B\) can be regarded as a geometric point of \(B_{0}\).  If
\(x\) factors through an \(\mathbb{F}_{q^{m}}\)-valued point, it also
factors through many \(\mathbb{F}_{q^{ms}}\)-valued points (for any
\(s \geq 1\)):
\[
\begin{tikzcd}
x\colon \operatorname{Spec}k \ar{r} &
\operatorname{Spec}\mathbb{F}_{q^{ms}} \ar{r} \ar[bend left=30,swap]{rr}{y_{0}} &
\operatorname{Spec}\mathbb{F}_{q^{m}} \ar{r}{x_{0}} & B_{0}
\end{tikzcd}
\]
Thus, the geometric Frobenius of \(y_{0}\) satisfies
\(F_{y_{0}} = F_{x_{0}}^{s}\).  It follows that the Newton polygon of
\(\det(1 - t F_{x_{0}} |\mathcal{F}_{x})\) with respect to
\(\mathrm{ord}_{q^{m}}\) is the same as that of
\(\det(1-t F_{y_{0}}|\mathcal{F}_{x})\) with respect to
\(\mathrm{ord}_{q^{ms}}\).

One extreme case is when \(B_0 = \operatorname{Spec}\mathbb{F}_q\), and
\(\mathcal{F} = \mathrm{H}^i(X;\mathcal{E})\) for some scheme \(X_0\) over \(\mathbb{F}_q\), and some constructible \(\overline{\mathbb{Q}}_{\ell}\)-sheaf \(\mathcal{E}\) on \(X_0\).
In this case, we shall simply write \(\mathrm{NP}(\mathrm{H}^i(X;\mathcal{E}))\) instead.

Let \(B\) be a smooth and connected curve, and let \(B^{\prime}\) be
the Zariski open dense subset of \(B\) over which \(\mathcal{F}\)
restricts to a local system. Essentially by class field theory,
\(\det \mathcal{F}|_{B^{\prime}}\) has a constant Newton polygon on
\(B^{\prime}\) (cf.~\cite[(1.3.3)]{deligne:weil-2}). Hence, for
\(b^{\prime} \in B^{\prime}\), \(\mathrm{NP}(\mathcal{F}_{b^{\prime}})\)
all have the same starting point and endpoint.

For \(b \in B \setminus B^{\prime}\), let \(\overline{\eta}_{b}\) be a
geometric generic point of the strict localization of \(B\) at \(b\).
While the geometric Frobenius action on
\(\mathcal{F}_{\overline{\eta}_{b}}\) depends on some choice, its
characteristic polynomial is well-defined
(\cite[Lemme~(1.7.4)]{deligne:weil-2}). Therefore, we can also define
the Newton polygon \(\mathrm{NP}(\mathcal{F}_{\overline{\eta}_{b}})\)
for \(\overline{\eta}_{b}\). We have the following \(\ell\)-adic
version of Grothendieck's specialization theorem due to Deligne.

\begin{proposition}[{\cite[Corollaire~(1.10.7)]{deligne:weil-2}}]
\label{proposition:weak-grothendieck-specialization}
Let the notation be as above.  Suppose \(\Gamma\) is a convex polygon
such that for any closed point \(b^{\prime} \in B^{\prime}\), we have
\[
\mathrm{NP}(\mathcal{F}_{b^{\prime}}) \text{ lies on or above } \Gamma.
\]
Then for any \(b \in B \setminus B^{\prime}\), we have
\[
\mathrm{NP}(\mathcal{F}_{\overline{\eta}_{b}}) \text{ lies on or above } \Gamma.
\]
Moreover, if \(\mathrm{NP}(\mathcal{F}_{b^{\prime}})\) and \(\Gamma\)
have the same starting and endpoint, then
\(\mathrm{NP}(\mathcal{F}_{\overline{\eta}_{b}})\) and \(\Gamma\) also
have the same starting and endpoint.
\end{proposition}

\begin{proof}
While Deligne's proof assumes that
\(\mathrm{NP}(\mathcal{F}_{b^{\prime}})\) is constant, a careful
examination of the argument reveals that all we need is a collection
of inequalities.  A uniform lower bound for these
\(\mathrm{NP}(\mathcal{F}_{b^{\prime}})\) is sufficient to carry out
the proof.
\end{proof}

The existence of crystalline companions for \(\ell\)-adic local
systems implies that there is an open subset of \(B^{\prime}\) on
which the Newton polygon of \(\mathcal{F}_{b^{\prime}}\) is constant.
For our applications, the weaker form stated above suffices.

\subsection{Weak Lefschetz theorems}
In this subsection, we record some Lefschetz-type theorems that will
be used later.  We fix an algebraically closed ground field \(k\).

\begin{lemma}
\label{lemma:weak-lefschetz}
Let \(Y\) be a projective variety.  Assume that \(\mathcal{F}\) is an
object of \({}^{\mathrm{p}}\mathcal{D}^{\geq 0}_Y\).  Let \(X\) be the
intersection of \(r\) ample Cartier divisors on \(Y\).  Then the map
\[
\mathrm{H}^{j-r}(Y; \mathcal{F}) \to \mathrm{H}^{j-r}(X; \mathcal{F})
\]
is bijective for \(j < 0\) and injective for \(j = 0\).
\end{lemma}

\begin{proof}
We use induction on \(r\).  When \(r=0\), there is nothing to prove.
For \(r > 1\), let \(X = D_1 \cap \cdots \cap D_r\), with \(D_j\)
ample.  Let \(Y^{\prime}\) be the intersection
\(D_2 \cap \cdots \cap D_r\).  By Lemma~\ref{lemma:cosupp},
\(\mathcal{F}|_{Y^{\prime}}[-r+1]\) satisfies the cosupport condition.
Since \(Y^{\prime}\setminus X\) is affine, it follows from Artin's
vanishing theorem
(\cite[Corollaire~4.1.2]{beilinson-bernstein-deligne:perverse-sheaves})
that
\begin{equation*}
\mathrm{H}_{c}^j(Y^{\prime}\setminus X;\mathcal{F}[-r+1]) = 0
\end{equation*}
for \(j < 0\).  The result now follows from the long exact sequence
\begin{equation*}
\cdots \to \mathrm{H}^{m}_{c}(Y^{\prime}\setminus X;\mathcal{F}) \to \mathrm{H}^{m}(Y^{\prime};\mathcal{F}) \to \mathrm{H}^m(X;\mathcal{F}) \to \cdots.
\qedhere
\end{equation*}
\end{proof}

\begin{theorem}[Deligne]
\label{theorem:weak-lefschetz}
Let \(f\colon X \to \mathbb{P}^{N}\) be a quasi-finite morphism to a
projective space.  Let \(\mathcal{F}\) be an object in
\({}^{\mathrm{p}}\mathcal{D}_{X}^{\geq0}\).  Then for a sufficiently
general hyperplane \(B\), the restriction morphism
\begin{equation*}
\mathrm{H}^i(X;\mathcal{F}) \to \mathrm{H}^i(f^{-1}B; \mathcal{F}|_{f^{-1}B})
\end{equation*}
is injective if \(i = -1\), and bijective if \(i < -1\).
\end{theorem}

\begin{proof}
This is proved in
\cite[Corollary~A.5]{katz:affine-cohomological-transforms-perversity-monodromy}.
After going through the proof, we find that only the condition
\(\mathcal{F} \in {}^{\mathrm{p}}\mathcal{D}_{X}^{\geq0}\) was used
(Katz needed only \(\mathbb{D}\mathcal{F}\) to be ``semiperverse'',
i.e.,
\(\mathbb{D}\mathcal{F} \in {}^{\mathrm{p}}\mathcal{D}_{X}^{\leq0}\),
where \(\mathbb{D}(-)\) is the Verdier duality functor).
\end{proof}

\begin{lemma}
\label{lemma:gysin-perv}
Let \(f\colon X \to \mathbb{P}^N\) be a quasi-finite morphism.  Let
\(\mathcal{F} \in {}^{\mathrm{p}}\mathcal{D}^{\leq 0}_{X}\).  Then for
a sufficiently general hyperplane \(B\), the purity map
\begin{equation*}
\mathrm{H}^{m-2}_{c}(f^{-1}B;\mathcal{F}|_{X\cap B}(-1)) \to \mathrm{H}^m_{c}(X;\mathcal{F}),
\end{equation*}
is surjective if \(m = 1\),
and is bijective if \(m \geq 2\).
\end{lemma}

\begin{proof}
This is \cite[Lemma~3.8]{wz2}.
\end{proof}

\section{Specialization of Betti numbers and Newton polygons}
\label{sec:specialization}
\numberwithin{equation}{section}
Let \(k\) be an algebraically closed field, and \(\ell\) a prime
number invertible in \(k\).  Let \(B\) be a connected nonsingular
curve over \(k\).  Let \(\pi\colon V \to B\) be a morphism of
\(k\)-schemes.  Let \(\mathcal{F}\) be a constructible
\(\overline{\mathbb{Q}}_{\ell}\)-complex on \(V\).  Since
\(R\pi_{!}\mathcal{F}\) remains constructible, there is an open dense
immersion \(j\colon B^{\prime} \to B\) such that the cohomology
sheaves of \(j^{\ast}R\pi_{!}\mathcal{F}\) are local systems on
\(B^{\prime}\).

\begin{lemma}[Lower semicontinuity of the middle Betti number]\label{lemma:middle-estimate}
In the situation above, assume in addition that
\begin{itemize}
\item \(\mathcal{F}\) satisfies the cosupport condition, i.e.,
\(\mathcal{F} \in {}^{\mathrm{p}}\mathcal{D}_{V}^{\geq 0}\), and
\item \(\pi\) is an affine morphism.
\end{itemize}
Then for any \(b \in B \setminus B^{\prime}\),
and any \(b^{\prime} \in B^{\prime}\),
we have
\begin{equation*}
\dim \mathrm{H}^{-1}_{c}(\pi^{-1}(b); \mathcal{F}|_{\pi^{-1}(b)}) \leq \dim \mathrm{H}^{-1}_{c}(\pi^{-1}(b^{\prime}); \mathcal{F}|_{\pi^{-1}(b^{\prime})}).
\end{equation*}
\end{lemma}

\begin{proof}
By shrinking \(B\), we may assume that
\(B\setminus B^{\prime} = \{b\}\).  Let \(i\colon \{b\} \to B\) be the
inclusion morphism.  Then we have a distinguished triangle
\begin{equation*}
j_{!}j^{\ast}(R\pi_{!} \mathcal{F}) \to R\pi_{!} \mathcal{F} \to i_{\ast} i^{\ast} (R\pi_{!} \mathcal{F}).
\end{equation*}
Since \(\pi\) is affine, Artin's vanishing theorem
(\cite[Corollaire~4.1.2]{beilinson-bernstein-deligne:perverse-sheaves})
implies that \(R\pi_{!}\mathcal{F}\) satisfies the cosupport condition
as well.  In particular, we have
\(\pH^{-1}(R\pi_{!} \mathcal{F}) = 0\).  Therefore, applying the
perverse cohomology functors to the above distinguished triangle gives
rise to an exact sequence of perverse sheaves
\begin{equation*}
0 \to \pH^{-1}i_{\ast} i^{\ast} (R\pi_{!} \mathcal{F}) \to \pH^{0}j_{!}j^{\ast}(R\pi_{!} \mathcal{F})
\to \pH^{0}(R\pi_{!} \mathcal{F}).
\end{equation*}
Using the t-exactness of \(j_{!}\) (since \(j\) is an affine,
quasi-finite morphism, \(j_{!}\) is t-exact by
\cite[Corollaire~4.1.3]{beilinson-bernstein-deligne:perverse-sheaves}),
we have
\[
\pH^{0}j_{!}j^{\ast}(R\pi_{!}\mathcal{F})=j_{!}\pH^{0}j^{\ast}(R\pi_{!}\mathcal{F}).
\]
Since the usual cohomology sheaves of
\(j^{\ast}(R\pi_{!}\mathcal{F})\) are assumed to be local systems, and
since \(B^{\prime}\) is smooth, connected, and one dimensional, we have
\begin{equation*}
\pH^{0}j^{\ast}(R\pi_{!}\mathcal{F})=\mathcal{H}^{-1}(j^{\ast}(R\pi_{!}\mathcal{F}))[1] = (R^{-1}\pi_{!}\mathcal{F})|_{B^{\prime}}[1].
\end{equation*}
By proper base change, the rank of the local system
\((R^{-1}\pi_{!}\mathcal{F})|_{B^{\prime}}\) equals
the dimension of the cohomology space
\(\mathrm{H}^{-1}_{c}(\pi^{-1}(b^{\prime});\mathcal{F}|_{\pi^{-1}(b^{\prime})})\), for any
\(b^{\prime} \in B^{\prime}\).

On the other hand, since \(i_{\ast}\) is t-exact
(\cite[Lemma~4.1]{kiehl-weissauer:weil-conjecture-perverse-sheaf-fourier-transform}),
and since the perverse cohomology and the usual cohomology agree on a
point, we have
\begin{align}\label{eq:what-is-e}
  \pH^{-1}i_{\ast} i^{\ast} (R\pi_{!} \mathcal{F})
  &= i_{\ast} \mathrm{H}^{-1}(i^{\ast} (R\pi_{!} \mathcal{F})) \\
 \text{(proper base change)} &= i_{\ast}\mathrm{H}_{c}^{-1}(\pi^{-1}(b); \mathcal{F}|_{\pi^{-1}(b)}). \nonumber
\end{align}
Thus, to prove the lemma, it suffices to apply the following simple
claim.

\begin{claim}\label{claim:middle}
Let \(B, b, i, j\) be as above.  Suppose
\(i_{\ast}E \to j_{!}\mathcal{E}[1]\) is an injective morphism
between perverse sheaves on \(B\), where \(E\) is a
\(\overline{\mathbb{Q}}_{\ell}\)-vector space, regarded as a
\(\overline{\mathbb{Q}}_{\ell}\)-sheaf on \(\{b\}\), and
\(\mathcal{E}\) is a local system on \(B^{\prime}\), then
\(\dim E \leq \operatorname{rank}\mathcal{E}\).
\end{claim}

To prove the claim, we apply the unipotent vanishing cycle functor
\({}^{\mathrm{p}}\Phi_{t,1} = R\Phi_{t,1}[-1]\) to the injective
morphism \(i_{\ast}E \to j_{!}\mathcal{E}[1]\), where
\(t\colon B \to \mathbb{A}^{1}\) is a regular function on a
neighborhood of \(b\) defined by a uniformizer of the local ring of
\(\mathcal{O}_{B,b}\).  On the one hand, we have
\({}^{\mathrm{p}}\Phi_{t,1}(i_{\ast}E) = E\).  On the other hand, we
have
\({}^{\mathrm{p}}\Phi_{t,1}(j_{!}\mathcal{E}[1])={}^{\mathrm{p}}\Psi_{t,1}(\mathcal{E}[1])\),
where \({}^{\mathrm{p}}\Psi_{t,1}=R\Psi_{t,1}[-1]\) is the unipotent
nearby cycle functor.  This can be seen from the distinguished
triangle
\begin{equation*}
i^{\ast}[-1] \to {}^{\mathrm{p}}\Psi_{t,1} \xrightarrow{\mathrm{can}} {}^{\mathrm{p}}\Phi_{t,1}
\end{equation*}
and the fact that \(i^{\ast}j_{!} = 0\).  Since
\({}^{\mathrm{p}}\Psi_{t,1}\) is a direct summand of the full nearby
cycle functor, we have
\begin{equation*}
\dim {}^{\mathrm{p}}\Phi_{t,1}(j_{!}\mathcal{E}[1])=\dim{}^{\mathrm{p}}\Psi_{t,1}(\mathcal{E}[1]) \leq \operatorname{rank}\mathcal{E}.
\end{equation*}
Finally, since \({}^{\mathrm{p}}\Phi_{t,1}\) is a t-exact functor,
\begin{equation*}
E \to {}^{\mathrm{p}}\Phi_{t,1}(j_{!}\mathcal{E}[1])
\end{equation*}
remains injective.  This concludes the proof of the claim and finishes
the proof of the lemma.
\end{proof}

To state an arithmetic variant of the lemma incorporating \(p\)-adic
information, let us assume \(k\) is an algebraic closure of
\(\mathbb{F}_{q}\), where \(q\) is a power of a prime \(p\),
\(B = B_{0} \otimes_{\mathbb{F}_{q}} k\),
\(V = V_{0} \otimes_{\mathbb{F}_{q}}k\), and \(\mathcal{F}\) has a
structure of a Weil sheaf on \(V_{0}\).  Thus
\(B^{\prime} = B_{0}^{\prime} \otimes_{\mathbb{F}_{q}} k\) for some
Zariski open dense subset \(B_{0}^{\prime}\) of \(B_{0}\), and
\(\Sigma_{0} = U_{0}\setminus B_{0}^{\prime}\) defines \(\Sigma\).
Let
\(\mathcal{E}_{0}[1]=\pH^{-1}(R\pi_{!}\mathcal{F}_{0})|_{B^{\prime}_{0}}\).
Then \(\mathcal{E}_{0}\) is a local system on \(B_{0}^{\prime}\).  Let
\(\mathcal{E}\) be the local system on \(B^{\prime}\) induced by
\(\mathcal{E}_{0}\).

\begin{variant}[Upper semicontinuity of the middle Newton polygon]\label{variant:newton}
In the situation above, let \(b \in \Sigma\).  Suppose that \(\Gamma\)
is a convex polygon in \(\mathbb{R}^{2}\) such that for any
\(b^{\prime} \in B^{\prime}\),
\(\mathrm{NP}(\mathcal{E}_{b^{\prime}})\) lies on or above \(\Gamma\).
Then for any \(b \in \Sigma\), the Newton polygon of
\(\mathrm{H}^{-1}_{c}(\pi^{-1}(b);\mathcal{F}|_{\pi^{-1}(b)})\) lies
on or above \(\Gamma\).
\end{variant}

\begin{proof}
As we have seen in the proof of Lemma~\ref{lemma:middle-estimate},
especially in \eqref{eq:what-is-e}, for \(b \in \Sigma\), the
cohomology space
\(\mathrm{H}^{-1}_{c}(\pi^{-1}(b);\mathcal{F}|_{\pi^{-1}(b)})\)
corresponds to the vector space \(E\) in Claim~\ref{claim:middle}.  It
injects into the unipotent nearby cycle of \(\mathcal{E}\), which is a
direct summand of the full nearby cycle
\(\mathcal{E}_{\overline{\eta}_{b}}\), where \(\overline{\eta}_{b}\)
is a geometric generic point of the strict localization at \(b\).  Now
we can apply
Proposition~\ref{proposition:weak-grothendieck-specialization}.
\end{proof}

We will also need the following variant of
Lemma~\ref{lemma:middle-estimate} in mixed characteristic context.  We
state it in terms of torsion coefficients since the needed
prerequisites are only documented for torsion coefficients, and we
will only refer to this torsion version.

\begin{variant}[Mixed characteristic case]\label{variant:middle-estimate}
Let \(S\) be a strictly henselian trait with special point \(s\) and
generic point \(\eta\).  Let \(\overline{\eta}\) be a geometric point
lying above \(\eta\).  Choose a prime number \(\ell\) invertible on
\(S\).  Suppose \(V\) is a flat, local complete intersection
\(S\)-scheme of relative dimension \(n\) everywhere.  Then
\begin{equation*}
\dim \mathrm{H}^{n}_{c}(V_{s};\mathbb{F}_{\ell}) \leq
\dim \mathrm{H}^{n}_{c}(V_{\overline{\eta}};\mathbb{F}_{\ell}).
\end{equation*}
\end{variant}

\begin{proof}
By \cite[\S2]{illusie:perversity-and-variation}, a selfdual perverse
t-structure can be defined on the constructible category of étale
\(\mathbb{F}_{\ell}\)-sheaves on \(V\) (which is amenable to the four
functor formalism, see \cite[Finitude]{deligne:sga4.5}).  The flat and
local complete intersection hypotheses imply that the shifted constant
sheaf \(\mathcal{F}=\mathbb{F}_{\ell,V}[n+1]\) is a perverse sheaf on
\(V\) (\cite[Corollaire~2.7]{illusie:perversity-and-variation}).  Let
\(\pi\colon V \to S\) be the structure morphism.  By Gabber's Artin
vanishing theorem
(\cite[Théorème~2.4]{illusie:perversity-and-variation}),
\(R\pi_{!}\mathcal{F} = R\pi_{!}(\mathbb{F}_{\ell,V}[n+1])\) has no
negative perverse cohomology sheaves.  With these ingredients, the
argument of Lemma~\ref{lemma:middle-estimate} goes through.
\end{proof}

\section{Betti numbers of set-theoretic complete intersections}
\label{sec:betti}
\numberwithin{equation}{subsection}

In this section, \(k\) is an algebraically closed field.  Our
objective is to estimate the compactly supported Betti numbers of
set-theoretic complete intersections in ambient spaces via
Lemma~\ref{lemma:middle-estimate}.  After establishing a slightly more
general framework (\S\ref{sec:general-upper-bounds-betti}), we
specialize to set-theoretic complete intersections in an affine space
(\S\ref{sec:complete-intersection-in-affine}), then to those in a
projective space (\S\ref{sec:proj-complete-intersection}).  Finally,
we apply these estimates to derive effective Lang--Weil type bounds.
\subsection{General upper bounds}
\label{sec:general-upper-bounds-betti}
A convenient framework, which is general enough to include many
interesting special cases, is as follows.

\begin{construction}\label{construction:betti}
Consider the following data:
\begin{itemize}
\item \(X\) is a proper variety.
\item \(D \subset X\) is an effective Cartier divisor.
\item \(U = X \setminus D\).  We assume that \(U\) is an affine
nonsingular variety.
\item \(\mathcal{L}_{1},\ldots,\mathcal{L}_{r}\) are invertible
sheaves on \(X\).  We assume that
\(\dim \mathrm{H}^{0}(X; \mathcal{L}_{i}) \geq 1\).
\end{itemize}

Each point \((F_{1},\ldots,F_{r})\) of the space
\(\mathcal{M}=\prod_{i=1}^{r}\mathbb{P}\mathrm{H}^{0}(X;\mathcal{L}_{i})\)
determines a closed subvariety \(\{F_{1} = \cdots = F_{r} = 0\}\) of
\(X\).  We denote the intersection
\(U \cap \{F_{1} = \cdots = F_{r} = 0\}\) by
\(V^{\circ}(F_{1},\ldots, F_{r})\).

We have an incidence correspondence
\(\mathcal{V} \subset U \times \mathcal{M}\): the fiber of
\(\mathcal{V} \to \mathcal{M}\) over
\((F_{1},\ldots,F_{r}) \in \mathcal{M}\) is
\(V^{\circ}(F_{1},\ldots,F_{r})\).  Denote the projection morphism by
\(\varpi\colon \mathcal{V} \to \mathcal{M}\).

Due to the constructibility of
\(R^{e}\varpi_{!}\overline{\mathbb{Q}}_{\ell}\), there is a Zariski
open dense subset \(\mathcal{U}\) of \(\mathcal{M}\) such that for any
\(e\), \(R^{e}\varpi_{!}\overline{\mathbb{Q}}_{\ell}|_{\mathcal{U}}\)
is a local system on \(\mathcal{U}\).  By the proper base change
theorem, if \(e\) is fixed and \((F_{1},\ldots,F_{r})\) varies in
\(\mathcal{U}\), the number
\begin{equation*}
\dim \mathrm{H}_{c}^{e}(V^{\circ}(F_{1},\ldots,F_{r});\overline{\mathbb{Q}}_{\ell})
\end{equation*}
remains constant.  We denote this constant by
\begin{equation}
B^{e}_{c}(X,D;\mathcal{L}_{1},\ldots,\mathcal{L}_{r})
\end{equation}
or simply
\(B_{c}^{e}\) when there is no ambiguity.
\end{construction}

\begin{proposition}\label{proposition:abstract-betti-bound}
In Construction~\ref{construction:betti}, let \((F_{1},\ldots,F_{r})\)
be an arbitrary \(r\)-tuple in \(\mathcal{M}\).  Then we have the
following inequality:
\begin{equation*}
\dim \mathrm{H}^{n-r}_{c}(V^{\circ}(F_{1},\ldots,F_{r});\overline{\mathbb{Q}}_{\ell}) \leq B_{c}^{n-r}(X,D;\mathcal{L}_{1},\ldots,\mathcal{L}_{r}).
\end{equation*}
\end{proposition}

\begin{proof}
Let \((F_{1},\ldots,F_{r}) \in \mathcal{M}\) be an arbitrary
\(r\)-tuple, and let \((F^{\prime}_{1},\ldots,F^{\prime}_{r})\) be an
\(r\)-tuple contained in \(\mathcal{U}\).  Consider the morphism
\begin{equation*}
u\colon \mathbb{A}^{1} \to \mathcal{M}, \quad
t\mapsto (1-t)\cdot (F_{1},\ldots,F_{r}) + t \cdot (F^{\prime}_{1},\ldots,F_{r}^{\prime}).
\end{equation*}
Form the fiber product
\(V = \mathcal{V} \times_{\mathcal{M},u}\mathbb{A}^{1}\), and let
\(\pi\) be the projection to \(\mathbb{A}^{1}\).  Since
\(u(1) \in \mathcal{U}\) by construction,
\(u(\mathbb{A}^{1}) \cap \mathcal{U}\) is nonempty.
Thus \(u^{-1}\mathcal{U} = B^{\prime}\) is an open dense subset of \(\mathbb{A}^{1}\).
By the proper base change theorem, for any \(b^{\prime} \in B^{\prime}\), we have
\begin{equation}\label{eq:generic-betti}
\mathrm{H}^{n-r}_{c}(\pi^{-1}(b^{\prime});\overline{\mathbb{Q}}_{\ell}) = B^{n-r}_{c}(X,D;\mathcal{L}_{1},\ldots,\mathcal{L}_{r}).
\end{equation}
Since \(V\) is defined in \(U \times \mathbb{A}^{1}\) by \(r\)
equations \((1-t)F_{1}+tF_{1}^{\prime},\ldots,(1-t)F_{r}+tF^{\prime}_{r}\), and
since \(U \times \mathbb{A}^{1}\) is nonsingular, we see that
\(\overline{\mathbb{Q}}_{\ell,V}[n+1-r]\) satisfies the cosupport condition (Lemma~\ref{lemma:cosupp}).
Applying Lemma~\ref{lemma:middle-estimate} to
\(\pi\colon V \to \mathbb{A}^{1}\) and combining it with \eqref{eq:generic-betti}, we get
\begin{equation*}
\dim \mathrm{H}^{n-r}_{c}(V^{\circ}(F_{1},\ldots,F_{r})) \leq B^{n-r}_{c}(X,D;\mathcal{L}_{1},\ldots,\mathcal{L}_{r}),
\end{equation*}
as desired.
\end{proof}

\subsection{Set-theoretic complete intersections in the affine space}
\label{sec:complete-intersection-in-affine}
Let \(n \geq r \geq 1\) and \(d_{1},\ldots,d_{r}\) be natural numbers.
We define
\begin{equation*}
 N(n;d_{1},\ldots,d_{r}) =
d_{1}\cdots d_{r} \cdot \sum\limits_{e=0}^{n-r} (-1)^{e}\binom{n}{e} \sum\limits_{a_{1}+\cdots +a_{r}=n-r-e} d_{1}^{a_{1}}\cdots d_{r}^{a_{r}},
\end{equation*}
and
\begin{equation*}
M(n;d_{1},\ldots,d_{r}) =
\begin{cases}
  N(n;d_1,\cdots, d_r)-(-1)^{n-r} & \text{if }n>r,\\
  N(n;d_1,\cdots, d_r)=d_{1}\cdots d_{r},  & \text{if }n=r.
\end{cases}
\end{equation*}
In this subsection, we prove the following:

\begin{theorem}\label{theorem:affine-betti-bound-detail}
Consider the affine scheme
\(V = \operatorname{Spec}k[x_{1},\ldots,x_{n}]/(f_{1},\ldots,f_{r})\),
where \(f_{1},\ldots,f_{r}\) satisfy \(\deg f_{i} \leq d_{i}\).  Then
\begin{enumerate}
\item We have
\begin{equation*}
\dim \mathrm{H}^{n-r}_{c}(V;\overline{\mathbb{Q}}_{\ell})
\leq M(n;d_{1},\ldots,d_{r}).
\end{equation*}
\item If, in addition, \(\dim V = n-r\), then for any \(0 \leq j\leq n-r\), we have
\begin{equation*}
\dim \mathrm{H}^{n-r+j}_{c}(V;\overline{\mathbb{Q}}_{\ell}) \leq M(n-j;d_{1},\ldots,d_{r}).
\end{equation*}
Thus, \(B_{c}(V, \ell) \leq \sum_{j=0}^{n-r} M(n-j;d_{1},\ldots,d_{r})\).
\end{enumerate}
\end{theorem}

\begin{proof}
Let us first prove (i).  We apply
Construction~\ref{construction:betti} to the following situation:
\begin{itemize}
\item Let \(X = \mathbb{P}^{n}\) with homogeneous coordinates \([x_{0},\ldots,x_{n}]\).
\item Let \(D = \{x_{0} = 0\} \cong \mathbb{P}^{n-1}\) be the hyperplane at infinity.
\item Let \(U = X \setminus D = \mathbb{A}^{n}\) with affine coordinates \((x_{1},\ldots,x_{n})\).
\item Let \(\mathcal{L}_{i} = \mathcal{O}_{\mathbb{P}^{n}}(d_{i})\).
\end{itemize}
The open set \(\mathcal{U}\) in Construction~\ref{construction:betti}
can be taken to be the set of \(r\)-tuples \((F_{1},\ldots,F_{r})\),
where each \(F_{i}\) is homogeneous of degree \(d_{i}\), such that
\begin{itemize}
\item \(\{F_{1} = \cdots = F_{r}\} \subset \mathbb{P}^{n}\) is
smooth of dimension \(n-r\).
\item
\(D \cap \{F_{1} = \cdots = F_{r}\} \subset D \cong
\mathbb{P}^{n-1}\) is smooth of dimension \(n-r-1\).
\end{itemize}
The following lemma computes the number
\(B_{c}^{e}(\mathbb{P}^{n},\mathbb{P}^{n-1};\mathcal{O}_{\mathbb{P}^{n}}(d_{1}),\ldots,\mathcal{O}_{\mathbb{P}^{n}}(d_{r}))\).

\begin{lemma}\label{lemma:exact-betti-for-generic-affine-complete-intersection}
Let \((F_{1},\ldots,F_{r})\) be an \(r\)-tuple in the open set
\(\mathcal{U}\) defined above.  Then the following holds:
\begin{enumerate}
\item We have
\[
\dim \mathrm{H}^{n-r}_{c}(V^{\circ}(F_{1},\ldots,F_{r});\overline{\mathbb{Q}}_{\ell}) = M(n;d_{1},\ldots,d_{r}).
\]
\item If \(n-r\geq 1\), then
\(\dim \mathrm{H}^{2n-2r}_{c}(V^{\circ}(F_{1},\ldots,F_{r})) = 1\).
\item If \(e \notin \{n-r, 2n-2r\}\), then
\(\mathrm{H}^{e}_{c}(V^{\circ}(F_{1},\ldots,F_{r})) = 0\).
\item In the case $d_1=\cdots =d_r =d$, we have
\begin{equation*}
  \binom{n-1}{r-1}(d-1)^n \leq  M(n; d, \cdots, d) \leq \binom{n-1}{r-1}d^n.
\end{equation*}
\end{enumerate}
\end{lemma}

The proof of this lemma is a standard exercise, except perhaps the last part, and we have placed it at
the end of the section for the reader's convenience.

The first
inequality in Theorem \ref{theorem:affine-betti-bound-detail} then follows from
Proposition~\ref{proposition:abstract-betti-bound}.
To prove (ii) of Theorem \ref{theorem:affine-betti-bound-detail}, we apply Lemma~\ref{lemma:gysin-perv} with
\(\mathcal{F} = \mathbb{Q}_{\ell,V}[n-r]\), \(X = V\), and let \(f\)
be the inclusion
\(V \hookrightarrow \mathbb{A}^{n} \hookrightarrow \mathbb{P}^{n}\).
Note that the hypothesis \(\dim V = n-r\) ensures that the shifted
constant sheaf \(\mathbb{Q}_{\ell,V}[n-r]\) is an object of
\({}^{\mathrm{p}}\mathcal{D}^{\leq0}_{V}\).

Thus, for a general hyperplane \(A \subset \mathbb{A}^{n}\), we have
\[
\dim \mathrm{H}^{n-r+1}_{c}(V;\overline{\mathbb{Q}}_{\ell}) \leq \dim
\mathrm{H}^{n-r-1}_{c}(A \cap V; \overline{\mathbb{Q}}_{\ell}).
\]
Since \(A\) is chosen to be general, \(A \cap V\) is a closed
subscheme of an \((n-1)\)-dimensional affine space, defined by \(r\)
polynomials \(f_{1}|_{A},\ldots,f_{r}|_{A}\), with
\(\deg f_{i}|_{A} = \deg f_{i} \leq d_{i}\), and
\(\dim A \cap V = n-1-r\).  Thus, by the first inequality, we obtain
\[
\dim \mathrm{H}^{n-r-1}_{c}(A \cap V; \overline{\mathbb{Q}}_{\ell})
\leq M(n-1;d_{1},\ldots,d_{r}).
\]
Repeating this argument for \(A \cap V\) and its higher-codimensional
linear subspace sections, we complete the proof of the second inequality.
\end{proof}

\begin{example}
Consider the case when \(r = 1\), i.e., \(V\) is a hypersurface in
\(\mathbb{A}^{n}\) defined by a polynomial \(f\) of degree \(\leq d\).
In this case, Theorem~\ref{theorem:affine-betti-bound-detail}
specializes to \(\dim \mathrm{H}^{n-1+j}_{c}(V) \leq (d-1)^{n-j}\) for \( 0\leq j\leq n-2\) and \(\dim \mathrm{H}^{n-1+n-1}_{c}(V) \leq d\).
In particular,
\begin{equation*}
B_{c}(V, \ell) = \sum_{j=0}^{n-1} \dim \mathrm{H}^{n-1+j}_{c}(V)\leq d+ \sum_{j=2}^{n} (d-1)^{j}\leq  d^n.
\end{equation*}
\end{example}

For $r\geq 2$, the formula for \(M(n;d_{1},\ldots,d_{r})\) is a bit messy when the integers $d_i$'s are not all the same.  In the case $d_1=\cdots d_r=d$, much simpler-looking upper and lower bounds for
\[
M(n;d,\ldots,d)
\]
are given in Lemma \ref{lemma:exact-betti-for-generic-affine-complete-intersection}.  As a consequence, we deduce.

\begin{corollary}\label{corollary:middle-cleaner-bound}
Let \(1\leq r\leq n\) be integers.  Suppose
\(f_{1},\ldots, f_{r} \in k[x_{1},\ldots,x_{n}]\) are polynomials
satisfying \(\deg f_{i} \leq d\).  Let
\(V = \operatorname{Spec}k[x_{1},\ldots,x_{n}]/(f_{1},\ldots,f_{r})\).
Then:
\begin{enumerate}
\item We have
\begin{equation*}
\dim \mathrm{H}_{c}^{n-r}(V;\overline{\mathbb{Q}}_{\ell}) \leq
\binom{n-1}{r-1}d^{n}.
\end{equation*}
\item Assume that \(\dim V = n-r\).  Then for \(0\leq j \leq n-r\), we
have
\begin{equation*}
\dim \mathrm{H}^{n-r+j}_{c}(V;\mathbb{Q}_{\ell}) \leq \binom{n-j-1}{r-1}d^{n-j}.
\end{equation*}
Therefore, \(B_{c}(V, \ell) \leq \binom{n-1}{r-1}(d+1)^{n}\).
\item For $1\leq r\leq n$, we have
\begin{equation*}
\binom{n-1}{r-1}(d-1)^{n} \leq M(n;d,\cdots, d) \leq B_{c}^{\mathrm{ci}}(n, r; d))\leq \binom{n-1}{r-1}(d+1)^{n}.
\end{equation*}
\label{item:complete-intersection-total}
\end{enumerate}
\end{corollary}


\begin{proof}[Proof of Lemma~\ref{lemma:exact-betti-for-generic-affine-complete-intersection}]
If \(n = r\), then the result follows directly from Bézout's theorem.

We now assume \(n > r\), that is, $n-r>0$.  Recall the description of \(\mathcal{U}\) in the
present situation.  Let \((F_{1},\ldots,F_{r}) \in \mathcal{U}\),
\(Z = \{F_{1} = \cdots = F_{r} = 0\} \subset \mathbb{P}^{n}\),
\(Z^{\prime} = Z \cap D\), and
\(V = V^{\circ}(F_{1},\ldots,F_{r}) \subset \mathbb{A}^{n}\).  Then we
know that both \(Z\) and \(V\) are connected, and we have
\((-1)^{n-r}\chi(V) = (-1)^{n-r}(\chi(Z) - \chi(Z^{\prime}))\), where
\(\chi(Y)\) denotes the Euler characteristic of \(Y\):
\begin{equation*}
\chi(Y) = \sum_{e} (-1)^{e} \dim \mathrm{H}^{e}_{c}(Y;\overline{\mathbb{Q}}_{\ell}).
\end{equation*}

We claim that \(\mathrm{H}^{e}_{c}(V) = 0\) unless \(e = n-r\) or \(e = 2(n-r)\).
If this assertion holds, then
\begin{align*}
  (-1)^{n-r}\chi(V)
  &= \dim \mathrm{H}^{n-r}_{c}(V;\overline{\mathbb{Q}}_{\ell}) + (-1)^{n-r}\dim \mathrm{H}^{2(n-r)}_{c}(V;\overline{\mathbb{Q}}_{\ell}) \\
  &= \dim \mathrm{H}^{n-r}_{c}(V;\overline{\mathbb{Q}}_{\ell}) + (-1)^{n-r},
\end{align*}
since \(V\) is connected.  Therefore, once the assertion is confirmed,
we need only compute the Euler characteristics to prove the lemma.

\medskip%
Let us prove the claim.  If \(e < n-r\), the claim follows from the
Artin vanishing theorem since \(V\) is smooth.  If \(n-r = 1\), the
claim is evident.  Now, assume \(n-r \geq 2\) and \(e > n-r\).  In
this case, both \(Z\) and \(Z^{\prime}\) are smooth, connected
complete intersections in projective spaces.  For \(\epsilon \geq 1\),
the restriction map
\[
\mathrm{H}^{n-r-1+\epsilon}(Z;\overline{\mathbb{Q}}_{\ell}) \to \mathrm{H}^{n-r-1+\epsilon}(Z^{\prime};\overline{\mathbb{Q}}_{\ell})
\]
fits into the following commutative diagram:
\begin{equation*}
\begin{tikzcd}
\mathrm{H}^{n-r-1-\epsilon}(Z;\overline{\mathbb{Q}}_{\ell}) \ar[sloped]{d}{\sim}
\ar[r,"{\smile h^{\epsilon}}"] & \mathrm{H}^{n-r-1+\epsilon}(Z;\overline{\mathbb{Q}}_{\ell})\ar{d} \\
\mathrm{H}^{n-r-1-\epsilon}(Z^{\prime};\overline{\mathbb{Q}}_{\ell}) \ar[r,,"\sim"',"{\smile h^{\epsilon}}"] & \mathrm{H}^{n-r-1+\epsilon}(Z^{\prime};\overline{\mathbb{Q}}_{\ell})
\end{tikzcd}
\end{equation*}

In the diagram, \(h\) denotes the Chern class of the invertible sheaf
\(\mathcal{O}_{\mathbb{P}^{n}}(1)\).  By the Hard Lefschetz theorem,
the lower horizontal arrow is an isomorphism.  The weak Lefschetz
theorem implies that the left vertical arrow is also an isomorphism.
Therefore, the right vertical arrow must be surjective.  Consequently,
for \(e \geq n-r\), the restriction maps
\begin{equation*}
\mathrm{H}^{e}(Z;\overline{\mathbb{Q}}_{\ell}) \to \mathrm{H}^{e}(Z^{\prime};\overline{\mathbb{Q}}_{\ell})
\end{equation*}
are all surjective.  Hence we obtain short exact sequences
\begin{equation*}
0 \to \mathrm{H}^{e}_{c}(V;\overline{\mathbb{Q}}_{\ell}) \to \mathrm{H}^{e}(Z;\overline{\mathbb{Q}}_{\ell})
\to \mathrm{H}^{e}(Z^{\prime};\overline{\mathbb{Q}}_{\ell}) \to 0.
\end{equation*}

It is well-known that for a smooth complete intersection \(Z\) in
\(\mathbb{P}^{n}\), and for any integer \(e\) such that
\(e \neq \dim Z\) and \(0 \leq e \leq 2\dim Z\), the cohomology group
\(\mathrm{H}^{e}(Z;\overline{\mathbb{Q}}_{\ell})\) is one-dimensional
if \(e\) is even, and zero if \(e\) is odd.  Thus, the exact sequence
implies that for \(e > n-r\), we have
\(\mathrm{H}_{c}^{e}(V;\overline{\mathbb{Q}}_{\ell}) = 0\), except
when \(e = 2(n-r)\), in which case it is one-dimensional.  This
completes the proof of the claim.

Now we turn to computing the Euler characteristic numbers.
By a Chern class argument, we find that
\begin{equation*}
\dim \mathrm{H}^{n-r}_{c}(V;\overline{\mathbb{Q}}_{\ell}) + (-1)^{n-r}=(-1)^{n-r}\chi(V) = (-1)^{n-r}[\chi(Z) - \chi(Z^{\prime})]
\end{equation*}
is \((-1)^{n-r}d_{1}\cdots d_{r}\) times the coefficient of
\(h^{n-r}\) in the power series expansion of
\begin{equation*}
\frac{(1+h)^{n+1}}{(1+d_{1}h)\cdots (1+d_{r}h)} - \frac{h(1+h)^{n}}{(1+d_{1}h)\cdots (1+d_{r}h)} = \frac{(1+h)^{n}}{(1+d_{1}h)\cdots (1+d_{r}h)}.
\end{equation*}
Equivalently, $\dim \mathrm{H}^{n-r}_{c}(V;\overline{\mathbb{Q}}_{\ell}) + (-1)^{n-r}$ is the coefficient of $h^{n-r}$ in the power
series expansion of
\begin{equation*}
  \frac{d_1\cdots d_r(1-h)^{n}}{(1-d_{1}h)\cdots (1-d_{r}h)}.
\end{equation*}
This coefficient is exactly $N(n;d_1,\cdots, d_r)=M(n;d_{1},\ldots,d_{r}) +(-1)^{n-r}$.
It follows that
\begin{equation*}
\dim \mathrm{H}^{n-r}_{c}(V;\overline{\mathbb{Q}}_{\ell})
=M(n;d_{1},\ldots,d_{r}).
\end{equation*}

In the special case $d_1=\cdots = d_r=d$,
$M(n;d,\ldots,d)+(-1)^{n-r}$
is the coefficient of $h^{n-r}$ in the power series expansion of
\begin{equation*}
  \frac{d^r(1-h)^{n}}{(1-dh)^r} =\frac{1}{d^{n-r}}\frac{((d-1) +(1-dh))^{n}}{(1-dh)^r} =\frac{1}{d^{n-r}}\sum_{i=0}^n \binom{n}i(d-1)^{n-i}(1-dh)^{i-r}.
\end{equation*}
This coefficient is given by
\begin{equation*}
\sum_{i=0}^n \binom{n}i(d-1)^{n-i}(-1)^{n-r}\binom{i-r}{n-r}
=\sum_{i=0}^{r-1} \binom{n}{i}(d-1)^{n-i}\binom{n-i-1}{n-r} +(-1)^{n-r}.
\end{equation*}
It follows that
\begin{equation*}
M(n;d,\ldots,d) = \sum_{i=0}^{r-1} \binom{n}{i}(d-1)^{n-i}\binom{n-i-1}{r-i-1}
\geq \binom{n-1}{r-1}(d-1)^n.
\end{equation*}
This proves the lower bound for $M(n;d,\ldots,d)$. For the upper bound, one checks that
\begin{equation*}
M(n;d,\ldots,d)
\leq \binom{n-1}{n-r}\sum_{i=0}^{r-1} \binom{n}{i} (d-1)^{n-i}
\leq \binom{n-1}{r-1}d^n.
\end{equation*}
The upper bound for $M(n;d,\ldots,d)$ is proved.
\end{proof}

\subsection{Projective complete intersections}
\label{sec:proj-complete-intersection}

Using Theorem~\ref{theorem:affine-betti-bound-detail}, we can obtain
bounds for Betti numbers for a complete intersection in a projective
space as well.

\begin{proposition}
\label{proposition:projective-bound}
Let \(F_{1},\ldots,F_{r} \in k[x_{0},\ldots,x_{n}]\) be homogeneous
polynomials.  Let
\(X = \{F_{1} = \cdots = F_{r}=0\}\subset \mathbb{P}^{n}\).  Assume
that
\begin{itemize}
\item \(\deg F_{i} \leq d\), and
\item \(\dim X = n-r \geq 1\).
\end{itemize}
Then we have
\begin{enumerate}
\item \label{item1}
\(\dim \mathrm{H}^{j}(X;\overline{\mathbb{Q}}_{\ell}) = \dim \mathrm{H}^{j}(\mathbb{P}^{n};\overline{\mathbb{Q}}_{\ell})\)
for \(j=0,1,\ldots,n-r-1\), and
\item \label{item2} \(\dim \mathrm{H}^{n-r+j}(X;\overline{\mathbb{Q}}_{\ell}) \leq  \binom{n-1-j}{r-1}[d^{n-j} + d^{n-j-1} + \cdots + d^{r}]\),
for \(0 \leq j \leq n-r\).
\end{enumerate}
\end{proposition}

\begin{proof}
Applying Lemma~\ref{lemma:weak-lefschetz} to the pairs
\((Y,D) = (\mathbb{P}^{n},\{F_{1} = 0\})\),
\((\{F_{1}=0\}, \{F_{1}=F_{2}=0\})\), \(\ldots\), and eventually to
\((\{F_{1}=\cdots=F_{r-1}=0\},X)\), we get (\ref{item1}).

We prove (\ref{item2}) by induction on the dimension of \(X\).  The
base case occurs when \(X\) is a zero-dimensional complete
intersection, i.e., when \(n = r\).  In this case, we have
\(\mathrm{H}^{0}(X;\overline{\mathbb{Q}}_{\ell}) \leq d^{n}\), which
is stronger than (\ref{item2}).

We now establish the induction step and assume $n>r$.  Let \(A\) be a sufficiently
general hyperplane in \(\mathbb{P}^{n}\).  Then \(X \cap A\) is a
complete intersection in the projective space \(A\) of one dimension
lower.  The complement \(V = X \setminus (X \cap A)\) is a complete
intersection in the affine space
\(\mathbb{A}^{n} = \mathbb{P}^{n} \setminus A\).  We obtain the
following exact sequence:
\begin{equation}
\mathrm{H}^{m-1}(X \cap A;\overline{\mathbb{Q}}_{\ell}) \to \mathrm{H}^{m}_{c}(V;\overline{\mathbb{Q}}_{\ell}) \to
\mathrm{H}^{m}(X;\overline{\mathbb{Q}}_{\ell}) \to \mathrm{H}^{m}(X\cap A;\overline{\mathbb{Q}}_{\ell}).
\end{equation}
It follows from Corollary~\ref{corollary:middle-cleaner-bound} that $\dim \mathrm{H}^{n-r+j}(X;\overline{\mathbb{Q}}_{\ell})$ is
\begin{align*}
  &\leq \dim \mathrm{H}^{n-r+j}_{c}(V;\overline{\mathbb{Q}}_{\ell}) + \dim \mathrm{H}^{(n-1-r)+(j+1)}(X \cap A;\overline{\mathbb{Q}}_{\ell}) \\
  &\leq \binom{n-1-j}{r-1}d^{n-j} + \binom{n-j-3}{r-1}\left\{d^{n-1-(j+1)} + \cdots + d^{r}\right\} \\
  &\leq  \binom{n-1-j}{r-1} \left\{ d^{n-j} + d^{n-j-1} + \cdots + d^{r} \right\}.
\end{align*}
This completes the induction.
\end{proof}

\begin{remark*}
Maxim, P{\u a}nescu, and Tib{\u a}r \cite[Corollary~1.5]{maxim-paunescu-tibar:vanishing-cohomology-and-betti-bounds-for-projective-hypersurfaces}
showed that for any complex projective hypersurface \(X\) in
\(\mathbb{P}^{n}\) of degree \(d\), we have
\begin{equation*}
\dim \mathrm{H}^{n-1}(X;\mathbb{Q}) \leq \frac{(d-1)^{n+1} + (-1)^{n}}{d} + \frac{3(-1)^{n-1}+1}{2},
\end{equation*}
which provides a more accurate bound than
Proposition~\ref{proposition:projective-bound}(\ref{item2}) when $j=0$, $r=1$.
\end{remark*}

\subsection{Application: Lang--Weil type estimates}
The bounds for the compactly supported Betti numbers immediately imply
some effective Lang--Weil type estimates.  We begin with the affine
version.

\begin{proposition}\label{proposition:lang-weil-affine-complete-intersection}
Let \(V\) be a geometrically irreducible closed subvariety of
\(\mathbb{A}^{n}_{\mathbb{F}_{q}}\) defined by \(r\) polynomials of
degree \(\leq d\).  Assume that \(\dim V = n-r\).  Then
\begin{equation*}
|\operatorname{Card} V(\mathbb{F}_{q})  - q^{n-r}| \leq \sum_{j=0}^{n-r-1} \binom{n-j-1}{r-1} d^{n-j} q^{\frac{n-r+j}{2}} < \binom{n-1}{r-1}(d+1)^{n}q^{n-r-\frac{1}{2}}.
\end{equation*}
\end{proposition}

\begin{proof}
Fix an algebraic closure \(\overline{\mathbb{F}}_{q}\) of
\(\mathbb{F}_{q}\).  Let
\(V_{\overline{\mathbb{F}}_{q}} = V \otimes_{\mathbb{F}_{q}} \overline{\mathbb{F}}_{q}\).  By
Grothendieck's trace formula, we have
\begin{equation*}
\operatorname{Card}V(\mathbb{F}_{q}) =
\sum_{j=0}^{n-r} (-1)^{j} \mathrm{Tr}(F|\mathrm{H}^{n-r+j}_{c}(V_{\overline{\mathbb{F}}_{q}};\overline{\mathbb{Q}}_{\ell})).
\end{equation*}
Since \(V\) is assumed to be geometrically irreducible,
\(\mathrm{H}^{2n-2r}_{c}(V_{\overline{\mathbb{F}}_{q}};\overline{\mathbb{Q}}_{\ell})\simeq\overline{\mathbb{Q}}_{\ell}(-n+r)\).
Therefore, the trace formula can be written as
\begin{equation*}
\operatorname{Card}V(\mathbb{F}_{q})-q^{n-r} =
\sum_{j=0}^{n-r-1}(-1)^{j}\mathrm{Tr}(F|\mathrm{H}^{n-r+j}_{c}(V_{\overline{\mathbb{F}}_{q}};\overline{\mathbb{Q}}_{\ell})).
\end{equation*}
By Deligne's fundamental theorem of weights, the Frobenius eigenvalues
of \(\mathrm{H}^{n-r+j}_{c}(V_{\overline{\mathbb{F}}_{q}};\overline{\mathbb{Q}}_{\ell})\) are Weil
\(q\)-numbers of weight \(\leq n-r+j\).  Therefore,
\begin{equation*}
|\operatorname{Card}V(\mathbb{F}_{q})-q^{n-r}| \leq
\sum_{j=0}^{n-r-1}|\mathrm{Tr}(F|\mathrm{H}^{n-r+j}_{c}(V_{\overline{\mathbb{F}}_{q}};\overline{\mathbb{Q}}_{\ell}))|.
\end{equation*}
By Corollary~\ref{corollary:middle-cleaner-bound}, for each \(j\),
\(F|_{\mathrm{H}^{n-r+j}_{c}(V;\overline{\mathbb{Q}}_{\ell})}\) has at
most \(\binom{n-j-1}{r-1}d^{n-j}\) eigenvalues.  The desired estimates then
follow.
\end{proof}

The estimates for the Betti numbers of a projective complete
intersection give the following more refined effective version of the classical Lang--Weil estimate
\cite{lang-weil:number-of-points-of-varieties-in-finite-fields}.

\begin{proposition}%
Let \(q\) be a power of \(p\).  Let \(X\) be the Zariski closed subset
in \(\mathbb{P}^{n}_{\mathbb{F}_{q}}\) defined by the common zeros of
polynomials
\(F_{1},\ldots,F_{r} \in \mathbb{F}_{q}[x_{1},\ldots,x_{n}]\).
Suppose that \(\dim X = n - r\geq 2\), \(X\) is geometrically
irreducible, and \(\deg F_{i} \leq d\).  Let \(\epsilon\) be the
dimension of the singular locus of \(X\).  Then
\begin{equation*}
\left| \operatorname{Card}X(\mathbb{F}_{q}) - \sum_{j=0}^{n-r}q^{j} \right|
<\binom{n-1}{r-1}(d+2)^{n}q^{\frac{n-r+\epsilon+1}{2}}.
\end{equation*}
\end{proposition}

If \(X\) is nonsingular, we make the convention that
\(\epsilon = -1\).

\begin{proof}
Since we assumed that \(\dim X \geq 2\), the complete intersection
condition ensures that \(X\) is normal, so that
\(\epsilon \leq n-r-2\).

Let \(A\) be a generic \((n-\epsilon-1)\)-dimensional linear
subspace of \(\mathbb{A}_{\overline{\mathbb{F}}_{q}}^{n}\).  By Lemma~\ref{lemma:gysin-perv},
applied to the perverse sheaf
\(\mathcal{F} = \overline{\mathbb{Q}}_{\ell}[n]\), we find that the
Gysin map
\begin{equation*}
\mathrm{H}^{n-r-\epsilon-1+i}(X_{\overline{\mathbb{F}}_{q}} \cap A;\overline{\mathbb{Q}}_{\ell})(-\epsilon-1)
\to \mathrm{H}^{n-r+\epsilon+1+i}(X_{\overline{\mathbb{F}}_{q}};\overline{\mathbb{Q}}_{\ell})
\end{equation*}
is surjective if \(i = 0\), and bijective if \(i > 0\).  Similarly, by
the weak Lefschetz theorem, for \(i>0\), the restriction map
\begin{equation*}
\mathrm{H}^{n-r-\epsilon-1-i}(X_{\overline{\mathbb{F}}_{q}};\overline{\mathbb{Q}}_{\ell}) \to
\mathrm{H}^{n-r-\epsilon-1-i}(X_{\overline{\mathbb{F}}_{q}} \cap A;\overline{\mathbb{Q}}_{\ell})
\end{equation*}
is injective if \(i=0\), and bijective if \(i>0\).  Thus, if
\(j\in \mathbb{Z}\), \(n-r> |j|>\epsilon+1\), we have
\(\mathrm{H}^{n-r+j}(X_{\overline{\mathbb{F}}_{q}};\overline{\mathbb{Q}}_{\ell})\simeq\overline{\mathbb{Q}}_{\ell}(-\frac{n-r+j}{2})\)
if \(n-r+j\) is even, and zero otherwise.  Combining this with the fact
that
\(\mathrm{H}^{0}(X_{\overline{\mathbb{F}}_{q}};\overline{\mathbb{Q}}_{\ell})=\overline{\mathbb{Q}}_{\ell}\),
that
\(\mathrm{H}^{2n-2r}(X_{\overline{\mathbb{F}}_{q}};\overline{\mathbb{Q}}_{\ell})=\overline{\mathbb{Q}}_{\ell}(-n-r)\),
and Deligne's theorem asserting that
\(\mathrm{H}^{i}(X_{\overline{\mathbb{F}}_{q}};\overline{\mathbb{Q}}_{\ell})\) has weight
\(\leq i\), we find that
\begin{equation*}
\left| \operatorname{Card}X(\mathbb{F}_{q}) - \sum_{j=0}^{n-r}q^{j} \right|
\leq \sum_{0\leq |j| \leq \epsilon+1} q^{\frac{n-r+j}{2}}\dim \mathrm{H}_{\mathrm{prim}}^{n-r+j}(X_{\overline{\mathbb{F}}_{q}};\overline{\mathbb{Q}}_{\ell}).
\end{equation*}
Here,
\(\mathrm{H}^{\ast}_{\mathrm{prim}}(X_{\overline{\mathbb{F}}_{q}};\overline{\mathbb{Q}}_{\ell})\)
is defined as the cokernel of the map
\(\mathrm{H}^{\ast}(\mathbb{P}^{n};\overline{\mathbb{Q}}_{\ell})\to\mathrm{H}^{\ast}(X_{\overline{\mathbb{F}}_{q}};\overline{\mathbb{Q}}_{\ell})\).
Since \(\mathrm{H}^{\ast}_{\mathrm{prim}}\) is a quotient of
\(\mathrm{H}^{\ast}\), we may conclude by applying
Corollary~\ref{proposition:projective-bound}.
\end{proof}

This, combined with the excision sequence,  gives the
following affine version which improves Proposition
\ref{proposition:lang-weil-affine-complete-intersection} if $V$ is not too singular.

\begin{proposition}
Let \(V\) be an affine variety in \(\mathbb{A}^{n}_{\mathbb{F}_{q}}\)
defined by \(r\) polynomials of degree \(\leq d\). Let $V_{\infty}$ be
the infinite part of $V$ in \(\mathbb{P}^{n}_{\mathbb{F}_{q}}\).
Assume that \(\dim V = n-r =\dim V_{\infty} +1 \geq 2\) (thus, both
$V$ and $V_{\infty}$ are complete intersection), and assume that both
$V$ and $V_{\infty}$ are geometrically irreducible. Let $\epsilon$ be
the dimension for the singular locus of the projective complete
intersection $V \cup V_{\infty}$ in
\(\mathbb{P}^{n}_{\mathbb{F}_{q}}\). Then
\begin{equation*}
|\operatorname{Card} V(\mathbb{F}_{q})  - q^{n-r}| \leq  \binom{n-1}{r-1}(d+1)^{n}q^{\frac{n-r+\epsilon+1}{2}}.
\end{equation*}
\end{proposition}

\section{Bounds for exponential sums}
\label{sec:bounds-exp-sums}

In this section, \(k\) is a fixed algebraic closure of a finite field
\(\mathbb{F}_{q}\).  All varieties are defined over \(k\).

We fix a nontrivial additive character
\(\psi\colon \mathbb{F}_{p} \to \mathbb{C}^{\ast}\).  Denote by
\(\mathcal{L}_{\psi}\) the Artin--Schreier sheaf on
\(\mathbb{A}^{1}_{\mathbb{F}_{p}}\) or \(\mathbb{A}^{1}_{k}\).

We will use the specialization lemma
(Lemma~\ref{lemma:middle-estimate}) to investigate the ``Artin--Schreier
cohomology spaces'' \(\mathrm{H}^\ast_c(U;f^\ast \mathcal{L}_\psi)\) associated with exponential sums.  The primary challenge
is that, to derive meaningful bounds using the specialization
lemma, we must maintain numerical control over the Artin--Schreier
cohomology of a \emph{generic} function.  Most of
Section~\ref{sec:general-method} is dedicated to establishing a
framework for determining this generic behavior.  The main result,
presented in Theorem~\ref{theorem:nondegenerate-exponential-sums},
shows that within a large family, the generic
Artin--Schreier cohomology is nontrivial only in the middle degree, thereby
allowing it to be bounded by Euler characteristics.  We also
provide a precise mathematical definition of what is meant by
``large''.  This result is well-known when the generic member is
tamely ramified or under additional hypotheses such as properness or
ampleness.  However, for our applications, including those involving
toric exponential sums, we cannot afford to make such strong
assumptions.  Therefore, we have formulated
Theorem~\ref{theorem:nondegenerate-exponential-sums} as broadly as
possible, which justifies the length of this subsection.

The remaining subsections are dedicated to applications of the
specialization method and the generic bound.  The first application lies
within a framework designed by Katz, discussed in
Section~\ref{sec:katz-framework}, which leads to the proof of
Theorem~\ref{theorem:last}(ii).  The second application
(Section~\ref{sec:tori-exp-sums}) focuses on the toric case, where we
prove the cohomological enhancement
(Theorem~\ref{theorem:ultimate-as}) of the Adolphson--Sperber total
degree bound given in equation \eqref{eq:as-exp-total-degree}.
Finally, we give the proof of Theorem~\ref{theorem:order} in
\S\ref{sec:proof-theorem-order}.

\subsection{General method}\label{sec:general-method}
Our basic setup for studying exponential sums is the following
(cf.~\cite[\S5]{katz-laumon:fourier-transform-upper-bounds-exponential-sums}).
Let \(U\) be an affine variety over \(k\) of pure dimension \(n\).
Suppose we are given a morphism
\begin{equation*}
u \colon U \to \mathbb{A}^{N}, \quad x \mapsto (u_{1}(x),\ldots,u_{N}(x))
\end{equation*}
Then for any \(a = (a_{1},\ldots,a_{N}) \in \check{\mathbb{A}}^{N}\) (the dual affine space), we can
define a regular function
\begin{equation*}
f_{a} \colon U \to \mathbb{A}^{1}, \quad
x\mapsto \sum_{i=1}^{N} a_{i} u_{i}(x).
\end{equation*}
When \(a\) varies in \(\check{\mathbb{A}}^{N}\), we get a family of morphisms
\begin{equation*}
\begin{tikzcd}
U \times \check{\mathbb{A}}^{N} \ar{r}{f} \ar{d}{\mathrm{pr}_{2}} & \mathbb{A}^{1} \\ \check{\mathbb{A}}^{N}
\end{tikzcd}
\end{equation*}

The cohomological counterpart of the variation of the exponential sums
defined by \(f_{a}\) is the \(\ell\)-adic complex
\(R\mathrm{pr}_{2!}(f^{\ast}\mathcal{L}_{\psi})\).  By proper base
change, for any \(b \in \check{\mathbb{A}}^{N}\),
\begin{equation*}
R^{i}\mathrm{pr}_{2!}(f^{\ast}\mathcal{L}_{\psi})_{b} = \mathrm{H}^{i}_{c}(U;f_{b}^{\ast}\mathcal{L}_{\psi}).
\end{equation*}
By constructibility, there is a Zariski open dense subset \(\mathcal{U}\) of
\(\check{\mathbb{A}}^{N}\) over which
\(R^{i}\mathrm{pr}_{2!}(f^{\ast}\mathcal{L}_{\psi})\) are local
systems for all \(i\).

If \(U\) is defined over a finite field \(\mathbb{F}_{q}\), i.e.,
\(U = U_{0} \otimes_{\mathbb{F}_{q}} k\) for some
\(\mathbb{F}_{q}\)-scheme \(U_{0}\), then
\(b\colon \operatorname{Spec}k \to U \to U_{0}\) factors through some
\(\mathbb{F}_{q^{m}}\)-valued point of \(U_{0}\) for some \(m\):
\begin{equation*}
b \colon \operatorname{Spec}k \to \operatorname{Spec}(\mathbb{F}_{q^{m}}) \xrightarrow{b_{0}} U_{0}.
\end{equation*}
We can then discuss the \(q^{m}\)-power geometric Frobenius operation on
both sides.  The Newton polygon of
\(\det(1-tF_{b_{0}}|\mathrm{H}^{i}_{c}(U;f_{b}^{\ast}\mathcal{L}_{\psi}))\)
with respect to \(\mathrm{ord}_{q^{m}}\) is independent of the
possible \(\mathbb{F}_{q^{m}}\)-point \(b_{0}\) through which \(b\)
factors.

\begin{theorem}\label{theorem:middle-upper-bound-exponential-sum}
Suppose that \(U\) is a purely \(n\)-dimensional local complete
intersection affine variety.  If
\(a \in \mathcal{U}\) and
\(b \in \mathbb{A}^{N}\) is arbitrary, then
\begin{equation*}
\dim \mathrm{H}_{c}^{n}(U;f_{b}^{\ast}\mathcal{L}_{\psi}) \leq \dim \mathrm{H}^{n}_{c}(U;f_{a}^{\ast}\mathcal{L}_{\psi}).
\end{equation*}
\end{theorem}

\begin{proof}
For an arbitrary \(b \in \mathbb{A}^{N}\), we join \(b\) and
\(a \in \mathcal{U}_{\psi}\) by a smooth curve \(B\) in
\(\mathbb{A}^{N}\) (e.g., a straight line).  This gives rise to the
following commutative diagram
\begin{equation*}
\begin{tikzcd}
U \times B \ar[hook]{r}{\iota^{\prime}} \ar{d}{\mathrm{pr}_{2}^{\prime}} & U \times \check{\mathbb{A}}^{N} \ar{r}{f} \ar{d}{\mathrm{pr}_{2}} & \mathbb{A}^{1} \\
B \ar[hook]{r}{\iota} & \check{\mathbb{A}}^{N}
\end{tikzcd}
\end{equation*}
in which the left square is Cartesian.  By proper base change, we have
\begin{equation*}
R\mathrm{pr}^{\prime}_{2!}((f\circ \iota^{\prime})^{\ast}\mathcal{L}_{\psi}) = \iota^{\ast} R\mathrm{pr}_{2!}(f^{\ast}\mathcal{L}_{\psi}).
\end{equation*}
Since \(B \cap \mathcal{U} \neq \emptyset\), for any
\(a \in B \cap \mathcal{U}\), and any \(i\), the dimension of
\(\mathrm{H}_{c}^{i}(U;f_{a}^{\ast}\mathcal{L}_{\psi})\) is constant
in \(a\).  Since \(f^{\ast}\mathcal{L}_{\psi}[N+n]\) is a perverse
sheaf on \(U \times \check{\mathbb{A}}^{N}\) (cosupport condition follows from Lemma~\ref{lemma:cosupp}, the support condition is trivially satisfied), and since
\(\mathrm{pr}_{2}\) is an affine morphism,
Lemma~\ref{lemma:middle-estimate} implies the desired result.
\end{proof}

\begin{variant}\label{variant:np-exp}
Let notation and conventions be as in
Theorem~\ref{theorem:middle-upper-bound-exponential-sum}.  Suppose
\(U = U_{0}\otimes_{\mathbb{F}_{q}} k\) is a \(k\)-variety defined
over \(\mathbb{F}_{q}\), and there is a convex polygon \(\Gamma\), and
a Zariski open subset \(\mathcal{U}_{\Gamma} \subset \mathcal{U}\),
such that for all \(a \in \mathcal{U}_{\Gamma}\), the Newton polygon
of \(\mathrm{H}^{n}_{c}(U;f_{a}^{\ast}\mathcal{L}_{\psi})\) lies on or
above \(\Gamma\).  Then for any \(b \in \check{\mathbb{A}}^{N}\), the
Newton polygon of
\(\mathrm{H}^{n}_{c}(U;f_{b}^{\ast}\mathcal{L}_{\psi})\) lies on or
above \(\Gamma\).
\end{variant}

\begin{proof}
Apply Variant~\ref{variant:newton}.
\end{proof}

\subsection{A variant of the generic vanishing theorem of Katz--Laumon}
Theorem~\ref{theorem:middle-upper-bound-exponential-sum} shows that
\(\dim \mathrm{H}^{n}_{c}(U;f_{b}^{\ast}\mathcal{L}_{\psi})\) for
\(b\) \emph{generic} is an upper bound for
\(\dim\mathrm{H}^{n}_{c}(U;f^{\ast}_{b}\mathcal{L}_{\psi})\) for \(b\)
\emph{arbitrary}.  But how do we determine the former?
The following theorem shows that if this \(\check{\mathbb{A}}^{N}\) is a
``large family'', in the sense that \(u\) is a \emph{quasi-finite}
morphism, then we can use the Euler characteristic to calculate this
bound.

\begin{theorem}\label{theorem:nondegenerate-exponential-sums}
Let \(U\) be a purely \(n\)-dimensional local complete intersection
affine variety.  Suppose that \(u\colon U \to \mathbb{A}^{N}\) is
quasi-finite.  Then there is a Zariski open dense subset \(\mathcal{U}\) of
\(\check{\mathbb{A}}^{N}\) such that for any \(a \in \mathcal{U}\),
\begin{enumerate}
\item
\(\mathrm{H}^{i}_{c}(U;f_{a}^{\ast}\mathcal{L}_{\psi})=\mathrm{H}^{i}(U;f^{\ast}_{a}\mathcal{L}_{\psi})=0\)
for any \(i\neq n\),\label{item:katz-laumon}
\item
\(\dim \mathrm{H}^{n}_{c}(U;f^{\ast}_{a}\mathcal{L}_{\psi})\leq (-1)^{n}
(\chi(U) - \chi(f_{a}^{-1}(t)))\) for \(t \in \mathbb{A}^{1}\)
sufficiently general.
\end{enumerate}
\end{theorem}

\begin{remark*}
Theorem~\ref{theorem:nondegenerate-exponential-sums}(\ref{item:katz-laumon})
is due to Katz and Laumon
\cite[Théorème~5.5.1(ii)]{katz-laumon:fourier-transform-upper-bounds-exponential-sums}
when \(u\) is finite and \(U\) is smooth and irreducible.  If \(f\) is
tamely ramified at \(\infty\) (for example, when \((U,f)\) arises from a
\emph{sufficiently general} mod \(p\) reduction of a spread-out of a
complex variety), it can be shown that equality holds in the second
item.

For our purposes, even when focusing on exponential sums defined on a
nonsingular space, we must employ the Lefschetz theorem to address
non-middle cohomology degrees.  Consequently, we need to investigate
the Artin--Schreier cohomology on a singular \(U\).  Additionally, even for
toric exponential sums, the morphism \(u\) may not be finite.

We note that if \(U\) is smooth, irreducible and \(u\) is finite, then
Katz and Laumon also established the purity of the Artin--Schreier
cohomology.  This purity theorem does not hold in general when \(u\)
is only assumed to be quasi-finite.
\end{remark*}

To prove Theorem~\ref{theorem:nondegenerate-exponential-sums}, we
will first establish some general lemmas regarding the cohomology
spaces of the Artin--Schreier sheaves.  While various forms of these
lemmas are well-known to specialists, they have typically been proven
only in the tame ramification case, avoiding wild ramification.  Our
approach will handle both cases uniformly.

\begin{lemma}\label{lemma:estimate-euler}
Let \(U\) be a purely \(n\)-dimensional affine variety.
Let \(f\colon U \to \mathbb{A}^{1}\) be a regular function.  Assume
that the direct images \(R^{i}f_{\ast}\overline{\mathbb{Q}}_{\ell}\)
are tamely ramified for \(i\neq n-1\).  Then
\begin{equation*}
(-1)^{n}\chi(U;f^{\ast}\mathcal{L}_{\psi}) \leq
(-1)^{n}[\chi(U) - \chi(F)],
\end{equation*}
where \(F\) is a general fiber of \(f\).
\end{lemma}

If \(R^{n-1}f_{\ast}\overline{\mathbb{Q}}_{\ell}\) also exhibits tame
ramification at infinity, then the Grothendieck--Ogg--Shafarevich
theorem implies that equality holds.  The complexity arises from
the presence of wild ramification.  Fortunately, a general principle
applies here: the occurrence of wild ramification will only reduce the
dimension of the space of interest.

\begin{proof}
Since \(\mathcal{L}_{\psi}\) is a local system, the projection formula
also works for the direct image functor:
\begin{equation*}
Rf_{\ast}(f^{\ast}\mathcal{L}_{\psi}) = (Rf_{\ast}\overline{\mathbb{Q}}_{\ell})
\otimes \mathcal{L}_{\psi}
\end{equation*}
(cf.~\cite[Proposition~8.14, Remark (2)]{freitag-kiehl/etale-cohomology}).
This gives us a spectral sequence
\begin{equation*}
E_{2}^{a,b} =
\mathrm{H}^{a}(\mathbb{A}^{1};R^{b}f_{\ast}\overline{\mathbb{Q}}_{\ell}\otimes \mathcal{L}_{\psi})
\Rightarrow \mathrm{H}^{a+b}(U;f^{\ast}\mathcal{L}_{\psi}),
\end{equation*}
Taking Euler characteristics gives
\begin{equation*}
\chi(U;f^{\ast}\mathcal{L}_{\psi}) =  \sum_{i}(-1)^{n-i}\chi(\mathbb{A}^{1};R^{i}f_{\ast}\overline{\mathbb{Q}}_{\ell}).
\end{equation*}
We shall analyze the key factor
\(\chi(\mathbb{A}^{1};R^{n-1}f_{\ast}\overline{\mathbb{Q}}_{\ell})\)
below.  To avoid using overly cumbersome notation, let us introduce a
shorthand
\(\mathcal{E} = R^{n-1}f_{\ast}\overline{\mathbb{Q}}_{\ell}\).

By generic base change, the rank of the monodromy representation
\(\mathcal{E}_{(\infty)}\) of \(\mathcal{E}\) at infinity of
\(\mathbb{A}^{1}\) equals
\(\dim \mathrm{H}^{n-1}(F;\overline{\mathbb{Q}}_{\ell})\).  The
monodromy representation has a break decomposition
(cf.~\cite[\S\S1.1--1.10]{katz:gauss-sums-kloosterman-sums-monodromy})
\begin{equation*}
\mathcal{E}_{(\infty)} = \bigoplus_{s} E_{s},
\end{equation*}
where \(E_{s}\) is the summand of pure break \(s\).  By the
Grothendieck--Ogg--Shafarevich formula, using the fact that
\(\mathcal{L}_{\psi}\) is unramified at any point of
\(\mathbb{A}^{1}\), we have
\begin{equation}\label{eq:gos}
\chi(\mathbb{A}^{1} ; \mathcal{E} \otimes \mathcal{L}_{\psi})
= \chi(\mathbb{A}^{1};\mathcal{E})
+ \mathrm{Sw}(\mathcal{E}_{(\infty)}) - \mathrm{Sw} (\mathcal{E}_{(\infty)} \otimes \mathcal{L}_{\psi (\infty)}).
\end{equation}

Since \(\mathcal{L}_{\psi}\) is pure of break one at infinity,
for each \(s < 1\), \(E_{s} \otimes \mathcal{L}_{\psi(\infty)}\)
has Swan conductor equal to \(\operatorname{rank}(E_{s})\).
Similarly, for \(s > 1\), the Swan conductor of
\(E_{s} \otimes \mathcal{L}_{\psi(\infty)}\) equals the
Swan conductor of \(E_{s}\).  It might happen that there exists some
irreducible subrepresentation \(G\) of \(E_{1}\) where
\(G \otimes \mathcal{L}_{\psi(\infty)}\) has break \(<1\).  If
so, the Swan conductor of
\(E_{1}\otimes \mathcal{L}_{\psi(\infty)}\) will become smaller than
that of \(E_{1}\):
\begin{align*}
  \mathrm{Sw}(\mathcal{E}_{(\infty)}\otimes \mathcal{L}_{\psi(\infty)})
  -\mathrm{Sw}(\mathcal{E}_{(\infty)})
  &=
     {\textstyle\sum\limits_{s<1}(1-s)\dim E_{s}} - [\mathrm{Sw}(E_{1}) - \mathrm{Sw}(E_{1} \otimes \mathcal{L}_{\psi(\infty)})] \\
  &\leq  \dim {\textstyle\bigoplus\limits_{s<1}E_{s}}  \leq \dim \mathcal{E}_{(\infty)} = \dim \mathrm{H}^{n-1}(F;\overline{\mathbb{Q}}_{\ell}),
\end{align*}
where \(F\) is a generic fiber of \(f\).  It then follows from
\eqref{eq:gos} that
\begin{equation}
  \label{eq:estimate-middle-cohomology}
   -\chi(\mathbb{A}^{1};\mathcal{E}\otimes\mathcal{L}_{\psi})
  \leq \dim \mathrm{H}^{n-1}(F;\overline{\mathbb{Q}}_{\ell})  - \chi(\mathbb{A}^{1};\mathcal{E}).
\end{equation}
On the other hand, since
\(R^{i}f_{\ast}\overline{\mathbb{Q}}_{\ell}\) are assumed to be tamely
ramified at infinity for \(i\neq n-1\), the
Grothendieck--Ogg--Shafarevich formula implies that
\begin{equation}\label{eq:other-degrees}
\chi(\mathbb{A}^{1};R^{i}f_{\ast}\overline{\mathbb{Q}}_{\ell}\otimes \mathcal{L}_{\psi})
= \chi(\mathbb{A}^{1};R^{i}f_{\ast}\overline{\mathbb{Q}}_{\ell}) -
\dim \mathrm{H}^{i}(F;\overline{\mathbb{Q}}_{\ell}) \quad \text{ for any }i\neq n-1.
\end{equation}
Therefore
\begin{align*}
 (-1)^{n}[\chi(U) - \chi(F)]
  &= \sum_{i=0}^{n}(-1)^{n-i} [\chi(\mathbb{A}^{1};R^{i}f_{\ast}\overline{\mathbb{Q}}_{\ell})
    - \dim \mathrm{H}^{i}(F;\overline{\mathbb{Q}}_{\ell})] \\
  [\text{by }\eqref{eq:other-degrees}]
  &= -\chi(\mathbb{A}^{1};\mathcal{E}) + \dim \mathrm{H}^{n-1}(F;\overline{\mathbb{Q}}_{\ell})
    + \sum_{i\neq n-1} (-1)^{n-i}\chi(\mathbb{A}^{1}; R^{i}f_{\ast}\overline{\mathbb{Q}}_{\ell}\otimes \mathcal{L}_{\psi})  \\
  [\text{by }\eqref{eq:estimate-middle-cohomology}]
  &\geq -\chi(\mathbb{A}^{1};\mathcal{E}\otimes\mathcal{L}_{\psi})
    +  \sum_{i\neq n-1} (-1)^{n-i}\chi(\mathbb{A}^{1}; R^{i}f_{\ast}\overline{\mathbb{Q}}_{\ell}\otimes \mathcal{L}_{\psi}) \\
  &= (-1)^{n}\chi(U;f^{\ast}\mathcal{L}_{\psi}).
\end{align*}
This completes the proof of the lemma.
\end{proof}

The second lemma concerns the vanishing of Artin--Schreier cohomology.
We shall state it more generally using the language of the perverse
t-structure.  In our applications, \(U\) will be a purely
\(n\)-dimensional local complete intersection affine variety and
\(\mathcal{P}\) will be the shifted constant local system
\(\overline{\mathbb{Q}}_{\ell}[n]\).

\begin{lemma}\label{lemma:vanishing-non-middle}
Let \(U\) be a purely \(n\)-dimensional affine scheme.  Let
\(\mathcal{P}\) be a constructible complex on \(U\) satisfying the
support condition (i.e., \(\mathcal{P} \in {}^{\mathrm{p}}\mathcal{D}_{U}^{\leq0}\) for the selfdual
perversity).  Let \(f\colon U \to \mathbb{A}^{1}\) be a morphism such
that the ``weak Lefschetz'' property holds for general
\(t \in \mathbb{A}^{1}\).  This means that for all but finitely many
\(t \in \mathbb{A}^{1}\), the restriction map
\begin{equation*}
\mathrm{H}^{m}(U;\mathcal{P}) \to \mathrm{H}^{m}(f^{-1}(t);\mathcal{P})
\end{equation*}
is injective for \(m = -1\) and bijective for \(m < -1\).  Then for
all but finitely many \(\lambda \in k^{\times}\), we have
\begin{equation*}
\mathrm{H}^{m}(U;\mathcal{P} \otimes (\lambda f)^{\ast}\mathcal{L}_{\psi}) = 0
\end{equation*}
for any \(m \neq 0\).
\end{lemma}

Lemma~\ref{lemma:vanishing-non-middle} is similar to
\cite[Proposition~3.1]{denef-loeser:weights-exponential-sums-newton-polyhedra}
and
\cite[Théorème~5.5.1(ii)]{katz-laumon:fourier-transform-upper-bounds-exponential-sums}
when \(U\) is smooth and
\(\mathcal{P} = \overline{\mathbb{Q}}_{\ell}[n]\).  Since we do not
impose any smoothness condition on \(U\), we cannot expect the validity of
the ``purity'' part of these cited theorems.

\begin{proof}
Since \(\mathcal{P}\) satisfies the support condition, so does
\(\mathcal{P}\otimes f^{\ast}\mathcal{L}_{\psi}\).  By Artin's
vanishing theorem,
\(\mathrm{H}^{m}(U;\mathcal{P}\otimes f^{\ast}\mathcal{L}_{\psi})=0\)
for \(m > 0\).  We shall now assume that \(m < 0\).  Since \(U\) is
affine, \(f\) is an affine morphism.  Therefore, by Artin's vanishing
theorem, \(Rf_{\ast}\mathcal{P}\) satisfies the support condition,
i.e., it only has non-positive perverse cohomology sheaves:
\(\pR^{j}f_{\ast}\mathcal{P} = 0\) for \(j > 0\).  We
also have
\(Rf_{\ast}(\mathcal{P}\otimes f^{\ast}\mathcal{L}_{\psi})=(Rf_{\ast}\mathcal{P})\otimes\mathcal{L}_{\psi}\).
Because twisting by a global local system of rank one is a t-exact
operation, we have
\begin{equation*}
\pR^{j}f_{\ast}(\mathcal{P} \otimes f^{\ast}\mathcal{L}_{\psi}) = (\pR^{j}f_{\ast}\mathcal{P} )\otimes \mathcal{L}_{\psi}.
\end{equation*}

The Leray spectral sequence for
\(Rf_{\ast}(\mathcal{P}\otimes f^{\ast}\mathcal{L}_{\psi})\) reads:
\begin{equation*}
E_{2}^{i,j} = \mathrm{H}^{i}(\mathbb{A}^{1}; \mathcal{L}_{\psi}\otimes\pR^{j}f_{\ast}\mathcal{P})
\Rightarrow \mathrm{H}^{i+j}(U;\mathcal{P} \otimes f^{\ast}\mathcal{L}_{\psi}).
\end{equation*}

By Artin's vanishing theorem, for any perverse sheaf \(\mathcal{Q}\)
on \(\mathbb{A}^{1}\), we have
\(\mathrm{H}^{i}(\mathbb{A}^{1};\mathcal{Q})=0\) if \(i \neq 0,-1\).
Thus, only possible nonzero terms are \(E^{-1,j}\) and \(E^{0,j}\)
(\(j \leq 0\)).  The spectral sequence thereby degenerates into a
collection of short exact sequences
\begin{equation*}
0 \to \mathrm{H}^{0}(\mathbb{A}^{1};\mathcal{L}_{\psi}\otimes\pR^{m}f_{\ast}\mathcal{P})\to
\mathrm{H}^{m}(U;\mathcal{P}\otimes f^{\ast}\mathcal{L}_{\psi})
\to \mathrm{H}^{-1}(\mathbb{A}^{1};\mathcal{L}_{\psi}\otimes \pR^{m+1}f_{\ast}\mathcal{P})
\to 0.
\end{equation*}

Similarly, degeneration of the Leray spectral sequence of
\(Rf_{\ast}\mathcal{P}\) gives exact sequences:
\begin{equation}\label{eq:leray-without-twist}
0 \to \mathrm{H}^{0}(\mathbb{A}^{1};\pR^{m}f_{\ast}\mathcal{P})\to
\mathrm{H}^{m}(U;\mathcal{P})
\xrightarrow{\beta} \mathrm{H}^{-1}(\mathbb{A}^{1};\pR^{m+1}f_{\ast}\mathcal{P})
\to 0.
\end{equation}

The lemma will be proved if the following assertions are established:
\begin{enumerate}
\item \label{item:ai} For \(m \leq 0\),
\(\mathrm{H}^{-1}(\mathbb{A}^{1};\pR^{m}f_{\ast}\mathcal{P}\otimes\mathcal{L}_{\psi})=0\).
\item For \(m \leq -1\),
\(\mathrm{H}^{0}(\mathbb{A}^{1};\pR^{m}f_{\ast}\mathcal{P}\otimes\mathcal{L}_{\psi})=0\). \label{item:bi}
\end{enumerate}

We shall show that (\ref{item:ai}) is true for any \(m \leq -1\), and when
\(m=0\), (\ref{item:ai}) is true after scaling \(f\) by a generic
nonzero number \(\lambda\) in \(k\).  The assertion (\ref{item:bi}) is
true for any \(m\leq-1\).

To simplify the notation, let
\(\mathcal{F} = \pR^{m}f_{\ast}\mathcal{P}\).  Let
\(j\colon V \to \mathbb{A}^{1}\) be the ``ouvert de lissité'' of
\(\mathcal{F}\), and let \(a\colon \Sigma \to \mathbb{A}^{1}\) be the
inclusion of the complement of \(V\).  Thus
\(\mathcal{F}|_{V} = \mathcal{E}[1]\), and
\(\mathcal{E} = (R^{m-1}f_{\ast}\mathcal{P})|_{V}\) is a local system on \(V\).
Then we form the following distinguished triangle
\begin{equation}\label{eq:loc-inv-cycle}
\mathcal{F} \to (R^{0}j_{\ast}\mathcal{E})[1] \to \mathcal{F}_{\Sigma} \xrightarrow{+1},
\end{equation}
where \(\mathcal{F}_{\Sigma}\) is the cone of \(\mathcal{F} \to (R^{0}j_{\ast}\mathcal{E})[1]\).
Similarly, we have the following distinguished triangle:
\begin{equation}\label{eq:loc-inv-cycle-with-as}
\mathcal{F} \otimes \mathcal{L}_{\psi} \to (R^{0}j_{\ast}(\mathcal{E} \otimes \mathcal{L}_{\psi}))[1] \to \mathcal{F}_{\Sigma} \xrightarrow{+1}
\end{equation}
Note that the two triangles \eqref{eq:loc-inv-cycle} and
\eqref{eq:loc-inv-cycle-with-as} have the same cone.  This is
because for each \(b \in \Sigma\), \(\mathcal{L}_{\psi}\) is
unramified at \(b\).

\begin{claim}\label{claim:no-punctual-section}
Assume that \(m \leq 0\).  Then the map
\(\alpha\colon\mathrm{H}^{-1}(\mathbb{A}^{1};\mathcal{F})\to\mathrm{H}^{0}(V;\mathcal{E})\)
(induced by \eqref{eq:loc-inv-cycle}) is injective, and
\(\beta\colon\mathrm{H}^{m-1}(U;\overline{\mathbb{Q}}_{\ell})\to\mathrm{H}^{-1}(\mathbb{A}^{1};\mathcal{F})\)
is bijective.
\end{claim}

\begin{proof}[Proof of Claim~\ref{claim:no-punctual-section}]
The surjectivity of
\(\beta\) is seen from \eqref{eq:leray-without-twist}.  We have a commutative diagram
\begin{equation*}
\begin{tikzcd}
\mathrm{H}^{-1}(\mathbb{A}^{1};\mathcal{F}) \ar{r}{\alpha} & \mathrm{H}^{0}(V;\mathcal{E}|_{V}) \ar[d,equal]\\
\mathrm{H}^{m-1}(U;\mathcal{P}) \ar[hook]{r} \ar[two heads]{u}{\beta}& \mathrm{H}^{m-1}(f^{-1}(t);\mathcal{P})^{\pi_{1}(V)}
\end{tikzcd}.
\end{equation*}
The right two spaces are identified because \(V\) is the ouvert de
lissité of \(\mathcal{F}\).  The bottom arrow is injective thanks to
the ``weak Lefschetz'' hypothesis.  Since \(\beta\circ\alpha\) is
injective and \(\beta\) is surjective, we conclude that \(\alpha\) is
injective and \(\beta\) is bijective.
\end{proof}

We now turn to the proof of (\ref{item:ai}).  It is clear from the
presentation \eqref{eq:loc-inv-cycle} that \(\mathcal{F}_{\Sigma}\)
has possibly nonzero perverse cohomology only in degree \(-1\) and
\(0\).  The injectivity of \(\alpha\) proved in
Claim~\ref{claim:no-punctual-section} implies that
\(\mathrm{H}^{-1}(\mathbb{A}^{1};\mathcal{F}_{\Sigma})=0\).  From
\eqref{eq:loc-inv-cycle-with-as} we deduce that
\(\mathrm{H}^{-1}(\mathbb{A}^{1};\mathcal{F}\otimes\mathcal{L}_{\psi})\)
is a subspace of
\(\mathrm{H}^{0}(V;\mathcal{E}\otimes\mathcal{L}_{\psi})\).  It
suffices to prove, therefore, for \(m\leq 0\),
\(\mathrm{H}^{0}(V;\mathcal{E}\otimes\mathcal{L}_{\psi})=0\).  When
\(m \leq -1\), this is immediate: the weak Lefschetz hypothesis
implies that \(\mathcal{E}\) is isomorphic to the constant local
system \(\mathrm{H}^{m-1}(U;\mathcal{P})\), hence must be tamely
ramified; but then \(\mathcal{E}\otimes\mathcal{L}_{\psi}\) will have
wild ramification at \(\infty\), forcing
\(\mathcal{E}\otimes\mathcal{L}_{\psi}\) to have no nonzero sections.

When \(m = 0\), it might happen that \(\mathcal{E}\) has wild
ramification at \(\infty\).  The assertion (\ref{item:ai}) may not be
true for \(f\), so we might have to choose a different scaling factor
\(\lambda\).  Let us use \(\mathcal{E}_{(\infty)}\) to denote the
representation of the inertia group of \(\infty \in \mathbb{P}^{1}\)
induced by \(\mathcal{E}\).  Then \(\mathcal{E}_{(\infty)}\) admits a
break decomposition
\begin{equation*}
\mathcal{E}_{(\infty)} = \bigoplus_{s \geq 0} E_{s},
\end{equation*}
where \(E_{s}\) is a representation of the inertia which has only one break \(s\)
(cf.~\cite[1.8 below]{katz:gauss-sums-kloosterman-sums-monodromy}).  Let
\(L_{\mu}\) be any rank one nontrivial representation of the inertia
at infinity induced by the Artin--Schreier extension
\begin{equation*}
k(\!(t^{-1})\!)[y]/(y^{p} - y - \mu t)\quad (t \text{ being the coordinate of }\mathbb{A}^{1}).
\end{equation*}
Such a representation has pure break \(1\).  If
\(\mathrm{H}^{0}(V;\mathcal{E}\otimes\mathcal{L}_{\psi})\neq0\), then
in the break decomposition of \(\mathcal{E}_{(\infty)}\), \(E_{1}\)
must contain some copies of \(L_{1}\) as its subrepresentation.
Furthermore, \(E_{1}\) may contain some other \(L_{\mu}\) as its
subrepresentation.  But there are only finitely many such \(\mu\).

Now let us consider the composition
\begin{equation*}
X \xrightarrow{f} \mathbb{A}^{1} \xrightarrow{\lambda} \mathbb{A}^{1}.
\end{equation*}
where we also denote the multiplication-by-\(\lambda\) morphism by
\(\lambda\).  Let
\(\mathcal{F}_{\lambda} = {}^{\mathrm{p}}R^{0}(\lambda f)_{\ast}\mathcal{P} = \lambda_{\ast}\mathcal{F}\).
Since \(\lambda\) is an isomorphism, we have
\begin{equation*}
\lambda^{\ast}(\mathcal{F}_{\lambda} \otimes \mathcal{L}_{\psi}) =
\mathcal{F} \otimes \lambda^{\ast} \mathcal{L}_{\psi},
\end{equation*}
Now
\begin{equation*}
(\mathcal{F} \otimes \lambda^{\ast}\mathcal{L}_{\psi})_{(\infty)}
= \mathcal{E}_{(\infty)}[1] \otimes L_{\lambda}.
\end{equation*}
Thus, so long as \(\lambda\) is not equal to \(\mu^{-1}\), where
\(L_{\mu}\) is a subrepresentation of \(E_{1}\),
\(\mathcal{E}_{(\infty)} \otimes L_{\lambda}\) has no trivial
subrepresentation.  Since
\(\lambda\colon \mathbb{A}^{1}\to \mathbb{A}^{1}\) is an isomorphism,
we conclude that \(\mathcal{F}_{\lambda} \otimes \mathcal{L}_{\psi}\)
has trivial \(\mathrm{H}^{-1}\) over any open set on which it is
lisse.  Thus we may use \(\lambda f\) in place of \(f\), and the
assertion (\ref{item:ai}) for \(m=0\) is proved.

Let us now prove (\ref{item:bi}) for \(m \leq -1\).  In this
situation, the weak Lefschetz hypothesis implies that \(\mathcal{E}\)
is a constant local system associated to the vector space
\(\mathrm{H}^{m-1}(U;\mathcal{P})\), hence is tamely ramified at
\(\infty\).  Using the result of (\ref{item:ai}), we have
\(\mathrm{H}^{-1}(\mathbb{A}^{1};\mathcal{F}\otimes\mathcal{L}_{\psi})=0\).
Recall that the Grothendieck--Ogg--Shafarevich formula says that
\begin{equation*}
\dim \mathrm{H}^{0}(\mathbb{A}^{1};\mathcal{F}\otimes\mathcal{L}_{\psi}) - \dim \mathrm{H}^{-1}(\mathbb{A}^{1};\mathcal{F}\otimes\mathcal{L}_{\psi})
\end{equation*}
equals
\begin{equation*}
\dim \mathrm{H}^{0}(\mathbb{A}^{1};\mathcal{F}) - \dim \mathrm{H}^{-1}(\mathbb{A}^{1};\mathcal{F})
- \mathrm{Sw}(\mathcal{F}_{(\infty)} \otimes \mathcal{L}_{\psi(\infty)}).
\end{equation*}
Therefore, it suffices to prove the above number is zero.  By
Claim~\ref{claim:no-punctual-section}, the mapping \(\beta\) in
\eqref{eq:leray-without-twist} is bijective.  This shows that
\(\mathrm{H}^{0}(\mathbb{A}^{1};\mathcal{F})=0\).  Thus
\(-\chi(\mathcal{F})\) equals the dimension of
\(\mathrm{H}^{-1}(\mathbb{A}^{1};\mathcal{F})=\mathrm{H}^{m-1}(U;\mathcal{P})\).
Due to the perverse shift, the Swan conductor of
\(\mathcal{F} \otimes \mathcal{L}_{\psi}\) is equal to
\(-\mathrm{Sw}(\mathcal{E}_{(\infty)}\otimes\mathcal{L}_{\psi(\infty)})\).
Since \(\mathcal{E}\) is tamely ramified, this number is equal to the rank of \(\mathcal{E}\), which is
\(\dim \mathrm{H}^{m-1}(f^{-1}(t);\mathcal{P})\) for
some general \(t\).  The desired identity then follows from the weak
Lefschetz hypothesis.
\end{proof}

\begin{corollary}\label{corollary:vanishing-non-middle-compact-support}
Let \(U\) be a purely \(n\)-dimensional affine scheme.  Let
\(\mathcal{P}\) be a constructible complex on \(U\) satisfying the
cosupport condition (i.e.,
\(\mathcal{P} \in {}^{\mathrm{p}}\mathcal{D}_{U}^{\geq0}\) for the
self-dual perversity).  Let \(f\colon U \to \mathbb{A}^{1}\) be a
morphism such that the ``weak Lefschetz'' property holds for general
\(t \in \mathbb{A}^{1}\).  More precisely, for all but finitely many
\(t \in \mathbb{A}^{1}\), the corestriction map (\(\iota\) denotes the
closed embedding \(f^{-1}(t)\to U\))
\begin{equation*}
\mathrm{H}_{c}^{m}(f^{-1}(t);\iota^{!}\mathcal{P}) \to \mathrm{H}_{c}^{m}(U;\mathcal{P})
\end{equation*}
is surjective for \(m = 1\) and bijective for \(m > 1\).  Then for
all but finitely many \(\lambda \in k^{\times}\),
\begin{equation*}
\mathrm{H}_{c}^{m}(U;\mathcal{P} \otimes f^{\ast}\mathcal{L}_{\psi}) = 0
\end{equation*}
for any \(m \neq 0\).
\end{corollary}

\begin{proof}
This follows from Lemma~\ref{lemma:vanishing-non-middle} by applying
Verdier duality.  Since Verdier duality is not compatible with tensor
products, some care is needed, although it is not particularly difficult.
For completeness, let us spell out the details.  Let us first recall
that if \(M\), \(N\) and \(L\) are all \(R\)-modules, then there is a
natural homomorphism
\begin{equation*}
\mathrm{Hom}_{R}(M,N) \otimes \mathrm{Hom}_{R}(L,R) \to \mathrm{Hom}_{R}(M\otimes L,N).
\end{equation*}
It is determined by the assignment sending a pair of linear maps
\((T,\varphi) \in \mathrm{Hom}_{R}(M,N) \times \mathrm{Hom}(L,R)\) to
the bilinear map \((m,x)\mapsto \varphi(x)\cdot T(m)\).  This is
clearly bilinear in both \(T\) and \(\varphi\), thus giving the said
homomorphism.  When \(L = R\) is the trivial \(R\)-module, the map is
clearly an isomorphism.  Now for a local system \(\mathcal{L}\) on
\(U\), denote by \(\mathcal{L}^{\vee}\) its dual local system.  For
each constructible complex \(\mathcal{F}\), the above construction
with \(R=\overline{\mathbb{Q}}_{\ell}\) gives rise to a morphism of
constructible complexes of \(\overline{\mathbb{Q}}_{\ell}\)-vector
spaces
\begin{equation*}
\mathbb{D}\mathcal{F} \otimes \mathcal{L}^{\vee}
= R\Hom(\mathcal{F},\omega_{X}) \otimes R\Hom(\mathcal{L},\overline{\mathbb{Q}}_{\ell})
\to R\Hom(\mathcal{F}\otimes \mathcal{L}, \omega_{X})  = \mathbb{D}(\mathcal{F}\otimes \mathcal{L})
\end{equation*}
which is an isomorphism if
\(\mathcal{L} \simeq \overline{\mathbb{Q}}_{\ell}\) is a constant
local system.  Here, \(\omega_{X}\) is Verdier's dualizing complex,
and \(\mathbb{D}(-) = R\Hom(-,\omega_{X})\) is the Verdier duality functor.

Now take \(\mathcal{L} = f^{\ast}\mathcal{L}_{\psi}\).  Then the dual
local system \(\mathcal{L}^{\vee}\) of \(\mathcal{L}\) is
\(f^{\ast}\mathcal{L}_{\psi^{-1}}\).  Since
\(f^{\ast}\mathcal{L}_{\psi}\) is trivialized by the finite étale
covering of \(U\) defined by \(y^{p}-y = f(x)\), and checking whether
an arrow is an isomorphism is an étale-local matter, we conclude that
\begin{equation*}
\mathbb{D}(\mathcal{F} \otimes f^{\ast}\mathcal{L}_{\psi}) \simeq
\mathbb{D}\mathcal{F} \otimes f^{\ast}\mathcal{L}_{\psi^{-1}}.
\end{equation*}
This completes the proof.
\end{proof}

\begin{proof}[Proof of Theorem~\ref{theorem:nondegenerate-exponential-sums}]
Without loss of generality we may assume \(u_{1}(x)=1\).  Otherwise we
can replace \(u\) by \((1,u)\colon U \to \mathbb{A}^{N+1}\).  Since
\(u\) is quasi-finite, the map
\(x \mapsto (1,u(x)) \in \mathbb{A}^{N+1}\) is also quasi-finite.
Since \(U\) is a purely \(n\)-dimensional local complete intersection,
\(\overline{\mathbb{Q}}_{\ell}[n]\) is a perverse sheaf on \(U\)
(see~\cite[Lemma~III.6.5]{kiehl-weissauer:weil-conjecture-perverse-sheaf-fourier-transform}).
Applying Deligne's weak Lefschetz theorem
(Theorem~\ref{theorem:weak-lefschetz}) to
\(\overline{\mathbb{Q}}_{\ell}[n]\) and
\(\mathbb{D}(\overline{\mathbb{Q}}_{\ell}[n])\), we find that the
hypotheses of Lemma~\ref{lemma:vanishing-non-middle} and
Corollary~\ref{corollary:vanishing-non-middle-compact-support} are
satisfied for \(f_{a}^{-1}(t)\), where \(a\) is sufficiently general
in \(\check{\mathbb{A}}^{N}\).  This proves the vanishing of
\(\mathrm{H}^{i}(U;f^{\ast}\mathcal{L}_{\psi})\) and
\(\mathrm{H}^{i}_{c}(U;f^{\ast}\mathcal{L}_{\psi})\) when \(i\neq n\).
The second assertion then follows from these vanishing results and
Lemma~\ref{lemma:estimate-euler}.
\end{proof}

\subsection{Katz's situation}
\label{sec:katz-framework}
Let us offer an immediate application of the theorems above to a
framework considered by Katz~\cite{katz:exponential-sums}.

Consider the following data:
\begin{itemize}
\item Let \(X\) be a smooth proper variety over \(k\).
\item Let \(D_{1},\ldots, D_{s}\) be effective, nonsingular, irreducible divisors of \(X\),
such that for any subset \(J\) of \(\{1,\ldots,s\}\),
\begin{equation*}
D_{J} = \bigcap_{j\in J} D_{j}
\end{equation*}
is either empty, or smooth of codimension \(\operatorname{Card}J\) in
\(X\).
\item Let \(e_{1},\ldots,e_{s} \geq 1\) be natural numbers.
Let
\(D = \sum_{i=1}^{s}e_{i} D_{i}\).  Thus
\[\mathcal{O}_{X}(D)=\bigotimes_{i=1}^{s}\mathcal{O}_{X}(e_{i}D_{i}).\]
We assume that \(\mathcal{O}_{X}(D)\) is very ample.
\item Let \(H = \sum_{j=1}^{l}H_{j}\) be an auxiliary divisor on
\(X\), such that for each \(J \subset \{1,\ldots,s\}\),
\(K \subset \{1,\ldots,l\}\), \(H_{K} \cap D_{J}\) is either empty or
smooth of codimension \(\operatorname{Card}J+\operatorname{Card}K\) in
\(X\).  In other words, \(H \cup D_{\mathrm{red}}\) is a divisor with
strict normal crossings.
\item Let \(\mathcal{L}_{1},\ldots,\mathcal{L}_{r}\) be very ample
invertible sheaves on \(X\).
\item Let \(U = X \setminus (H \cup D_{\mathrm{red}})\).
\end{itemize}

From these data, we can construct a family of regular functions as
follows.  Let \(s_{1},\ldots,s_{N}\) be a basis of
\(\mathrm{H}^{0}(X;\mathcal{O}_{X}(D))\), with \(s_{1}\) being the
section defining the divisor \(D\).  Then on the variety \(U\) we have
a morphism
\begin{equation*}
u\colon U \to \mathbb{A}^{N}, x \mapsto \left( 1, \frac{s_{2}}{s_{1}},\ldots, \frac{s_{N}}{s_{1}} \right).
\end{equation*}
We are then reduced to the situation considered in
\S\ref{sec:bounds-exp-sums}.  Since \(D\) is very ample,
\(u\colon U \to \mathbb{A}^{N}\) is a locally closed immersion.

For any \(r\)-tuple
\((F_{1},\ldots,F_{r})\in\prod_{i=1}^{r}\mathrm{H}^{0}(X;\mathcal{L}_{i})\),
we define \(U(F_{1},\ldots,F_{r})\) to be the intersection of \(U\)
with the common zero locus of \(F_{1},\ldots,F_{r}\) in \(X\).  Set
\begin{equation*}
\mathcal{M} = \mathrm{H}^{0}(X;\mathcal{O}_{X}(D)) \times \prod_{i=1}^{r}\mathbb{P}\mathrm{H}^{0}(X;\mathcal{L}_{i}).
\end{equation*}
Then for any element \((s,F_{1},\ldots,F_{r}) \in \mathcal{M}\), we
get a regular function
\(f\colon U(F_{1},\ldots,F_{r}) \to \mathbb{A}^{1}\), \(x \mapsto s(x)/s_{1}(x)\).

\medskip
The following proposition (cf.~\cite[172]{katz:exponential-sums})
follows from a straightforward Chern class computation.

\begin{proposition}\label{proposition:chern}
Let \((s,F_{1},\ldots,F_{r})\) be a sufficiently general \((r+1)\)-tuple in
\(\mathcal{M}\).  Let \(Z = f^{-1}(0) \cap U(F_{1},\ldots,F_{r})\).
Then
\begin{align*}
   (-1)^{n-r}[\chi(U(F_{1},\ldots,F_{r})) -  \chi&(Z)]  \\
  = \, \int_{X} & \frac{c(X) c_{1}(\mathcal{L}_{1}) \cdots c_{1}(\mathcal{L}_{r})}{(1 + D)\cdot \prod\limits_{j=1}^{s}(1+D_{j})\prod\limits_{j=1}^{l}(1+H_{j})\prod\limits_{i=1}^{r}(1+ c_{1}\mathcal{L}_{i})}.
\end{align*}
\end{proposition}

\begin{theorem}\label{theorem:katz}
Let \((s^{\prime},F^{\prime}_{1},\ldots,F^{\prime}_{r})\) be a
sufficiently general point in \(\mathcal{M}\), defining a regular function
\[f^{\prime}\colon U(F^{\prime}_{1},\ldots,F^{\prime}_{r})\to \mathbb{A}^{1}.\]
Let \((s,F_{1},\ldots,F_{r}) \in \mathcal{M}\) be arbitrary, defining
\(f\colon U(F_{1},\ldots,F_{r}) \to \mathbb{A}^{1}\).
Then we have
\begin{equation*}
\dim \mathrm{H}^{n-r}_{c}(U(F_{1},\ldots,F_{r});f^{\ast}\mathcal{L}_{\psi}) \leq
\dim \mathrm{H}^{n-r}_{c}(U(F^{\prime}_{1},\ldots,F^{\prime}_{r});f^{\prime\ast}\mathcal{L}_{\psi}).
\end{equation*}
In particular, we have the following bound:
\begin{align*}
  \dim \mathrm{H}^{n-r}_{c}(U(F_{1},\ldots,F_{r}); f^{\ast}&\mathcal{L}_{\psi})  \\
  \leq
  \int_{X} &  \frac{c(X) c_{1}(\mathcal{L}_{1}) \cdots c_{1}(\mathcal{L}_{r})}{(1 + D)\cdot \prod\limits_{j=1}^{s}(1+D_{j})\prod\limits_{j=1}^{l}(1+H_{j})\prod\limits_{i=1}^{r}(1+ c_{1}\mathcal{L}_{i})}.
\end{align*}
\end{theorem}

\begin{proof}
The result follows from combining Theorem~\ref{theorem:middle-upper-bound-exponential-sum},
Theorem~\ref{theorem:nondegenerate-exponential-sums}, and
Proposition~\ref{proposition:chern}.
\end{proof}

To get an explicit bound, let us take \(X = \mathbb{P}^{n}\),
\(H = \emptyset\), \(D = d\mathbb{P}^{n-1}\) a multiple of the
infinity hyperplane.  Then \(U = \mathbb{A}^{n}\).

Let us first suppose that there are no invertible sheaves given
(i.e., \(r=0\)).  In this case we have the following:

\begin{corollary}
\label{theorem:bombieri-improvement}
Let \(d \geq 2\) be an integer.  Suppose the characteristic \(p\) of
\(k\) does not divide \(d\).  Let
\(f \colon \mathbb{A}^{n} \to \mathbb{A}^{1}\) be a regular function
of degree \(\leq d\).  Let
\(g_{i} \colon \mathbb{A}^{n-i} \to \mathbb{A}^{1}\) be any
sufficiently general polynomial of degree \(d\).  Then for
\(i \geq 0\), we have
\begin{equation*}
\dim \mathrm{H}^{n+i}_{c}(\mathbb{A}^{n};f^{\ast}\mathcal{L}_{\psi}) \leq
\dim \mathrm{H}_{c}^{n-i}(\mathbb{A}^{n-i};g_{i}^{\ast}\mathcal{L}_{\psi}(-i)).
\end{equation*}
Therefore, the total number of reciprocal roots and poles of
\(L_{f}(t)\) does not exceed
\begin{align*}
\sum_{m\geq0}\dim\mathrm{H}^{m}_{c}(\mathbb{A}^{n};f^{\ast}\mathcal{L}_{\psi})
&\leq (d-1)^{n} + (d-1)^{n-1} + \cdots + 1 \leq d^{n}.
\end{align*}
Moreover, if \((d,q)=1\), the Newton polygon of
\(\det(1-tF|_{\mathrm{H}^{n+i}_{c}(\mathbb{A}^{n};f^{\ast}\mathcal{L})})\)
lies above or on the Newton polygon of
\begin{equation*}
\prod_{m=0}^{(n-i)(d-2)}(1 - q^{i+\frac{m+n-i}{d}}t)^{U_{m,n-i}},
\end{equation*}
where \(U_{m,n}\) is the coefficient of \(x^{m}\) in the expansion of
\((1+x+\cdots +x^{d-2})^{n}\).
\end{corollary}

\begin{proof}
The integral in
Proposition~\ref{proposition:chern} can be computed exactly:
\begin{equation*}
(-1)^{n}[\chi(U) - \chi(f^{-1}(0))] = (d-1)^{n}.
\end{equation*}
Thus we conclude from Theorem~\ref{theorem:katz} that
\begin{equation*}
\dim \mathrm{H}^{n}_{c}(\mathbb{A}^{n};f^{\ast}\mathcal{L}_{\psi}) \leq (d-1)^{n}
\end{equation*}
for \emph{any} polynomial \(f\) of degree \(\leq d\).
Lemma~\ref{lemma:gysin-perv} tells us that we can restrict to a
generic affine hyperplane to bound the dimension of higher cohomology
spaces.  Thus
\begin{equation*}
\dim \mathrm{H}^{n+j}_{c}(\mathbb{A}^{n};f^{\ast}\mathcal{L}_{\psi})
\leq (d-1)^{n-j}.
\end{equation*}

Regarding the generic Newton polygon, we cite~\cite[Theorem
4.3]{adolphson-sperber:exponential-sums-on-an}.  Together with
Variant~\ref{variant:np-exp}, we see that the Newton polygon of
\(\det(1-tF|_{\mathrm{H}^{n}_{c}(\mathbb{A}^{n};f^{\ast}\mathcal{L}_{\psi})})\)
lies on or above the Newton polygon (``irregular Hodge polygon'') of
\begin{equation*}
\prod_{m=0}^{n(d-2)}(1 - q^{\frac{m+n}{d}}t)^{U_{m,n}},
\end{equation*}
where \(U_{m,n}\) is the coefficient of \(x^{m}\) in the expansion of
\((1+x+\cdots +x^{d-2})^{n}\).  Replacing \(n\) by \(n-j\) completes
the proof.
\end{proof}

When there are invertible sheaves present, say we are given
\(\mathcal{L}_{1}=\cdots=\mathcal{L}_{r}=\mathcal{O}_{\mathbb{P}^{n}}(d)\),
the exact generic formula, while attainable, is slightly complicated.  Therefore, we resort to an estimate.  In this case, we denote \(U(F_{1},\ldots,F_{r})\) by \(V\), and
we denote by \(f_{i}\) the dehomogenization of \(F_{i}\).  Then \(\deg f_{i} \leq d\).
Thus
\begin{equation*}
V = \{f_{1} = \cdots = f_{r} = 0\} \subset \mathbb{A}^{n}, \deg f_{i} \leq d.
\end{equation*}

\begin{corollary}\label{corollary:katz-to-an}
In the situation above, assume that \(\dim V = n-r \).
\begin{enumerate}
\item\label{item:general-complete-intersection-exp-sum} If the pair
$(V,f)$ is sufficiently general, then
\begin{equation*}
\sum_{m} \dim \mathrm{H}^{m}_{c}(V; f^{\ast}\mathcal{L}_{\psi}) = \dim \mathrm{H}^{n-r}_{c}(V; f^{\ast}\mathcal{L}_{\psi}) \leq \binom{n}{r}d^{n}.
\end{equation*}

\item\label{item:lowerbound}
In \textup{(\ref{item:general-complete-intersection-exp-sum})},
if \(d\) is coprime to the characteristic of \(k\), then
\begin{equation*}
\sum_{m} \dim \mathrm{H}^{m}_{c}(V; f^{\ast}\mathcal{L}_{\psi}) = \dim \mathrm{H}^{n-r}_{c}(V; f^{\ast}\mathcal{L}_{\psi}) \geq \binom{n}{r}(d-1)^{n}.
\end{equation*}
\item If \(V\) is an arbitrary complete intersection, then
\begin{equation*}
\sum_{m} \dim \mathrm{H}^{m}_{c}(V; f^{\ast}\mathcal{L}_{\psi}) \leq \binom{n}{r}(d+1)^{n}
\end{equation*}
\end{enumerate}
\end{corollary}

\begin{proof}
Let $C(n,r;d)$ denote the coefficient of $h^{n-r}$ in the power series
expansion of
\[
(-1)^{n-r}d^r \frac{(1+h)^{n}}{(1+dh)^{r+1}}.
\]
Expanding the
Chern class formula, Theorem~\ref{theorem:katz} gives the following
estimate:
\begin{equation}\label{eq:inequality}
\dim \mathrm{H}^{n-r}_{c}(V;f^{\ast}\mathcal{L}_{\psi})  \leq C(n,r;d).
\end{equation}
Moreover:
\begin{itemize}
\item If \((V,f)\) is sufficiently general, by
Theorem~\ref{theorem:nondegenerate-exponential-sums}, we have
\(\mathrm{H}^{i}_{c}(V;f^{\ast}\mathcal{L}_{\psi}) = 0\) for all $i \neq n-r$.

\item If \((V,f)\) is sufficiently general and if \(d\) is coprime to
the characteristic of \(k\), the inequality \eqref{eq:inequality} is
an equality by a theorem of Katz \cite{katz:exponential-sums}.

\item For an arbitrary pair \((V,f)\) satisfying the hypotheses,
applying Lemma~\ref{lemma:gysin-perv} and \eqref{eq:inequality}
repeatedly gives
\begin{equation*}
\sum_{j=0}^{n-r} \dim \mathrm{H}_{c}^{n-r+j}(V;f^{\ast}\mathcal{L}_{\psi})
\leq \sum_{j=0}^{n-r} C(n-j,r;d).
\end{equation*}
\end{itemize}

Therefore, all we need to do is to prove
\begin{equation*}
\binom{n}{r}(d-1)^{n} \leq C(n,r;d) \leq \binom{n}{r}d^{n}.
\end{equation*}
Replacing $h$ by $-h$, one sees that $C(n,r;d)$ is the coefficient of
$h^{n-r}$ in the power series expansion of
\begin{align*}
 d^r \frac{(1-h)^{n}}{{(1-dh)}^{r+1}}
 &= \frac{d^r}{d^{n}}\frac{((d-1)+(1-dh))^{n}}{{(1-dh)}^{r+1}} \\
 &= \frac{d^r}{d^{n}}\sum_{i=0}^{n}\binom{n}{i}(d-1)^{n-i}(1-dh)^{i-r-1}.
\end{align*}
Therefore, for the lower bound, we have
\begin{align*}
  C(n,r;d)
  &= \frac{d^r}{d^{n}}\sum_{i=0}^{n}\binom{n}{i} (d-1)^{n-i}(-d)^{n-r}
    \binom{i-r-1}{n-r}\\
  &=\sum_{i=0}^{r}\binom{n}{i}(d-1)^{n-i}\binom{n-i}{n-r} \\
  & \geq \binom{n}{r}  (d-1)^{n}. \nonumber
\end{align*}
For the upper bound, we have
\begin{align*}
  C(n,r;d)
  &\leq \binom{n}{r}\sum_{i=0}^{n}\binom{n}{i} (d-1)^{n-i} \\
  &= \binom{n}{r}d^n. \nonumber
\end{align*}
This completes the proof.
\end{proof}

For future applications, we record the following case of exponential sums over certain special complete intersection curves.

\begin{corollary}
Let \(q\) be a power of a prime number \(p\).  Let
$h(x), g_1(x), \cdots, g_n(x) \in \mathbb{F}_q[x]$ be nonzero
polynomials in one variable, each has degree $d$ not divisible by $p$.
Then, we have the estimate
\begin{equation*}
\left|
\sum_{\substack{(x_{1},\ldots,x_{n}) \in \mathbb{F}^{n}_{q} \\ g_1(x_1)=\cdots =g_n(x_n)}}
\psi\left({\rm Tr}_{\mathbb{F}_q/\mathbb{F}_p}h(x_1)\right)
\right| \leq n d^n \sqrt{q}.
\end{equation*}
\end{corollary}

\begin{proof}
Let us consider the complete intersection curve $V$ in
\(\mathbb{A}^{n}\) defined by the following $n-1$ equations
\begin{equation*}
g_1(x_1) - g_2(x_2)= g_1(x_1) -g_3(x_3) = \cdots
= g_1(x_1) - g_n(x_n)=0,
\end{equation*}
and the regular function $f: V \rightarrow \mathbb{A}^1$ is given by
$f(x_1, \cdots, x_n) = h(x_1)$.

The assumption $(d,p)=1$ implies that the map \(V \to \mathbb{A}^{1}\)
defined by \((x_{1},\ldots,x_{n}) \mapsto x_{1}\) is tamely ramified
at \(\infty\).  Indeed, \(V\) can be realized as a fiber product
\begin{equation*}
V \cong \mathbb{A}^{1} \times_{g_{1}, \mathbb{A}^{1}, g_{2}} \mathbb{A}^{1} \times_{g_{2},\mathbb{A}^{1},g_{3}} \cdots \times_{g_{n-1},\mathbb{A}^{1},g_{n}} \mathbb{A}^{1}.
\end{equation*}
Let \(K_{i} \coloneqq \mathbb{F}_{q}(\!(x_{i})\!)\) be the extension of
\(K\coloneqq\mathbb{F}_{q}(\!(t^{-1})\!)\) induced by
\(t \mapsto g_{i}(t)\). Then the étale \(K\)-algebra induced by the
morphism \(V \to \mathbb{A}^{1}\) is therefore
\(K_{1}\otimes_{K} K_{2}\cdots \otimes_{K} K_{n}\).  Since \(K_{i}\)
is tamely ramified over \(K\), the \(K_1\)-algebra
\(K_{1}\otimes_{K}K_{2}\otimes_{K}\cdots \otimes_{K}K_{n}\) is a
tamely ramified étale \(K_{1}\)-algebra.  This implies the
\(\overline{\mathbb{Q}}_{\ell}\)-sheaf
\begin{equation*}
\mathcal{F}\coloneqq R^{0}\mathrm{pr}_{1\ast}(\overline{\mathbb{Q}}_{\ell, V})\quad
(\text{where } \mathrm{pr}_{1}\colon V \to \mathbb{A}^{1}
\text{ is the projeciton to the first factor})
\end{equation*}
on \(\mathbb{A}^{1}\) is tamely ramified at \(\infty\).  Since
\(\deg h\) is coprime to \(p\), the local system
$h^*\mathcal{L}_{\psi}$ on \(\mathbb{A}^{1}\) is wildly ramified at
infinity.  Hence the \(\overline{\mathbb{Q}}_{\ell}\)-sheaf
\(h^{\ast}\mathcal{L}_{\psi} \otimes \mathcal{F}\) is wildly ramified
at \(\infty\).  In particular,
\begin{equation*}
\mathrm{H}^{2}_{c}(\mathbb{A}^{1};\mathcal{F}\otimes h^{\ast}\mathcal{L}_{\psi}) = 0.
\end{equation*}
Since \(\mathrm{pr}_{1}\) is a finite morphism,
\(R^{0}\mathrm{pr}_{1\ast}=R\mathrm{pr}_{1\ast}\).  The projection
formula then implies that
\(\mathrm{H}^{2}_{c}(V; f^{\ast}\mathcal{L}_{\psi}) = 0\).  Since
\(f^{\ast}\mathcal{L}_{\psi}\) is a local system on \(V\),
\(\mathrm{H}^{0}_{c}(V;f^{\ast}\mathcal{L}_{\psi})=0\) as well.
By Grothendieck's trace formula~\eqref{eq:trace-formula}
and Weil II~\cite[Théorème 1]{deligne:weil-2}, we get
\begin{align*}
\left|
\sum_{x \in V(\mathbb{F}_{q})}\psi(f(x))
\right|
&= \left| \mathrm{Tr}(F^{m}|\mathrm{H}^{1}_{c}(V;f^{\ast}\mathcal{L}_{\psi}))  \right|\\
&\leq \sqrt{q} \cdot \dim \mathrm{H}^{1}_{c}(V;f^{\ast}\mathcal{L}_{\psi}) \\
&\leq nd^{n} \sqrt{q}  .
\end{align*}
In the last step, we applied Corollary~\ref{corollary:katz-to-an}.
This completes the proof.
\end{proof}


\subsection{Toric exponential sums}
\label{sec:tori-exp-sums}
In this section we apply the results developed so far to exponential
sums on an algebraic torus \(\mathbb{G}^{n}_{\mathrm{m}}\) and its
subvarieties defined by a collection of Laurent polynomials.  If \(A\) is a ring and we are given a
Laurent polynomial
\(f \in A[x_{1},\ldots,x_{n},(x_{1}\cdots x_{n})^{-1}]\),
\(f = \sum a_{u} x^{u}\), the \emph{Newton polytope} \(\Delta_{0}(f)\) of \(f\)
is the convex hull, in \(\mathbb{R}^{n}\), of
\begin{equation*}
\operatorname{Supp}(f) = \{u \in \mathbb{Z}^{n} : a_{u} \neq 0\}.
\end{equation*}
Its \emph{Newton polytope at infinity} is
\begin{equation*}
\Delta_{\infty}(f) = \Delta_{0}(f- T)
\end{equation*}
where \(T\) is a dummy variable, and we view \(f-T\) as a Laurent
polynomial with coefficients in the ring \(A(T)\).  Thus,
\(\Delta_{\infty}\) is the convex hull of
\(\{0\} \cup \operatorname{Supp}(f)\) in \(\mathbb{R}^{n}\).  For each
lattice polytope \(\Delta\) in \(\mathbb{R}^{n}\) and any ring \(A\),
we define
\begin{equation*}
L(\Delta)_{A} = \{f \in A[x_{1},\ldots,x_{n},(x_{1}\cdots x_{n})^{-1}] :
\Delta_{0}(f) \subset \Delta\}.
\end{equation*}
We shall write \(L(\Delta)\) in place of \(L(\Delta)_{k}\) when \(A\)
is the ground field \(k\).

When \(k\) is a field of characteristic \(0\),
Khovanskii \cite[\S3, Theorem~2]{khovanskii:newton-polyhedra-and-genus-of-complete-intersections}
gives an exact formula for the Euler characteristic of a generic
complete intersection in \(\mathbb{G}_{\mathrm{m}}^{n}\) in terms of
the Minkowski mixed volume.

\begin{theorem}\label{theorem:khovanskii}
Suppose \(k\) is an algebraically closed field of characteristic
\(0\).  Let \(\Delta_{1},\ldots,\Delta_{r}\) be lattice polytopes in
\(\mathbb{Z}^{n}\).  Suppose \(h_{i} \in L(\Delta_{i})\),
\(i=1,\ldots,r\) are sufficiently general.  Let
\(V = \{h_{1}=\cdots=h_{r}=0\} \subset \mathbb{G}_{\mathrm{m}}^{n}\).
Then the morphism
\(\mathrm{H}^{i}(\mathbb{G}_{\mathrm{m}}^{n};\overline{\mathbb{Q}}_{\ell})\to\mathrm{H}^{i}(V;\overline{\mathbb{Q}}_{\ell})\)
is bijective for \(i < n-r\), and
\begin{equation*}
\chi(V) = \prod_{i=1}^{r}\frac{\Delta_{i}}{1+\Delta_{i}}.
\end{equation*}
\textup{(Recall Notation~\ref{notation:minkowski-mixed-volume}.)}
\end{theorem}

The original proof of Theorem~\ref{theorem:khovanskii} uses the
Bertini--Sard theorem, so it cannot be directly carried out in positive
characteristics.  Indeed, if \(k\) has positive characteristic, the
morphism \(\mathbb{G}_{\mathrm{m}}^{n} \to \mathbb{A}^{N}\),
\(x \mapsto (x^{w})_{w\in \Delta\cap \mathbb{Z}^{n}}\) may not be
generically smooth on the target, and all hypersurfaces defined by
\(\Delta\) may be singular.  Therefore, we should not expect
Theorem~\ref{theorem:khovanskii} to hold in positive
characteristics.  Nevertheless, by applying
Variant~\ref{variant:middle-estimate}, we can still obtain an upper bound
for the Euler characteristic of \(V\).

\begin{proposition}\label{proposition:khovanskii-positive-characteristic}
Let \(k\) be an algebraically closed field of characteristic \(p>0\).
Let \(\ell\) be a prime different from \(p\).  Let
\(\Delta_{1},\ldots,\Delta_{r}\) be \(n\)-dimensional lattice polytopes
in \(\mathbb{Z}^{n}\).  Suppose \(h_{i} \in L(\Delta_{i})\),
\(i=1,\ldots,r\) are sufficiently general.  Let
\(V = \{h_{1}=\cdots=h_{r}=0\} \subset \mathbb{G}_{\mathrm{m}}^{n}\).
Then
\begin{equation*}
(-1)^{n-r}\chi(V;\overline{\mathbb{Q}}_{\ell}) \leq (-1)^{n-r}\prod_{i=1}^{r}\frac{\Delta_{i}}{1+\Delta_{i}}.
\end{equation*}
\end{proposition}

\begin{proof}
The idea is simple: we lift a generic \(V\) to characteristic \(0\)
which satisfies Theorem~\ref{theorem:khovanskii}, and then we apply
Variant~\ref{variant:middle-estimate}.

As is well-known, the Euler characteristic of \(V\) is independent of
\(\ell\) so long as \(\ell\) is not equal to the characteristic of the
ground field.  To prove the lemma, therefore, it matters not which
\(\ell\) we choose, and it suffices to prove it for one particular
\(\ell\).  We will later explain which \(\ell\) we shall use.

Let \(\mathcal{O}\) be a characteristic \(0\) discrete valuation ring
with residue field \(k\).  Let \(K\) be the field of fractions of
\(\mathcal{O}\).  For each \(r\)-tuple
\((\widetilde{h}_{1},\ldots,\widetilde{h}_{r})\) of \(K\)-points of
\(\prod_{i}\mathbb{P}L(\Delta_{i})_{K}\), we can scale the
coefficients without changing their common zero locus in
\(\mathbb{G}_{\mathrm{m},K}^{n}\).  After scaling, we can
assume \(\widetilde{h}_{i}\) has coefficients in \(\mathcal{O}\).  In
this situation, we write \(\mathcal{V}\) to be the
\(\mathcal{O}\)-subscheme of
\(\mathbb{G}_{\mathrm{m},\mathcal{O}}^{n}\) defined by the vanishing
of \(\widetilde{h}_{i}\).

Since the map that sends a Laurent polynomial \(f\) to \(f\) modulo
the maximal ideal is surjective, for any preassigned Zariski open
dense subset \(\mathcal{U}\) of the \(k\)-scheme
\(\prod_{i=1}^{r}\mathbb{P}L(\Delta_{i})_{k}\), we can always find
\((\widetilde{h}_{1},\ldots,\widetilde{h}_{r}) \in
\prod_{i}\mathbb{P}L(\Delta_{i})_{\mathcal{O}}\) such that the special
fiber \(\mathcal{V} \otimes_{\mathcal{O}} k = V\) is defined by an
\(r\)-tuple \((h_{1},\ldots,h_{r})\) in \(\mathcal{U}\).

Applying Lemma~\ref{lemma:gysin-perv} to
\(\overline{\mathbb{Q}}_{\ell,\mathbb{G}_{\mathrm{m}}^{n}}[n]\), we can find
a Zariski open dense subset \(\mathcal{U}\) of
\(\prod_{i=1}^{r}\mathbb{P}L(\Delta_{i})_{k}\) such that if
\((h_{1},\ldots,h_{r}) \in \mathcal{U}\), then
\begin{equation}\label{eq:weak-lef-for-toric-complete-intersection}
\mathrm{H}_{c}^{n-r+i}(V;\overline{\mathbb{Q}}_{\ell})
\to \mathrm{H}^{n+r+i}_{c}(\mathbb{G}^{n}_{\mathrm{m}};\overline{\mathbb{Q}}_{\ell})
\end{equation}
is bijective for all \(i \geq 1\).  Thus we may choose lifts
\(\widetilde{h}_{1},\ldots,\widetilde{h}_{r}\) of
\(h_{1},\ldots,h_{r}\) such that on the one hand
\((\widetilde{h}_{1},\ldots,\widetilde{h}_{r})\) satisfies
Theorem~\ref{theorem:khovanskii} (when passed to the algebraic closure
of \(K\)), and on the other hand
\eqref{eq:weak-lef-for-toric-complete-intersection} holds for \(V\).

Fix an algebraic closure \(\overline{K}\) of \(K\).
It suffices to show
\begin{equation}\label{eq:lift-1}
\dim \mathrm{H}^{n-r}_{c}(V;\overline{\mathbb{Q}}_{\ell}) \leq \dim \mathrm{H}_{c}^{n-r}(\mathcal{V}_{\overline{\eta}};\overline{\mathbb{Q}}_{\ell}).
\end{equation}
where \(\mathcal{V}_{\overline{\eta}}\) is the geometric generic fiber
of \(\mathcal{V}\) valued in \(\overline{K}\).

Since \(\mathcal{V}_{\overline{\eta}}\) is of finite type over
\(\overline{K}\), there is an algebraically closed subfield
\(K^{\prime}\) of \(\overline{K}\) which is isomorphic to
\(\mathbb{C}\), and \(\mathcal{V}_{\overline{\eta}}\) is defined over
\(K^{\prime}\simeq \mathbb{C}\).  Using this isomorphism, the
\(\mathbb{Z}_{\ell}\)-cohomology spaces
\(\mathrm{H}_{c}^{n-r}(\mathcal{V}_{\overline{\eta}};\mathbb{Z}_{\ell})\)
and
\(\mathrm{H}^{n-r}_{c}(\mathcal{V}_{\mathbb{C}};\mathbb{Z}_{\ell})\)
are isomorphic to the Betti cohomology space
\(\mathrm{H}^{n-r}_{c}(\mathcal{V}^{\mathrm{an}}_{\mathbb{C}};\mathbb{Z}_{\ell})\).
We shall now choose \(\ell\) so that the \(\mathbb{Z}_{\ell}\)-Betti
cohomology of \(\mathcal{V}^{\mathrm{an}}\) has no \(\ell\)-power
torsion in any degree.  Hence
\begin{equation}\label{eq:lift-2}
\dim \mathrm{H}^{i}_{c}(\mathcal{V}_{\overline{\eta}};\overline{\mathbb{Q}}_{\ell}) = \dim\mathrm{H}^{i}_{c}(\mathcal{V}_{\overline{\eta}};\mathbb{F}_{\ell})
\end{equation}
holds for any \(i\).  By Variant~\ref{variant:middle-estimate}, we have
\begin{equation}\label{eq:lift-3}
\dim \mathrm{H}^{n-r}_{c}(V;\mathbb{F}_{\ell}) \leq \dim\mathrm{H}^{n-r}_{c}(\mathcal{V}_{\overline{\eta}};\mathbb{F}_{\ell}).
\end{equation}
Since the \(\mathbb{Z}_{\ell}\)-cohomology of the variety \(V\) may
have \(\ell\)-torsion,
\begin{equation}\label{eq:lift-4}
\dim \mathrm{H}^{n-r}_{c}(V;\overline{\mathbb{Q}}_{\ell})\leq \dim \mathrm{H}^{n-r}_{c}(V;\mathbb{F}_{\ell}).
\end{equation}
Combining \eqref{eq:lift-2}, \eqref{eq:lift-3}, and \eqref{eq:lift-4}
yields \eqref{eq:lift-1}, and completes the proof.
\end{proof}

\begin{theorem}\label{theorem:ultimate-as}
Let \(\Delta_{1},\ldots,\Delta_{r}\) and \(\Delta_{\infty}\) be
\(n\)-dimensional lattice polytopes in \(\mathbb{R}^{n}\).  Assume that
\(0 \in \Delta_{\infty}\).  Suppose we are given
\(h_{i} \in L(\Delta_{i})\) (\(i=1,\ldots,r\)) and
\(f \in L(\Delta_{\infty})\).  Let
\(V = \{h_{1}=\cdots=h_{r}=0\}\subset \mathbb{G}_{\mathrm{m}}^{n}\).
Then:
\begin{enumerate}
\item We have the following bound on cohomology:
\begin{equation*}
\dim \mathrm{H}^{n-r}_{c}(V; f^{\ast}\mathcal{L}_{\psi}) \leq (-1)^{n-r}\frac{1}{1+\Delta_{\infty}} \prod_{i=1}^{r}\frac{\Delta_{i}}{1+\Delta_{i}}.
\end{equation*}
\item If moreover \(\dim V = n-r\), then for any \(0\leq j \leq n-r\), we have:
\begin{equation*}
\dim \mathrm{H}^{n-r+j}_{c}(V; f^{\ast}\mathcal{L}_{\psi}) \leq (-1)^{n-r}\frac{1}{1+\Delta_{\infty}} \prod_{i=1}^{r}\frac{\Delta_{i}}{1+\Delta_{i}}\cdot (-1)^{j} \left( \frac{S}{1+S} \right)^{j}
\end{equation*}
where \(S\) can be any \(n\)-dimensional lattice polytope in
\(\mathbb{R}^{n}\).
\end{enumerate}
\end{theorem}

\begin{proof}
To prove the first assertion, it suffices to assume that \(V\) is
defined by a sufficiently general \(r\)-tuple \((h_{1},\ldots,h_{r})\)
of \(\prod_{i=1}^{r}\mathbb{P}L(\Delta_{i})\).  Indeed, if
\((h_{1}^{\prime},\ldots,h_{r}^{\prime})\) is a sufficiently general
\(r\)-tuple defining a complete intersection \(V^{\prime}\) in
\(\mathbb{G}_{\mathrm{m}}^{n}\), we can form a family
\begin{equation*}
\mathcal{V} = \{(x,t) \in \mathbb{A}^{1} : (1-t)h_{i} + th_{i}^{\prime} = 0, i=1,\ldots,r\}.
\end{equation*}
Then \(\mathcal{V}\) is defined by \(r\) equations in
\(\mathbb{G}_{\mathrm{m}}^{n} \times \mathbb{A}^{1}\), and
\((f^{\ast}\mathcal{L}_{\psi} \boxtimes \overline{\mathbb{Q}}_{\ell})[n+1-r]\) satisfies the cosupport
condition by Lemma~\ref{lemma:cosupp}.  By Lemma~\ref{lemma:middle-estimate}, we find that
\begin{equation*}
\dim \mathrm{H}^{n-r}_{c}(V;f^{\ast}\mathcal{L}_{\psi}) \leq \dim \mathrm{H}^{n-r}_{c}(V^{\prime};f^{\ast}\mathcal{L}_{\psi}).
\end{equation*}
Therefore, without loss of generality, we may assume that \(V\) is already a sufficiently
general complete intersection.

Consider the map
\(u\colon \mathbb{G}_{\mathrm{m}}^{n} \to \mathbb{A}^{N}\) sending
\(x\) to \((x^{w})_{w\in\Delta_{\infty}\cap\mathbb{Z}^{n}}\).  Since
\(\Delta_{\infty}\) is \(n\)-dimensional, \(u\) is a quasi-finite
morphism.  In particular, the composition
\(V \to \mathbb{G}_{\mathrm{m}}^{n}\to \mathbb{A}^{N}\) is also
quasi-finite.  In this case, the affine space attached to
\(L(\Delta_{\infty})\) is identified with \(\check{\mathbb{A}}^{N}\).
By Theorem~\ref{theorem:middle-upper-bound-exponential-sum}, there is a Zariski open
dense subset \(\mathcal{U}\) of \(\check{\mathbb{A}}^{N}\), such that
for any \(f^{\prime} \in \mathcal{U}\) and any
\(f \in \check{\mathbb{A}}^{N}\), we have
\begin{equation*}
\dim \mathrm{H}^{n-r}_{c}(V;f^{\ast}\mathcal{L}_{\psi})\leq \dim \mathrm{H}^{n-r}_{c}(V;f^{\prime\ast}\mathcal{L}_{\psi}),
\end{equation*}
and, by Theorem~\ref{theorem:nondegenerate-exponential-sums},
\begin{equation*}
\dim \mathrm{H}^{n-r}_{c}(V;f^{\prime\ast}\mathcal{L}_{\psi}) \leq (-1)^{n-r}\chi(V) + (-1)^{n-r-1}\chi(f^{\prime-1}(t) \cap V)
\end{equation*}
for \(t\) sufficiently general.  Applying
Proposition~\ref{proposition:khovanskii-positive-characteristic}, we
conclude that
\begin{equation*}
\dim \mathrm{H}^{n-r}_{c}(V;f^{\ast}\mathcal{L}_{\psi}) \leq
(-1)^{n-r} \prod_{i=1}^{r}\frac{\Delta_{i}}{1+\Delta_{i}} + (-1)^{n-r-1} \frac{\Delta_{\infty}}{1+\Delta_{\infty}}\prod_{i=1}^{r}\frac{\Delta_{i}}{1+\Delta_{i}}.
\end{equation*}
This proves the first assertion.

For the second assertion, we use generic elements in
\(\mathbb{P}L(S)\) to ``cut down'' \(V\) and apply
Lemma~\ref{lemma:gysin-perv}.  For \(j \geq 0\), we can choose \(j\)
sufficiently general functions \(l_1,\ldots,l_j\) whose Newton
polytope (with respect to 0) is \(S\).  This choice ensures the existence of a
surjective map
\begin{equation*}
\mathrm{H}^{n-r-j}_{c}(V \cap L;f^{\ast}\mathcal{L}_{\psi})(-j) \to \mathrm{H}^{n-r+j}_{c}(V;f^{\ast}\mathcal{L}_{\psi}),
\end{equation*}
where
\(L = \{l_1 = \cdots = l_j = 0\} \subset \mathbb{G}_{\mathrm{m}}^{n}\).
The second assertion then follows directly from the first one.
\end{proof}

\medskip Let us work out a more explicit bound when
\(V = \mathbb{G}_{\mathrm{m}}^{n}\).  In this case, we have
\begin{equation*}
\sum \dim \mathrm{H}^{n+j}_{c}(\mathbb{G}_{\mathrm{m}}^{n};f^{\ast}\mathcal{L}_{\psi})
\leq
\frac{(-1)^{n}}{1+\Delta_{\infty}}\sum_{j=0}^{n-r} (-1)^{j}\left( \frac{S}{1+S} \right)^{j}.
\end{equation*}
For $i\geq 1$, the coefficient before \(\Delta_{\infty}^{n-i}S^{i}\) is given by
\begin{equation*}
(-1)^{i}\sum_{j=0}^{i}(-1)^{j}\binom{-j}{i-j} = \sum_{j=1}^{i}\binom{i-1}{i-j} = 2^{i-1}.
\end{equation*}
Here we have applied the binomial coefficient identity
\begin{equation*}
(-1)^{a}\binom{z}{a} = \binom{-z+a-1}{a}.
\end{equation*}
Therefore, the final upper bound is
\begin{equation}\label{eq:total-degree-bound-toric-exponential-sum}
\sum \dim \mathrm{H}^{i}_{c}(\mathbb{G}_{\mathrm{m}}^{n};f^{\ast}\mathcal{L}_{\psi})
\leq \Delta_{\infty}^{n} + \sum_{i=1}^{n} 2^{i-1}\Delta_{\infty}^{n-i}S^{i}.
\end{equation}

We also have a variant of the above theorem for Newton polygons of
toric exponential sums.  Let us recall the Adolphson--Sperber lower
bound for nondegenerate Laurent polynomials.  Suppose \(\Delta\) is an
\(n\)-dimensional lattice polytope in \(\mathbb{R}^{n}\) containing
the origin.  For each \(u \in \mathbb{Z}^{n}\), suppose the half-line
from the origin through \(u\) intersects \(\Delta\) in a face that
does not contain the origin.  Let
\(\sum_{i=1}^{n}\epsilon_{i} x_{i}=1\) be the equation of a hyperplane
passing through this face.  Then the weight of \(u\) is defined as
\(w(u)=\sum \epsilon_{i}u_{i}\).  If the intersection of the open half-line
with \(\Delta\) is empty, we set \(w(u)=+\infty\). The weight of the origin is always zero.
Let \(D\) be the
least common denominator of all the \(\epsilon_{i}\)'s over all faces
of \(\Delta\).  Then clearly
\(w(\mathbb{Z}^{n})\subset\frac{1}{D}\mathbb{Z}_{\geq0}\cup\{+\infty\}\).

For the Newton polygon calculation, we define:
\begin{align*}
  W(m) &= \operatorname{Card}\left\{u \in \mathbb{Z}^{n}:w(u) = \frac{m}{D}\right\},\\
  W^{\prime}(m) &= \sum_{l\geq 0} (-1)^{l}\binom{n}{l}W(m-lD).
\end{align*}
The \emph{Hodge polygon} \(\mathrm{HP}(\Delta)\) of $\Delta$ is defined as the convex polygon with vertices at the points:
\begin{equation*}
\left( \sum_{i=0}^{m}W^{\prime}(m), \sum_{i=0}^{m} \frac{i}{D} W^{\prime}(m) \right),\quad
m=0,1,2,\ldots.
\end{equation*}
The theorem of Adolphson and Sperber states that if \(f \in L(\Delta)_{\mathbb{F}_{q}}\) is
\emph{nondegenerate} with respect to \(\Delta\) (see
\cite[p.~376]{adolphson-sperber:exponential-sums-newton-polyhedra} for the
definition), then the Newton polygon of
\(\det(1-tF|\mathrm{H}^{n}_{c}(\mathbb{G}_{\mathrm{m}}^{n};f^{\ast}\mathcal{L}_{\psi}))\)
lies on or above the Hodge polygon \(\mathrm{HP}(\Delta)\), and has the same endpoints.

\begin{variant}
Let
\(f \in \mathbb{F}_{q}[x_{1},\ldots,x_{n}, (x_{1}\cdots x_{n})^{-1}]\)
be a Laurent polynomial such that
\(\Delta_{\infty}(f) \subset \Delta\).  Assume that there exists a
nondegenerate Laurent polynomial whose Newton polytope equals
\(\Delta\).  Then the Newton polygon of
\(\det(1-tF|\mathrm{H}^{n}_{c}(\mathbb{G}_{\mathrm{m}}^{n};f^{\ast}\mathcal{L}_{\psi}))\)
lies on or above the Hodge polygon \(\mathrm{HP}(\Delta)\).
\end{variant}

Since we do not know what the Hodge lower bound should be for the exponential
sum of a general function on a general complete intersection in
\(\mathbb{G}_{\mathrm{m}}^{n}\), we cannot extract any
information about the Newton polygon of
\(\mathrm{H}^{i}_{c}(\mathbb{G}_{\mathrm{m}}^{n};f^{\ast}\mathcal{L}_{\psi})\)
when \(i > n\).


\subsection{Proofs of \ref*{theorem:order} and \ref*{theorem:last}(\ref*{item:bomb-restricted-to-sub})}
\label{sec:proof-theorem-order}

We begin with the proof of Theorem~\ref{theorem:order}.  The lower
bound is easy, since by \eqref{eq:chain-of-b} and
Corollary~\ref{corollary:middle-cleaner-bound}(\ref{item:complete-intersection-total})
we have
\[
(d-1)^{n} \leq B_{c}^{\mathrm{ci}}(n,1;d) = B_{c}(n,1;d) \leq B_{c}(n,r;d).
\]
So we only need to prove the upper bound.


By a standard argument using Deligne's constructibility theorem
(namely, extract the finitely many defining coefficients, spread out
over \(\operatorname{Spec}\mathbb{Z}\), find a geometric point lying
above a closed point---whose residue field is necessarily a finite
field---at which the direct image attains its generic value), it
suffices to prove Theorem~\ref{theorem:order} over the algebraic
closure of a finite field.


Next, we recall a well-known theorem that connects the standard
compactly supported Betti numbers with those of Artin–Schreier
cohomology.  This relationship will be used in our later arguments.

\begin{theorem}\label{theorem:dwork-coh}
Let \(\psi\colon\mathbb{F}_{p}\to\mathbb{C}^{\ast}\) be a nontrivial
additive character.  Consider a geometric vector bundle
\(\pi\colon E \to X\) of rank \(r\) and a section \(s\colon X \to E\).
Let \(i\colon V \to X\) be the vanishing scheme of \(s\), and let
\(\varpi\colon E^{\vee} \to X\) denote the dual bundle of \(E\).
Define the function
\[
g\colon E^{\vee} \xrightarrow{s\circ (\varpi \times \mathrm{Id})} E \times E^{\vee} \xrightarrow{\mathrm{can}} \mathbb{A}^{1}.
\]
Then we have a natural isomorphism:
\[
\varpi_{\ast}(\varpi^{\ast}\mathcal{F} \otimes g^{\ast}\mathcal{L}_{\psi}) \simeq i_{\ast}i^{!} \mathcal{F}.
\]
\end{theorem}

\begin{proof}
This theorem follows from the basic properties of the Deligne–Fourier
transform, as developed by Laumon \cite{laumon:fourier-transform}.  We
shall omit the proof here, as the reader can find an essentially identical
proof in
\cite{baldassarri-d-agnolo:dwork-cohomology-algebraic-d-modules},
where Baldassarri and D'Agnolo proved a similar theorem for
\(\mathcal{D}\)-modules.  (Note that the \(\ast\)-pullback used in
\cite{baldassarri-d-agnolo:dwork-cohomology-algebraic-d-modules} is a
shift of the \(!\)-pullback.)
\end{proof}

After taking \(\mathcal{F} = \overline{\mathbb{Q}}_{\ell}\),
applying the global section functor \(R\Gamma(X;-)\),
and applying Poincaré duality, we obtain the following consequence.
For simplicity, we omit the Tate twists in the statement below.

\begin{corollary}\label{corollary:cayley}
In the situation above, we have an isomorphism of
\(\overline{\mathbb{Q}}_{\ell}\) vector spaces
\begin{equation*}
\mathrm{H}^{j}_{c}(V;\overline{\mathbb{Q}}_{\ell}) \simeq \mathrm{H}^{j+2r}_{c}(E^{\vee};g^{\ast}\mathcal{L}_{\psi}).
\end{equation*}
\end{corollary}

The particular case we will consider is when \(X = \mathbb{A}^{n}\),
\(E = X \times \mathbb{A}^{r}\), with \(s\) defined by \(r\)
polynomials \(f_{1},\ldots,f_{r}\) of degree at most \(d\).  Thus,
\(V = \operatorname{Spec} k[x_{1},\ldots,x_{n}]/(f_{1},\ldots,f_{r})\)
represents the common zero locus of the \(f_{i}\).  The function \(g\)
is a regular function on \(\mathbb{A}^{n+r}\) in \(n+r\) variables,
given by:
\[
g(x_{1},\ldots,x_{n+r}) = x_{n+1}f_{1}(x) + \cdots + x_{n+r}f_{r}(x).
\]
For convenience, we introduce the following notation:
\[
B_{c}(X;f) = \sum \dim \mathrm{H}^{i}_{c}(X;f^{\ast}\mathcal{L}_{\psi}).
\]
Then by Corollary~\ref{corollary:cayley}, Theorem~\ref{theorem:order}
will follow from the following inequality:
\begin{equation}\label{eq:claim}
B_{c}(\mathbb{A}^{n+r};g) \leq 3^{r} \binom{n+r-1}{r-1}  (2d+1)^{n}.
\end{equation}

Before we proceed to the proof of \eqref{eq:claim}, we provide another
preliminary remark.  Let \(f\colon X \to \mathbb{A}^{1}\) be a regular
function on a variety \(X\).  Let \(U\) be a Zariski open subset of
\(X\), and let \(Z\) be the complement of \(U\).  Then we have a long exact
sequence:
\begin{equation}
\cdots \to \mathrm{H}^{i}_{c}(U;(f|_{U})^{\ast}\mathcal{L}_{\psi}) \to
\mathrm{H}^{i}_{c}(X;f^{\ast}\mathcal{L}_{\psi}) \to
\mathrm{H}^{i}_{c}(Z;(f|_{Z})^{\ast}\mathcal{L}_{\psi}) \to \cdots .
\end{equation}
From this sequence, we obtain the inequality:
\begin{equation}
\label{eq:accumulation-prime}
B_{c}(X;f)\leq B_{c}(U;f|_{U}) + B_{c}(Z;f|_{Z}).
\end{equation}

We will need the following lemma, which is derived from our earlier
work on exponential sums.  This lemma will establish the main term of
the estimate in \eqref{eq:claim}.  Additionally, it will be clear that
the error terms in \eqref{eq:claim} can also be obtained from this
lemma.

\begin{lemma}\label{lemma:torus-total}
Suppose \(f_{1},\ldots,f_{r} \in k[x_{1},\ldots,x_{n}]\) satisfy \(\deg f_{i} \leq d\).
Let \(g = \sum_{j=1}^{r} x_{n+j}f_{j}\).
Then
\[
B_{c}(\mathbb{G}_{\mathrm{m}}^{n+r};g) \leq 2^{n+r} \times \binom{n+r-1}{r-1} \times d^{n}.
\]
\end{lemma}

\begin{proof}
Let \(S_{d}\) be the convex hull of
\(0, de_{1},\ldots,de_{n}\) in the Euclidean space
$\prod_{i=1}^n \mathbb{R}e_i$ with the standard orthonormal basis $\{e_1,\ldots,e_n\}$.
Then Newton polytope
\(\Delta\) of \(g\) is contained in the convex hull of
\[
\{0\} \cup (S_{d}\times \{e_{n+1}\}) \cup \cdots \cup (S_{d} \times \{e_{n+r}\}).
\]
in the Euclidean space
\[
\mathbb{R}^{n+r} = \left( \prod_{i=1}^{n} \mathbb{R}e_{i} \right) \times
\left( \prod_{i=1}^{r} \mathbb{R}e_{n+i} \right).
\]
with the standard orthonormal basis $\{e_1,\ldots,e_n,e_{n+1},\ldots, e_{n+r}\}$.
We claim that
\begin{equation}
\label{eq:as-volume}
\Delta^{n+r} \leq \binom{n+r-1}{r-1} d^{n}.
\end{equation}
Once this claim is established, the lemma will follow from
\eqref{eq:more-like-as}.

Although the claim is elementary, let us work out the detail for the
sake of completeness.

Let \(T\) be the convex hull in \(\prod_{i=1}^{r}\mathbb{R}e_{n+i}\)
of the vectors \(e_{n+1},\ldots,e_{n+r}\).  Then \(T\) is an
\((r-1)\)-dimensional simplex.  Its volume is \(\frac{1}{(r-1)!}\)
times volume of the parallelotope spanned by the vectors
\(e_{n+1}-e_{n+r},\ldots, e_{n+r-1}-e_{n+r}\).
If $r=1$, this is the Euclidean volume of a single point, which is $1$. If $r > 1$, the volume is given by
\begin{equation*}
\mathrm{Vol}(T) = \frac{1}{(r-1)!}\left| \det B \right|^{\frac{1}{2}}
\end{equation*}
where \(B\) is the \((r-1)\times(r-1)\) Gram matrix, whose \((i,j)\)
entry is given by the inner product
\[
(e_{n+i} - e_{n+r})^{T} \cdot (e_{n+j} - e_{n+r}).
\]
It is clear that
\begin{equation*}
B =
\begin{bmatrix}
  2 & 1  & \cdots & 1 \\
  1 & 2  & \cdots & 1 \\
  \vdots  & \vdots & \ddots & \vdots \\
  1 & 1 &  \cdots & 2
\end{bmatrix},
\end{equation*}
and a simple computation shows that \(\det B = r\).
Hence, \(\mathrm{Vol}(T) = \frac{\sqrt{r}}{(r-1)!}\).

By definition, the polytope \(\Delta\) is contained in the pyramid
\(C\) having \(S_{d} \times T\) as its base, with
\[
\boldsymbol{0} = (\underbrace{0,\ldots,0}_{(n+r) \text{ 0's}}) \in \mathbb{R}^{n+r}
\]
serving as the vertex.  Since the vector
\begin{equation*}
v = \frac{1}{r}\sum_{i=1}^{r} e_{n+i} \in S_{d} \times T,
\end{equation*}
and since \(\|v\| = 1/\sqrt{r}\), the distance from \(\boldsymbol{0}\)
to the base \(S_{d} \times T\) is no larger than \(1/\sqrt{r}\).
Therefore, we have
\begin{align*}
  (n+r)! \operatorname{Vol}(\Delta)
  &\leq (n+r)!\operatorname{Vol}(C) \\
  &\leq (n+r)! \times \underbrace{\textstyle\frac{1}{n+r}}_{\text{scaling factor}}
    \times \operatorname{Vol}(\underbrace{S_{d}\times T}_{\text{base of }C})
    \times \underbrace{\textstyle\frac{1}{\sqrt{r}}}_{\text{height of }C} \\
  &= (n+r-1)! \times \frac{d^{n}}{n!} \times \frac{\sqrt{r}}{(r-1)!} \times \frac{1}{\sqrt{r}} \\
  &= \binom{n+r-1}{r-1} d^{n}.
\end{align*}
In the third step, we used the Fubini theorem.
This completes the proof of the claim.
\end{proof}

\begin{proof}[Proof of \eqref{eq:claim}]
For each \(I \subset \{1,2,\ldots,n+r\}\), define
\begin{equation*}
T_{I} = \{(x_{1},\ldots,x_{n+r}) \in \mathbb{A}^{n+r}:  x_{j} \neq 0 \text{ if }j \in I, x_{i} = 0 \text{ if }i\notin I\}
\end{equation*}
Then \(T_{I}\) is isomorphic to \(\mathbb{G}_{\mathrm{m}}^{|I|}\),
and we have
\begin{equation*}
\mathbb{A}^{n+r} = \bigsqcup_{I \subset \{1,2,\ldots,n+r\}} T_{I}.
\end{equation*}

By applying the relative long exact sequence
\eqref{eq:accumulation-prime} for compactly supported cohomology
several times, we find that
\begin{equation*}
B_{c}(\mathbb{A}^{n+r};g) \leq \sum_{I\subset \{1,2,\ldots,n+r\}} B_{c}(T_{I};g|_{T_{I}}).
\end{equation*}

For \(I \subset \{1,\ldots,n+r\}\), let
\(I^{\prime} = I \cap \{1,\ldots,n\}\) and
\(I^{\prime\prime}=I\cap \{n+1,\ldots,n+r\}\).
Thus
\begin{equation*}
g|_{T_{I}} = \sum_{j+r\in I^{\prime\prime}} x_{j+r}f_{j}(\boldsymbol{x}),
\quad \boldsymbol{x} \in \mathbb{G}_{\mathrm{m}}^{I^{\prime}}.
\end{equation*}
We can now apply Lemma~\ref{lemma:torus-total} to \(T_{I}\) and \(g\),
and conclude that
\begin{equation*}
B_{c}(T_{I};g) \leq 2^{|I|}\times \binom{|I|-1}{|I^{\prime\prime}|-1}\times d^{|I^{\prime}|}.
\end{equation*}
Hence, by \eqref{eq:accumulation-prime},
\begin{align*}
  B_{c}(\mathbb{A}^{n+r};g)
  &\leq  \sum_{I^{\prime}} \left[ \sum_{I^{\prime\prime}} 2^{|I|} \times \binom{|I|-1}{|I^{\prime\prime}|-1} \right] \times d^{|I^{\prime}|} \\
  &\leq \binom{n+r-1}{r-1} \sum_{I^{\prime}} (2d)^{|I^{\prime}|} \sum_{I^{\prime\prime}} 2^{|I^{\prime\prime}|} \\
  &= (2+1)^{r} \binom{n+r-1}{r-1} (2d+1)^{n}
\end{align*}
This completes the proof of \eqref{eq:claim}, and therefore proves
Theorem~\ref{theorem:order}.
\end{proof}

\begin{corollary}\label{corollary:proj-order}
Let \(X\) be a closed subvariety of \(\mathbb{P}^{n}\) that is the
common zero locus of \(r\) homogeneous polynomials of degree
\(\leq d\).  Then
\begin{equation*}
\sum_{i} \dim \mathrm{H}^{i}(X_{\overline{k}};\overline{\mathbb{Q}}_{\ell})
\leq 3^{r} \binom{n+r-1}{r-1}(2d+2)^{n}
\end{equation*}
where
\(\mathrm{H}^{i}(X_{\overline{k}};\overline{\mathbb{Q}}_{\ell})\) is
the \(\ell\)-adic cohomology of \(X\).
\end{corollary}

\begin{proof}
For a proper variety over \(k\), the compactly supported cohomology of
\(X\) is the same as the usual \(\ell\)-adic cohomology.  So we only
need to bound \(B_{c}(X, \ell)\).

Suppose now \(X \subset \mathbb{P}^{n}\) is defined by homogeneous
polynomials \(F_{1},\ldots,F_{r}\) of degree \(\leq d\).  Then
\(X = U \sqcup Z\), where \(U = X \cap \mathbb{A}^{n}\).  Here we
regard \(\mathbb{A}^{n}\) as a Zariski open subset of
\(\mathbb{P}^{n}\) in the standard way, and \(Z\) denotes the intersection
of \(X\) with the hyperplane at infinity.  The subset
\(U\subset \mathbb{A}^{n}\) is cut out by the dehomogenization of
\(F_{1},\ldots,F_{r}\), while \(Z\) is cut out in the hyperplane at
infinity by \(r\) homogeneous polynomials of degree \(\leq d\).

Using Theorem~\ref{theorem:order} to bound \(B_{c}(U, \ell)\), using
induction to bound \(B_{c}(Z, \ell)\), and applying the inequality
above, we obtain:
\begin{align*}
  B_{c}(X, \ell) &\leq B_{c}(U, \ell) + B_{c}(Z, \ell) \\
           &\leq 3^{r}\binom{n+r-1}{r-1}(2d+1)^{n} + 3^{r}\binom{n+r-2}{r-1}(2d+2)^{n-1} \\
           &\leq 3^{r}\binom{n+r-1}{r-1}(2d+2)^{n}.
\end{align*}
This completes the proof of the corollary.
\end{proof}

Although not directly related to our main results, we mention the
following consequence of \eqref{eq:as-volume} and the classical
theorem of Adolphson and Sperber
\cite{adolphson-sperber:newton-polyhedra-degree-l-function} on
estimating the Euler characteristic of an affine variety:

\begin{corollary}\label{corollary:euler-characteristic}
Let \(k\) be an algebraically closed field.
Let \(V \subset \mathbb{A}^{n}_{k}\) be the zero locus of
polynomials \(f_{1},\ldots, f_{r} \in k[x_{1},\ldots,x_{n}]\).
Then for any \(\ell\) invertible in \(k\), we have
\begin{equation*}
\left| \chi(V;\overline{\mathbb{Q}}_{\ell}) \right|
\leq 2^{r} \binom{n+r-1}{r-1}(d+1)^{n}.
\end{equation*}
\end{corollary}

\begin{proof}
Again, by the standard reduction, we can assume \(k\) is an algebraic
closure of a finite field \(\mathbb{F}_{q}\), and
\(f_{i} \in \mathbb{F}_{q}[x_{1},\ldots,x_{n}]\).  Write
\(g = \sum_{i=1}^{r} x_{n+i}f_{i}\), and
\(\mathbb{A}^{n+r} = \bigsqcup_{I\subset\{1,\ldots,n+r\}} T_{I}\) as
in the proof of \eqref{eq:claim}.  Then by
\cite[(1.9)]{adolphson-sperber:newton-polyhedra-degree-l-function},
and \eqref{eq:as-volume}, we have
\begin{equation*}
0 \leq (-1)^{|I|} \chi(T_{I};g|_{T_{I}}^{\ast}\mathcal{L}_{\psi})
\leq \binom{|I|-1}{|I^{\prime\prime}|-1}d^{|I^{\prime}|},
\end{equation*}
where $I$, $I^{\prime}$, and $I^{\prime\prime}$ are defined in the proof of
\eqref{eq:claim}.
Hence
\begin{align*}
  |\chi(V;\overline{\mathbb{Q}}_{\ell})|= |\chi(\mathbb{A}^{n+r};g^{\ast}\mathcal{L}_{\psi})|   &\leq \sum_{I} |\chi(T_{I};g|^{\ast}_{T_{I}}\mathcal{L}_{\psi})| \\
  &\leq \sum_{I} \binom{|I|-1}{|I^{\prime\prime}|-1}\times d^{|I^{\prime}|}\\
  & \leq \binom{n+r-1}{r-1} \times 2^{r}\times (d+1)^{n} .
\end{align*}
This completes the proof.
\end{proof}

\begin{proof}[Proof of Theorem~\ref{theorem:last}(\ref{item:bomb-restricted-to-sub})]
Suppose \(V = \{f_{1}= \cdots = f_{r}=0\}\).  In
Theorem~\ref{theorem:dwork-coh}, we take \(X = \mathbb{A}^{n}\),
\(E = \mathbb{A}^{n+r}\), and
\(\mathcal{F} = f^{\ast}\mathcal{L}_{\psi}\).  Then by a direct
computation,
\(g^{\ast}\mathcal{L}_{\psi}\otimes\varpi^{\ast}f^{\ast}\mathcal{L}_{\psi}\)
is isomorphic to \((f+g)^{\ast}\mathcal{L}_{\psi}\).  This gives us
\begin{equation*}
\mathrm{H}_{c}^{j}(V;f^{\ast}\mathcal{L}_{\psi}) \simeq
\mathrm{H}_{c}^{j+2r}(\mathbb{A}^{n+r};(f+g)^{\ast}\mathcal{L}_{\psi}),
\end{equation*}
where \(g = x_{n+1}f_{1}+\cdots + x_{n+r}f_{r}\).

To complete the proof, we need to bound \(B_{c}(\mathbb{A}^{n+r};f+g)\).
Let \(\Delta\) be the Newton polytope of \(f+g\). We claim that
\begin{equation*}
\Delta^{n+r} \leq \binom{n+r}{r} d^{n}.
\end{equation*}
The result follows from the above inequality and an argument similar
to the proof of \eqref{eq:claim}.

Again, the proof of the inequality is elementary.  Let \(S_{d}\) be
the \(n\)-dimensional simplex
with vertices \(0, de_{1},\ldots,de_{n}\)
in the Euclidean space $\prod_{i=1}^n \mathbb{R}e_i$ with the standard orthonormal basis $\{e_1,\ldots,e_n\}$.
Then the Newton polytope
\[
\Delta \subset \mathbb{R}^{n+r} =
\left( \prod_{i=1}^{n} \mathbb{R}e_{i} \right) \times
\left( \prod_{i=1}^{r} \mathbb{R}e_{n+i} \right)
\]
of \(f+g\) is contained in the convex hull of
\begin{equation*}
\big(S_{d} \times \{\underbrace{(0,\ldots,0)}_{r\text{ zeros}}\}\big)
\cup (S_{d} \times \{e_{n+1}\}) \cup \cdots \cup (S_{d} \times \{e_{n+r}\})
\end{equation*}
in the Euclidean space $\prod_{i=1}^{n+r} \mathbb{R}e_i$
with the standard orthonormal basis
$\{e_1,\ldots,e_n, e_{n+1}, \ldots, e_{n+r}\}$.
Let \(T^{\prime}\) be the \(r\)-dimensional simplex in
\(\prod_{i=1}^{r}\mathbb{R}e_{n+i}\) with vertices
\(0, e_{n+1},\ldots,e_{n+r}\).  Then we have
\begin{equation*}
\Delta \subset S_{d} \times T^{\prime}.
\end{equation*}
It follows from Fubini that
\begin{align*}
  (n+r)!\mathrm{Vol}(\Delta)
  &\leq (n+r)! \times \mathrm{Vol}(S_{d}) \times \mathrm{Vol}(T^{\prime}) \\
  &= (n+r)! \times \frac{d^{n}}{n!} \times \frac{1}{r!} \\
  &= \binom{n+r}{r} d^{n}.
\end{align*}
This completes the proof.
\end{proof}

\subsection{A more elementary approach to \ref*{corollary:total}}
\label{sec:elementary}

\begin{enumerate}[wide, label={(\alph*)}]
\item In this subsection, we give a more elementary proof of
Corollary~\ref{corollary:total}, which avoids using the total Betti
number estimate for exponential sums.  Instead, we combine the
reduction of \cite{wz2} with Katz's Euler characteristic method.  This
method needs a good upper bound for the absolute value of Euler
characteristics of varieties in \(\mathbb{A}^{n}_{k}\).  For this, we
use Corollary~\ref{corollary:euler-characteristic}. Let \(V\) be any
variety in \(\mathbb{A}^{n}_{k}\) defined by \(r\) polynomials of
degree \(\leq d\).  Then we have:
\begin{equation}\label{eq:euler-bound}
|\chi(V;\overline{\mathbb{Q}}_{\ell})| \leq
2^{r} \binom{n+r-1}{r-1}(d+1)^{n}
\end{equation}

Note that the proof of \eqref{eq:euler-bound} relies on just two
ingredients: a volume computation and a classical theorem of Adolphson
and Sperber.  In \cite{katz:sums-of-betti-numbers}, Katz used a weaker
Euler characteristic upper bound of \(2^{r} \times (rd + r + 1)^{n}\).
The improved bound \eqref{eq:euler-bound} leads to better estimates
throughout loc.~cit.

\item The first step of the elementary argument is to obtain an upper
bound for the total compactly supported Betti numbers of a complete
intersection in \(\mathbb{A}^{n}_{k}\).  For this, we have two
options: we could use Theorem~\ref{theorem:main}, which is logically
independent of Corollary~\ref{corollary:total}, or we could employ
Katz's more elementary Euler characteristic method
\cite{katz:sums-of-betti-numbers}.  Using the
bound~\eqref{eq:euler-bound} as input, Katz's method yields the
estimate
\begin{equation}\label{eq:katz-complete-intersection-bound}
B_{c}^{\mathrm{ci}}(n,r;d) \leq 2^{r}\binom{n+r-1}{r-1}(d+2)^{n}.
\end{equation}
(We omit proof of this inequality.  In Step \ref{step-c} a similar
argument will be given.)

While \eqref{eq:katz-complete-intersection-bound} is weaker than
Theorem~\ref{theorem:main}, using this bound instead of
Theorem~\ref{theorem:main} does not affect the final result, since the
difference between them is negligible compared to the losses
introduced in subsequent steps.

\item\label{step-c}
Now let \(W \subset \mathbb{A}^{n}_{k}\) be the common zero
locus of polynomials
\(f_{1},\ldots, f_{s} \in k[x_{1},\ldots,x_{n}]\), where we assume
\(\deg f_{i} \leq d\) and \(\dim W = n-s\).  Thus, \(W\) is
set-theoretically a complete intersection.  Let
\(g \in k[x_{1},\ldots,x_{n}]\) be any polynomial of degree
\(\leq e\).  We shall further assume \(e \geq d\), as this will be
the case in later steps.  Let \(V(g) = W \cap \{g=0\}\),
and \(W[g^{-1}] = W \setminus V(g)\).  Then a simple computation using
\eqref{eq:euler-bound} shows that
\begin{align}
  |\chi(W[g^{-1}];\overline{\mathbb{Q}}_{\ell})|
  &=|\chi(W;\overline{\mathbb{Q}}_{\ell}) - \chi(V(g);\overline{\mathbb{Q}}_{\ell})|
  \nonumber \\
  &\leq 3 \times 2^{s} \times \binom{n+s}{s}\times (e+1)^{n}. \label{eq:euler-difference-bound}
\end{align}

With the Euler characteristic bound of \(W[g^{-1}]\), we can now
recast Katz's method to get a total compactly supported Betti number
bound for \(W[g^{-1}]\) using Deligne's perverse Lefschetz theorem as
follows.  To begin with, since \(W\) is set-theoretically a complete
intersection, the shifted constant sheaf
\(\overline{\mathbb{Q}}_{\ell,W}[n-s]\) is a perverse sheaf on \(W\)
(the support condition is trivially satisfied, and the cosupport condition
follows from Lemma~\ref{lemma:cosupp}).  Since \(W[g^{-1}]\) is an
open subscheme of \(W\),
\(\overline{\mathbb{Q}}_{\ell,W[g^{-1}]}[n-s]\) is also perverse on
\(W[g^{-1}]\).  We can now apply Lemma~\ref{lemma:gysin-perv} with
\(X = W[g^{-1}]\) and \(f\) being the inclusion morphism, to conclude
that for a general hyperplane \(B\),
\begin{equation*}
\dim \mathrm{H}^{n-s+1}_{c}(W[g^{-1}];\overline{\mathbb{Q}}_{\ell})
\leq \dim \mathrm{H}^{n-s-1}_{c}((W\cap B)[g^{-1}];\overline{\mathbb{Q}}_{\ell}),
\end{equation*}
and, for \(j\geq 2\),
\begin{equation*}
\dim \mathrm{H}^{n-s+j}_{c}(W[g^{-1}];\overline{\mathbb{Q}}_{\ell})
= \dim \mathrm{H}^{n-s+j-2}_{c}((W\cap B)[g^{-1}];\overline{\mathbb{Q}}_{\ell}).
\end{equation*}
Combining these inequalities with \eqref{eq:euler-difference-bound},
we get
\begin{align}
  \dim \mathrm{H}^{n-s}_{c}(W[g^{-1}])
  &\leq (-1)^{n-s}\chi(W[g^{-1}];\overline{\mathbb{Q}}_{\ell}) + (-1)^{n-s+1}\chi((W \cap B)[g^{-1}];\overline{\mathbb{Q}}_{\ell}) \nonumber \\
  &\leq 3 \times 2^{s+1} \times \binom{n+s}{s}\times (e+1)^{n}.\label{eq:difference-middle}
\end{align}
Applying \eqref{eq:difference-middle} to \((W\cap B)[g^{-1}]\) yields
\begin{equation*}
\dim \mathrm{H}^{n-1-s}_{c}((W\cap B)[g^{-1}];\overline{\mathbb{Q}}_{\ell})
\leq 3 \times 2^{s+1} \times \binom{n-1+s}{s} \times (e+1)^{n-1}.
\end{equation*}
Continuing slicing \(W[g^{-1}]\) by hyperplanes and using the weak
Lefschetz theorem mentioned above, we conclude that
\begin{align}
  B_{c}(W[g^{-1}],\ell)
  &\leq 3 \times 2^{s+1}\times \sum_{i=0}^{n-s} \binom{n+s-i}{s} (e+1)^{n} \nonumber\\
  &< 3\times 2^{s+1} \times \binom{n+s}{s} \times (e+2)^{n}.
    \label{eq:1-extra-bound}
\end{align}

\item Now let us consider the general case.  Thus let \(V\) be the
common zero locus of \(r\) polynomials \(f_{1} = \cdots = f_{r} = 0\),
satisfying \(\deg f_{i} \leq d\).  Suppose \(\dim V = n-s\).  If
\(s = r\), then \(V\) is a set-theoretic complete intersection and we
have obtained estimates for its total Betti number.  In the following,
we shall assume \(r - s \geq 1\), namely \(s \leq r-1\).

By~\cite[Lemma~5.1]{wz1}, upon choosing a different set of \(r\)
defining polynomials (without changing \(V\)), we can assume
\(W = \{f_{1} = \cdots = f_{s}=0\}\) is a set-theoretic complete
intersection, i.e., \(\dim W = \dim V = n-s\).  Thus, the remaining
\(r-s\) polynomials \(f_{s+1},\ldots,f_{r}\) will not drop the
dimension of \(W\).  For each nonempty subset \(I\) of
\(\{s+1,\ldots,r\}\), let \(f_{I} = \prod_{i\in I} f_{i}\). Then there
is a Mayer--Vietoris spectral sequence
\begin{equation*}
E_{2}^{-p,q} = \bigoplus_{|I|=p+1} \mathrm{H}^{q}_{c}(W[f_{I}^{-1}];\overline{\mathbb{Q}}_{\ell}) \Rightarrow \mathrm{H}^{q-p}_{c}(W\setminus V;\overline{\mathbb{Q}}_{\ell}).
\end{equation*}
Using the spectral sequence, \eqref{eq:1-extra-bound} implies that
\begin{align}
  B_{c}(W\setminus V;\ell)
  &\leq \sum_{\substack{I\subset \{s+1,\ldots,r\} \\ I \neq \emptyset}} B_{c}(W[f_{I}^{-1}],\ell) \nonumber \\
  &\leq 3\times 2^{s+1}\times \binom{n+s}{s} \times \sum_{\substack{I\subset \{s+1,\ldots,r\} \\ I \neq \emptyset}}(|I|d+2)^{n}\nonumber \\
  &\leq 3\times 2^{s+1} \times \binom{n+s}{s} \times 2^{r-s} \times [(r-s)d+2]^{n} \nonumber\\
  &= 3\times 2^{r+1} \times \binom{n+s}{s} \times [(r-s)d+2]^{n}.\label{eq:complement}
\end{align}

It follows that
\begin{align*}
  B_{c}(V,\ell)
  & \leq B_{c}(W,\ell) + B_{c}(W \setminus V,\ell) \\
  \text{[by  \eqref{eq:katz-complete-intersection-bound} and
    \eqref{eq:complement}]}
  & \leq 2^{s}\binom{n+s-1}{s-1}(d+2)^{n} + 3 \cdot 2^{r+1} \binom{n+s}{s}((r-s)d+2)^{n}\\
  & < 7 \times 2^{r} \times \binom{n+s}{s} \times [(r-s)d+2]^{n}
\end{align*}

Therefore, we have arrived at Corollary~\ref{corollary:total} once again.
This time, the implied constant depends on the dimension \(n-s\) of
\(V\) in \(\mathbb{A}^{n}\).

\begin{theorem}
\label{theorem:elementary}
Let \(V\) be a subvariety of \(\mathbb{A}^{n}_{k}\) of codimension $s$.
Suppose \(V\) is the common zero locus of \(r\) polynomials of degree
\(\leq d\).
\begin{enumerate}
\item If \(r = s\), then we have \(B_{c}(V,\ell) \leq \binom{n-1}{r-1}(d+1)^{n}\).
\item If \(r \geq s + 1\),
then we have
\begin{equation*}
  B_{c}(V,\ell)
  \leq 7 \times 2^{r} \times \binom{n+s}{s} \times [(r-s) d + 2]^{n}
\end{equation*}
\end{enumerate}
\end{theorem}

Without prior knowledge of the relation between \(r\) and \(s\), the
inequality (ii) is generally weaker than Theorem~\ref{theorem:order}.
However, when \(r = s + 1\), (ii) can be better than
Theorem~\ref{theorem:order}.
\end{enumerate}

\bibliographystyle{amsalpha}
\bibliography{betti}%

\providecommand{\bysame}{\leavevmode\hbox to3em{\hrulefill}\thinspace}
\providecommand{\MR}{\relax\ifhmode\unskip\space\fi MR }
\providecommand{\MRhref}[2]{%
  \href{http://www.ams.org/mathscinet-getitem?mr=#1}{#2}
}
\providecommand{\href}[2]{#2}
\begin{thebibliography}{CRW22}

\bibitem[AS87]{adolphson-sperber:newton-polyhedra-degree-l-function}
Alan Adolphson and Steven Sperber, \emph{Newton polyhedra and the degree of the
  {L}-function associated to an exponential sum}, Invent. Math. \textbf{88}
  (1987), no.~3, 555–569. \MR{884800}

\bibitem[AS89]{adolphson-sperber:exponential-sums-newton-polyhedra}
\bysame, \emph{Exponential sums and {N}ewton polyhedra: cohomology and
  estimates}, Ann. of Math. (2) \textbf{130} (1989), no.~2, 367–406.
  \MR{1014928}

\bibitem[AS00]{adolphson-sperber:exponential-sums-on-an}
\bysame, \emph{Exponential sums on {$\mathbf{A}^n$}}, Israel J. Math.
  \textbf{120} (2000), 3--21. \MR{1815368}

\bibitem[BBD82]{beilinson-bernstein-deligne:perverse-sheaves}
Alexander~A. Beilinson, Joseph Bernstein, and Pierre Deligne, \emph{Faisceaux
  pervers}, Analysis and topology on singular spaces, {I} ({L}uminy, 1981),
  Astérisque, vol. 100, Soc. Math. France, Paris, 1982, p.~5–171.
  \MR{751966}

\bibitem[BD04]{baldassarri-d-agnolo:dwork-cohomology-algebraic-d-modules}
Francesco Baldassarri and Andrea D'Agnolo, \emph{On {D}work cohomology and
  algebraic {D}-modules}, Geometric aspects of {D}work theory, Walter de
  Gruyter, Berlin, 2004, p.~245–253. \MR{2023291}

\bibitem[Bom78]{bombieri:exponential-sums-in-finite-fields-2}
E.~Bombieri, \emph{On exponential sums in finite fields. {II}}, Invent. Math.
  \textbf{47} (1978), no.~1, 29--39. \MR{506272}

\bibitem[BS88]{BS88}
Enrico Bombieri and Steven Sperber, \emph{On the degree of {A}rtin
  {$L$}-functions in characteristic {$p$}}, C. R. Acad. Sci. Paris S\'er. I
  Math. \textbf{306} (1988), no.~9, 393--398. \MR{934603}

\bibitem[CRW22]{CMW22}
Qi~Cheng, J.~Maurice Rojas, and Daqing Wan, \emph{Computing zeta functions of
  large polynomial systems over finite fields}, J. Complexity \textbf{73}
  (2022), Paper No. 101681, 10. \MR{4474656}

\bibitem[Del77]{deligne:sga4.5}
Pierre Deligne, \emph{Cohomologie \'{e}tale}, Lecture Notes in Mathematics,
  vol. 569, Springer-Verlag, Berlin, 1977, S\'{e}minaire de g\'{e}om\'{e}trie
  alg\'{e}brique du Bois-Marie (SGA $4\frac{1}{2}$). \MR{463174}

\bibitem[Del80]{deligne:weil-2}
\bysame, \emph{La conjecture de {W}eil. {II}}, Inst. Hautes Études Sci. Publ.
  Math. (1980), no.~52, 137–252. \MR{601520 (83c:14017)}

\bibitem[DK73]{sga7-2}
Pierre Deligne and Nichlas~M. Katz, \emph{Groupes de monodromie en
  g\'{e}om\'{e}trie alg\'{e}brique. {II}}, Lecture Notes in Mathematics, vol.
  340, Springer-Verlag, Berlin-New York, 1973, S\'{e}minaire de
  G\'{e}om\'{e}trie Alg\'{e}brique du Bois-Marie 1967--1969 (SGA 7 II),
  Dirig\'{e} par P. Deligne et N. Katz. \MR{354657}

\bibitem[DL91]{denef-loeser:weights-exponential-sums-newton-polyhedra}
Jan Denef and François Loeser, \emph{Weights of exponential sums, intersection
  cohomology, and {N}ewton polyhedra}, Invent. Math. \textbf{106} (1991),
  no.~2, 275–294. \MR{1128216}

\bibitem[FK88]{freitag-kiehl/etale-cohomology}
Eberhard Freitag and Reinhardt Kiehl, \emph{\'{E}tale cohomology and the {W}eil
  conjecture}, Ergebnisse der Mathematik und ihrer Grenzgebiete (3) [Results in
  Mathematics and Related Areas (3)], vol.~13, Springer-Verlag, Berlin, 1988,
  Translated from the German by Betty S. Waterhouse and William C. Waterhouse,
  With an historical introduction by J. A. Dieudonn\'{e}. \MR{926276}

\bibitem[FW03]{FW03}
Lei Fu and Daqing Wan, \emph{Total degree bounds for {A}rtin {$L$}-functions
  and partial zeta functions}, Math. Res. Lett. \textbf{10} (2003), no.~1,
  33--40. \MR{1960121}

\bibitem[Har15]{Har15}
David Harvey, \emph{Computing zeta functions of arithmetic schemes}, Proc.
  Lond. Math. Soc. (3) \textbf{111} (2015), no.~6, 1379--1401. \MR{3447797}

\bibitem[HT25]{hu-teyssier}
Haoyu Hu and J.-B. Teyssier, \emph{Estimates for {B}etti numbers and relative
  {Hermite--Minkowski} theorem for perverse sheaves}, (to appear) (2025).

\bibitem[Ill03]{illusie:perversity-and-variation}
Luc Illusie, \emph{Perversit\'{e} et variation}, Manuscripta Math. \textbf{112}
  (2003), no.~3, 271--295. \MR{2067039}

\bibitem[Kat80]{katz:exponential-sums}
Nicholas~M. Katz, \emph{Sommes exponentielles}, Ast\'{e}risque, vol.~79,
  Soci\'{e}t\'{e} Math\'{e}matique de France, Paris, 1980, Course taught at the
  University of Paris, Orsay, Fall 1979, With a preface by Luc Illusie, Notes
  written by G\'{e}rard Laumon, With an English summary. \MR{617009}

\bibitem[Kat88]{katz:gauss-sums-kloosterman-sums-monodromy}
\bysame, \emph{Gauss sums, {K}loosterman sums, and monodromy groups}, Annals of
  Mathematics Studies, vol. 116, Princeton University Press, Princeton, NJ,
  1988. \MR{955052}

\bibitem[Kat93]{katz:affine-cohomological-transforms-perversity-monodromy}
\bysame, \emph{Affine cohomological transforms, perversity, and monodromy}, J.
  Amer. Math. Soc. \textbf{6} (1993), no.~1, 149--222. \MR{1161307}

\bibitem[Kat99]{katz:singular-sums}
\bysame, \emph{Estimates for ``singular'' exponential sums}, Internat. Math.
  Res. Notices (1999), no.~16, 875--899. \MR{1715519}

\bibitem[Kat01]{katz:sums-of-betti-numbers}
\bysame, \emph{Sums of {B}etti numbers in arbitrary characteristic}, Finite
  Fields Appl. \textbf{7} (2001), no.~1, 29--44, Dedicated to Professor Chao Ko
  on the occasion of his 90th birthday. \MR{1803934}

\bibitem[Kho78]{khovanskii:newton-polyhedra-and-genus-of-complete-intersections}
A.~G. Khovanski\u{i}, \emph{Newton polyhedra, and the genus of complete
  intersections}, Funktsional. Anal. i Prilozhen. \textbf{12} (1978), no.~1,
  51--61. \MR{487230}

\bibitem[KL85]{katz-laumon:fourier-transform-upper-bounds-exponential-sums}
Nicholas~M. Katz and G\'{e}rard Laumon, \emph{Transformation de {F}ourier et
  majoration de sommes exponentielles}, Inst. Hautes \'{E}tudes Sci. Publ.
  Math. (1985), no.~62, 361--418. \MR{823177}

\bibitem[KW01]{kiehl-weissauer:weil-conjecture-perverse-sheaf-fourier-transform}
Reinhardt Kiehl and Rainer Weissauer, \emph{Weil conjectures, perverse sheaves
  and {\(l\)}'adic {F}ourier transform}, Ergebnisse der Mathematik und ihrer
  Grenzgebiete. 3. Folge. A Series of Modern Surveys in Mathematics [Results in
  Mathematics and Related Areas. 3rd Series. A Series of Modern Surveys in
  Mathematics], vol.~42, Springer-Verlag, Berlin, 2001. \MR{1855066
  (2002k:14026)}

\bibitem[Lau87]{laumon:fourier-transform}
Gérard Laumon, \emph{Transformation de {F}ourier, constantes d'équations
  fonctionnelles et conjecture de {W}eil}, Inst. Hautes Études Sci. Publ.
  Math. (1987), no.~65, 131–210. \MR{908218}

\bibitem[LW54]{lang-weil:number-of-points-of-varieties-in-finite-fields}
Serge Lang and Andr\'{e} Weil, \emph{Number of points of varieties in finite
  fields}, Amer. J. Math. \textbf{76} (1954), 819--827. \MR{65218}

\bibitem[LW08]{AW08}
Alan G.~B. Lauder and Daqing Wan, \emph{Counting points on varieties over
  finite fields of small characteristic}, Algorithmic number theory: lattices,
  number fields, curves and cryptography, Math. Sci. Res. Inst. Publ., vol.~44,
  Cambridge Univ. Press, Cambridge, 2008, pp.~579--612. \MR{2467558}

\bibitem[MPT22]{maxim-paunescu-tibar:vanishing-cohomology-and-betti-bounds-for-projective-hypersurfaces}
Lauren\c{t}iu~G. Maxim, Lauren\c{t}iu P\u{a}unescu, and Mihai Tib\u{a}r,
  \emph{Vanishing cohomology and {B}etti bounds for complex projective
  hypersurfaces}, Ann. Inst. Fourier (Grenoble) \textbf{72} (2022), no.~4,
  1705--1731. \MR{4485837}

\bibitem[Per42]{perron:kronecker-theorem}
Oskar Perron, \emph{Beweis und {V}ersch\"arfung eines {S}atzes von
  {K}ronecker}, Math. Ann. \textbf{118} (1942), 441--448. \MR{10538}

\bibitem[Saw20]{sawin20}
Will Sawin, \emph{A representation theory approach to integral moments of
  {$L$}-functions over function fields}, Algebra Number Theory \textbf{14}
  (2020), no.~4, 867--906. \MR{4114059}

\bibitem[Saw21]{Saw21}
\bysame, \emph{Square-root cancellation for sums of factorization functions
  over short intervals in function fields}, Duke Math. J. \textbf{170} (2021),
  no.~5, 997--1026. \MR{4255048}

\bibitem[Ste85]{steenbrink:semicontinuity-of-spectrum}
Joseph H.~M. Steenbrink, \emph{Semicontinuity of the singularity spectrum},
  Invent. Math. \textbf{79} (1985), no.~3, 557--565. \MR{782235}

\bibitem[Wan93]{Wan93}
Daqing Wan, \emph{Newton polygons of zeta functions and {$L$} functions}, Ann.
  of Math. (2) \textbf{137} (1993), no.~2, 249--293. \MR{1207208}

\bibitem[Wan04]{Wan04}
\bysame, \emph{Variation of {$p$}-adic {N}ewton polygons for {$L$}-functions of
  exponential sums}, Asian J. Math. \textbf{8} (2004), no.~3, 427--471.
  \MR{2129244}

\bibitem[Wan08]{Wan08}
\bysame, \emph{Algorithmic theory of zeta functions over finite fields},
  Algorithmic number theory: lattices, number fields, curves and cryptography,
  Math. Sci. Res. Inst. Publ., vol.~44, Cambridge Univ. Press, Cambridge, 2008,
  pp.~551--578. \MR{2467557}

\bibitem[WZ22]{wz1}
Daqing Wan and Dingxin Zhang, \emph{Revisiting {Dwork} cohomology: Visibility
  and divisibility of frobenius eigenvalues in rigid cohomology},
  \href{https://arxiv.org/abs/2207.12633}{\texttt{arXiv:2207.12633}} (2022).

\bibitem[WZ23]{wz2}
\bysame, \emph{Hodge and {Frobenius} colevels of algebraic varieties},
  \href{https://arxiv.org/abs/2309.16290}{\texttt{arXiv:2309.16290}} (2023).

\bibitem[Zha25]{zhang-betti}
Dingxin Zhang, \emph{On sums of {B}etti numbers of affine varieties}, Finite
  Fields Appl. \textbf{103} (2025), no.~1, Paper No. 102583.

\end{thebibliography}

\end{document}